\documentclass[11pt,reqno]{amsart}
\usepackage{amssymb,mathrsfs,color}
\usepackage{enumerate}
\usepackage{graphicx}
\usepackage{xcolor} % A package to add color.
\usepackage{geometry}
\usepackage{amsfonts,amssymb,amsmath}
\usepackage{slashed}
\usepackage{ mathrsfs }
\usepackage{txfonts}
\usepackage[titletoc,title]{appendix}
\usepackage{wrapfig}
\usepackage{amsmath}

\usepackage{cancel}
\usepackage{cite}
\usepackage{bm}
\usepackage{nth}
\usepackage[T1]{fontenc}
\usepackage{hyperref}

\usepackage{marginnote}
\usepackage{cancel}
\usepackage{pgfplots}
\usepackage{upgreek}
\usepackage{slashed}

\definecolor{green}{rgb}{0,0.8,0} % Redefines the color green.

\definecolor{deepgreen}{cmyk}{1,0,1,0.5}

\newcommand{\Del}[1]{}

\numberwithin{equation}{section}

\newtheorem{theorem}{Theorem}[section]
\newtheorem{lemma}[theorem]{Lemma}%[section]
\newtheorem{proposition}[theorem]{Proposition}%[section]
\newtheorem{remark}[theorem]{Remark}%[section]
\newcommand{\abs}[1]{\vert#1\vert}

\newcommand{\barA}{{\overline A}}

\newcommand{\barE}{{\overline E}}

\newcommand{\barH}{{\overline H}}

\renewcommand{\hbar}{{\underline h}}

\newcommand{\tilI}{{\tilde{I}}}

\makeatletter
\newsavebox{\@brx}
\newcommand{\llangle}[1][]{\savebox{\@brx}{\(\m@th{#1\langle}\)}%
	\mathopen{\copy\@brx\kern-0.5\wd\@brx\usebox{\@brx}}}
\newcommand{\rrangle}[1][]{\savebox{\@brx}{\(\m@th{#1\rangle}\)}%
	\mathclose{\copy\@brx\kern-0.5\wd\@brx\usebox{\@brx}}}
\makeatother

\usepackage{nth}
\begin{document}
	\title{Well-posedness for the free boundary barotropic fluid model in general relativity}
	\author{ZeMing Hao 
		\and Wei Huo 
		\and Shuang Miao}
	\maketitle
\begin{abstract}
	In the framework of general relativity, the dynamics of a general barotropic fluid are coupled to the Einstein equations, which govern the structure of the underlying spacetime. We establish a priori estimates and well-posedness in Sobolev spaces for this model with a free boundary. Within the frame parallel-transported by the fluid velocity, we decompose the fluid and geometric quantities. The fluid components are estimated via a coupled interior-boundary wave equation, while the geometric quantities are analyzed through the Bianchi equations. Compared to a previous work \cite{Miao1}, the results in present paper allow general equations of state and non-zero vorticities. 
\end{abstract}    
\section{Introduction}
	In the context of general relativity, the dynamics of an ideal fluid are coupled to Einstein's equations
	\begin{equation}\label{einstein equ}
	R_{\mu\nu}-\frac{1}{2}g_{\mu\nu}S=T_{\mu\nu},
	\end{equation}
	where $R_{\mu\nu}$ is the Ricci curvature and $S=g^{\mu\nu}R_{\mu\nu}$ is the scalar curvature of the metric $g$. The energy-momentum tensor $T$ is given by 
	\begin{equation}\label{energy momentum tensor}
	   T_{\mu\nu}=\begin{cases}
	      (\rho + p )u_{\mu}u_{\nu}+pg_{\mu\nu} &\text{in}\ \Omega,\\
	      0  &\text{in}\ \Omega^{c},\\
	   \end{cases}
	\end{equation}
	where $\rho$ is the energy density, $p$ is the pressure, $u$ is a future-directed unit timelike 4-vector, and $\Omega$ denotes the fluid domain. The components of the 4-velocity $u^{\mu}$ satisfy   
	\begin{equation*}
	       g_{\alpha\beta}u^{\alpha}u^{\beta}=-1,\ u^{0}>0.
	\end{equation*}
	The motion of the fluid in $\Omega$ is governed by the conservation laws
	\begin{equation}\label{energycons equ}
	     \nabla_{\mu}T^{\mu\nu}=0,
	\end{equation} 
	and
	\begin{equation}\label{masscons equ}
	     \nabla_{\mu}I^{\mu}=0,
	\end{equation}
	where $\nabla$ is the covariant derivative associated with the metric $g$, and
	\begin{equation*}
	    I^{\mu}=nu^{\mu}
	\end{equation*}
	is the particle current. Here $n$ is the number density of particles. 
	
	Let $s$ denote the entropy per particle and $\theta$  the temperature. According to the laws of thermodynamics, the energy density $\rho$ and the pressure $p$ can be considered to be nonnegative functions of $n$ and $s$, satisfying the relations 
	\begin{equation*}
	    p=n\frac{\partial \rho}{\partial n}-\rho,\ \theta=\frac{1}{n}\frac{\partial \rho}{\partial s}.
	\end{equation*}
	The sound speed $\eta$ is defined by 
	\begin{equation*}
	    \eta^{2}=\left(\frac{\partial p}{\partial \rho}\right)_{s}.
	\end{equation*}
	It is assumed that $0< \eta\le c$, where the speed of light $c$ is normalized to 1. The fluid is barotropic, that is, the pressure is a function of the density
	\begin{equation}\label{state equ}
	      p=f(\rho),\quad f^{'}>0.
	\end{equation}
	In this case, the equations $(\ref{energycons equ})$ and $(\ref{masscons equ})$ are decoupled. Furthermore, $\rho$ and $p$ are functions of a single variable $\sigma$, given by
	\begin{equation*}
	    p+\rho=\sigma\frac{\mathrm{d}\rho}{\mathrm{d}\sigma}.
	\end{equation*}
	In addition to the fact that the function $f$ in $(\ref{state equ})$ is strictly increasing, we further assume that when the pressure $p=0$ the corresponding energy density $\rho_{0}>0$ and the integral
	\begin{equation*}
	     \int_{0}^{p}\frac{\mathrm{d}p}{\rho+p}=F(p)
	\end{equation*}
	exists. Then the conservation laws $(\ref{energycons equ})$ can be expressed in terms of the fluid velocity 
	\begin{equation*}
	    V=\Vert V \Vert u,
	\end{equation*}
	where $\Vert V \Vert:=e^{F}=\sqrt{-g(V,V)}$. 
	
	The component of $(\ref{energycons equ})$ along $u$ and the projection of $(\ref{energycons equ})$ onto the space orthogonal to $u$ reduce to the following equations in $\Omega$ (See \cite{Christodoulou}):
	\begin{equation}\label{orthu}
	\nabla_{\mu}\left(G\left(\Vert V\Vert\right)V^{\mu} \right)=0,
	\end{equation}
	\begin{equation}\label{alongu}
	     \nabla_{V}V^{\mu}+\frac{1}{2}\nabla^{\mu}\left(\Vert V\Vert^2\right)=0,
	\end{equation}	
	where 
	\begin{equation*}
	     G=\frac{\rho +p}{\Vert V\Vert^{2}}.
	\end{equation*}
    In the case where a barotropic fluid is additionally irrotational and the sound speed equals the speed of light, the resulting model is known as the hard phase model. This model represents an idealized physical scenario in which the mass-energy density exceeds the nuclear saturation density during gravitational collapse. Conversely, the equation of state remains soft when the mass-energy density is below the nuclear saturation density (see \cite{Christodoulou,Rezzolla}). The well-posedness of the free boundary problem of the hard phase model is proved in \cite{Miao1, Miao2}. See also \cite{Wang} for details on general barotropic fluids in the case of relativistic fluids in Minkowski spacetime.
    
\subsection{Model setup and satement of the main result}
	In this paper, we investigate the motion of a general barotropic fluid coupled to Einstein's eqautions with free boundary, surrounded by vaccuum. The fluid domain is, in principle, unknown and must be identified on the basis of the boundary conditions: 
	\begin{equation*}
	   \begin{aligned}
	      &\Vert V\Vert=1\  \text{on}\ \partial\Omega,\\
	      &V|_{\partial\Omega}\in \mathcal{T}\partial\Omega.
	   \end{aligned}
	\end{equation*}
	Here, $\mathcal{T}\partial\Omega$ denotes the tangent space of $\partial\Omega$. The first condition, which is equivalent to $p=0$ on $\partial\Omega$, is the analogue of no surface tension condition in the Newtonian setting. The second condition ensures that no fluid particle moves across $\partial\Omega$. Since $\Vert V\Vert^{2}$ is constant on the boundary, $\nabla\Vert V\Vert^{2}$ is normal to $\partial\Omega$. Thus, we can express this as
	\begin{equation*}
	\nabla\Vert V\Vert^{2}=-an,\quad \text{on}\ \partial\Omega,
	\end{equation*}  
	where 
	\begin{equation*}
	a:=\sqrt{\nabla_{\mu}\Vert V\Vert^{2}\nabla^{\mu}\Vert V\Vert^{2}}.
	\end{equation*}
	
	The free boundary problem for the general barotropic fluid equations in our consideration can be summarized as
	\begin{equation}\label{main equ}
	   \begin{cases}
	       \nabla_{V}V+\frac{1}{2}\nabla(\Vert V\Vert^{2})=0\ &\text{in}\ \Omega,\\
	       \nabla_{\mu}\left(G\left(\Vert V\Vert\right)V^{\mu} \right)=0\ &\text{in}\ \Omega,\\
	       \Vert V\Vert=1\  &\text{on}\ \partial\Omega,\\
	       V|_{\partial\Omega}\in \mathcal{T}\partial\Omega.
	   \end{cases}
	\end{equation}
	We then use $(\ref{einstein equ})$, which is coupled to the fluid, to derive governing equations for the evolution of the metric. Considering the diffeomorphism invariance of the equations, we will fix coordinates and choose a gauge to study the problem. Specifically, we adopt Lagrangian coordinates, $(t,x)=(x^{0},x^{1},x^{2},x^{3})$, such that $\frac{V}{\Vert V\Vert}=\partial_{t}$ in the fluid domain. The fluid velocity $V$ is extended to the exterior $\Omega^{c}$ using Sobolev extensions, ensuring that $V$ remains parallel to $\partial_{t}$. In particular, $V$ equals $\partial_{t}$ outside a compact set. Since the fluid domain is structured in a product form, this extension can be achieved through Sobolev extensions on each slice, ensuring that $\partial_{t}$ and the extension commute. 
	
	\begin{remark}
		In light of the fact that $V$ is tangential to the boundary of the fluid domain, the latter, when restricted to the interval $[0,T]$, can be defined in our coordinates as follows:
		\begin{equation}\label{fluid domain equ}
		\Omega_{0}^{T}=\left[0,T\right] \times \Omega_{0},
		\end{equation} 
		where, in general 
		\begin{equation*}
		\Omega_{t}=\Omega\cap\left\{x^{0}=t\right\},
		\end{equation*}
		and $\Omega_{0}\subseteq\left\{x^{0}=0\right\}$ represents the initial domain, which is assumed to be diffeomorphic to the unit ball. For the boundary, we use the notation $\partial\Omega_{0}^{T}=\left[0,T\right]\times\partial\Omega_{t}$. 
	\end{remark}
	
	In the following discussion, we will analyze the problem using an orthogonal frame that satisfies 
	\begin{equation*}
	    g(e_{I},e_{I})=\epsilon_{I},\quad \epsilon_{0}=-1,\ \epsilon_{\tilde{I}}=1, \ \tilI=1, 2, 3.
	\end{equation*}
	The frame elements $\left\{e_{I}\right\}$ are chosen such that $\nabla_{V}e_{I}=0$, that is, they are parallel transported by the fluid velocity. Additionally, We require that initially
	\begin{equation*}
	     e_{\tilde{I}}|_{t=0} \ \text{tangent to} \ \left\{t=0\right\}, \ \tilde{I}=1, 2, 3.
	\end{equation*}
	Thus, we have the relations
	\begin{equation}\label{connection of g}
	      m_{IJ}=g_{\mu\nu}e_{I}^{\mu}e_{J}^{\nu}, \quad g^{\mu\nu}=m^{IJ}e_{I}^{\mu}e_{J}^{\nu},
	\end{equation} 
	where $m_{IJ}$ denotes the components of the Minkowski metric in rectangular coordinates and $e_{I}=e_{I}^{\mu}\partial_{\mu}$. Let $\Gamma_{IJ}^{K}$ be the connection coefficients for this frame, such that 
	\begin{equation*}
	      \nabla_{e_{I}}e_{J}=\sum_{K} \Gamma_{IJ}^{K}e_{K}.
	\end{equation*}
	The fluid velocity $V$ can be expressed as 
	\begin{equation*}
	     V=\sum_{K}\Theta^{K}e_{K},
	\end{equation*}
	where $\Theta^{K}$ are components of $V$ with respect to the orthonormal frame $\left\{e_{I}\right\}_{I=0,1,2,3}$.  
	The renormalized fluid velocity is given by $\hat{V}=\frac{V}{\Vert V \Vert}=\partial_{t}$, where
	\begin{equation*}
	\hat{V}=\hat{\Theta}^{K}e_{K},\quad \hat{\Theta}^{K}=\frac{1}{\Vert V \Vert}\Theta^{K}. 
	\end{equation*}
	Note that the tangency condition have incroporated into the definition $(\ref{fluid domain equ})$ of $\Omega$. In terms of $\Theta^{I}$, equations $(\ref{main equ})$ become
	\begin{equation}\label{frame main equ}
	    \begin{cases}
	        \epsilon_{I}D_{V}\Theta^{I}+\frac{1}{2}D_{I}\Vert V \Vert^{2}=0\quad &\text{in} \ \Omega \\
	        D_{V}G+G\nabla_{I}\Theta^{I}=0\quad &\text{in} \ \Omega \\
	        \Theta_{I}\Theta^{I}=-1\quad &\text{on} \ \partial\Omega
	    \end{cases}.
	\end{equation}
	We also use the notation 
	\begin{equation*}
	    \nabla_{I}\Theta^{J}=D_{I}\Theta^{J}+\Gamma_{IK}^{J}\Theta^{K}.
	\end{equation*}
    Here and in what follows, $D_{I}$ stands for the derivative in the direction of $e_{I}$ applied to a scalar function, 
    \begin{equation*}
        D_{I}f=e_{I}(f):=e_{I}^{\mu}\partial_{\mu}f.
    \end{equation*}
    We write the components of the curvature tensor in the frame as
	\begin{equation*}
	     R_{IJKL}=R(e_{I},e_{J},e_{K},e_{L}).
	\end{equation*} 
   The Ricci tensor and scalar curvature are then expressed as
	\begin{equation*}
	    R_{IJ}=\sum_{K}\epsilon_{K}R_{IKJK},
	\end{equation*}
	and
	\begin{equation*}
	    S=\sum_{I}\epsilon_{I}R_{II}.
	\end{equation*}
	Given the discontinuity of the energy-momentum tensor across the fluid boundary, the curvature is an $L^{2}$ function. However, under appropriate initial regularity assumptions, the curvature becomes more regular within both $\Omega$ and $\Omega^{c}$. Finally, expressing the Einstein equations $(\ref{einstein equ})$ in the chosen frame yields
	\begin{equation}\label{Ricci equ}
	R_{IJ}=\begin{cases}
	G\Theta_{I}\Theta_{J}-pm_{IJ}+\frac{1}{2}G\Vert V\Vert^{2}m_{IJ} \quad &\text{in} \ \Omega\\
	0 \quad &\text{in} \ \Omega^{c}
	\end{cases}.
	\end{equation}
	The condition $\nabla_{V}e_{I}=0$ gives transport equations for $e_{I}$ in terms of $\Gamma_{IJ}^{K}$ and $\Theta^{I}$, and for $\Gamma_{IJ}^{K}$ in terms of $R_{IJKL}$ and $\Theta^{I}$. Indeed, through cumbersome calculations, $e_{I}$ and $\Gamma_{IJ}^{K}$ satisfy
	\begin{equation}\label{transport e}
        \begin{aligned}
           &\partial_{t}e_{\tilde{I}}=-(D_{\tilde{I}}\hat{\Theta}^{J})e_{J}-\hat{\Theta}^{J}\Gamma_{\tilde{I}J}^{K}e_{K},\quad \tilde{I}=1,2,3,\\
           &e_{0}=\frac{1}{\hat{\Theta}^{0}}(\hat{V}-\hat{\Theta}^{\tilde{I}}e_{\tilde{I}}),
        \end{aligned}
	\end{equation}
	and
	\begin{equation}\label{transport Ch}
	   \begin{aligned}
	       &\partial_{t}\Gamma_{\tilde{I}J}^{K}=\hat{\Theta}^{I}(R^{K}_{\ JI\tilde{I}}-\Gamma_{LJ}^{K}\Gamma_{\tilde{I}I}^{L})-\Gamma_{IJ}^{K}D_{\tilde{I}}\hat{\Theta}^{I},\quad \tilde{I}=1,2,3,\\
	       &\Gamma_{0J}^{K}=-\frac{\hat{\Theta}^{\tilde{I}}}{\hat{\Theta}^{0}}\Gamma_{\tilde{I}J}^{K}.
	   \end{aligned}
	\end{equation}
	
	We assume that the initial fluid domain $\Omega_{0}$ and the initial data $\Theta_{0}^{I}=\Theta^{I}(0)$ are provided, satisfying the following conditions:
	\begin{equation}\label{compatibility conditions}
	\begin{cases}
	\Vert V\Vert_{0}^{2}:=-\sum_{I}\epsilon_{I}(\Theta_{0}^{I})^{2}\ge 1, \Theta_{0}^{0}>0,\quad &\text{in} \quad \Omega_{0}\\
	\Vert V\Vert_{0}^{2}=1,\quad &\text{on} \quad \partial\Omega_{0}\\
	\nabla_{I}\Vert V\Vert_{0}^{2}\nabla^{I}\Vert V\Vert_{0}^{2}\ge c_{0}^{2}>0,\quad &\text{on} \quad \partial\Omega_{0}
	\end{cases}.
	\end{equation}
   The last term of the above system is the relativistic analogue of the Taylor sign condition, which will be used crucially in the analysis, and whose failure is known to lead to instabilities, cf. \cite{Geoffrey}. It is assumed that, at $t=0$, the following regularity and compatibility conditions are satisfied
    \begin{equation}\label{regularity conditions}
       \begin{aligned}
          \partial^{p}\partial_{t}^{k}\Theta^{I}, \partial^{p}\partial_{t}^{k+1}\Vert V\Vert^{2}\in L^{2}(\Omega_{t}), \quad& k\le K+1,2p+k\le K+2,\\
          \partial_{t}^{k}\Vert V\Vert^{2}\in H_{0}^{1}(\Omega_{t}), \quad& k\le K+1\\
          \partial_{t}^{K}R_{IJML},\  \partial\partial_{t}^{K-1}R_{IJML},\  \partial^{p}\partial_{t}^{k}R_{IJKL}\in L^{2}(\Omega_{t})\cap L^{2}(\Omega_{t}^{c}),\quad &2p+k\le K,\\
          \partial_{t}^{k}R_{IJKL}\in L^{2}(\mathbb{R}^{3}),\quad &k\le K,\\
          \partial^{p}\partial_{t}^{k}e_{I}\in L^{2}(\Omega_{t})\cap L^{2}(\Omega_{t}^{c}), \quad &2p+k\le K+1,\\
          e_{I}-\partial_{I}\in L^{2}(\Omega_{t})\cap L^{2}(\Omega_{t}^{c}), \quad &I=0,1,2,3,\\
          \partial^{p}\partial_{t}^{k}\Gamma_{IJ}^{K}\in L^{2}(\Omega_{t})\cap L^{2}(\Omega_{t}^{c}), \quad &2p+k\le K+1.
       \end{aligned}
    \end{equation}
	We also require the following initial vanishing conditions:	 
	\begin{equation}\label{vanishing fluid quantities}
	   \begin{aligned}
	       &(\Theta^{J}\nabla_{J}G+G\nabla_{I}\Theta^{I})|_{t=0},\quad(\Vert V\Vert^{2}+\Theta^{I}\Theta_{I})\big|_{t=0}, \quad (\Theta^{J}\nabla_{J}\Theta_{I}+\frac{1}{2}\nabla_{I}\Vert V\Vert^{2})\big|_{t=0},\\
	       &(\Theta^{J}\nabla_{J}\left(\Theta^{K}\nabla_{K}\Theta_{I}\right)-\frac{1}{2}\nabla^{J}\Vert V\Vert^{2}\nabla_{J}\Theta_{I}+\frac{1}{2}\nabla_{I}(\Theta^{J}\nabla_{J}\Vert V\Vert^{2})-\frac{1}{2}\epsilon_{J}\omega_{IJ}\nabla_{J}\Vert V\Vert^{2})\big|_{t=0},\\
	       &(\square \Vert V\Vert^{2}+2(\nabla^{I}\Theta^{J})(\nabla_{I}\Theta_{J})+2R_{KI}\Theta^{K}\Theta^{I}+2\epsilon_{K}\nabla_{K}\omega_{KI}\Theta^{I}-\frac{2}{G}(\Theta^{J}\nabla_{J}(\Theta^{K}\nabla_{K}G)+\nabla_{J}\Theta^{J}\Theta^{K}\nabla_{K}G))\big|_{t=0},\\
	       &(\omega_{IJ}-\nabla_{I}\Theta_{J}+\nabla_{J}\Theta_{I})\big|_{t=0}=0
	   \end{aligned}
	\end{equation}
	for the fluid quantities, and 
	\begin{equation}\label{vanishing geometric quantities}
	   \begin{aligned}
	      &(\Gamma_{JI}^{K}e_{K}-\Gamma_{IJ}^{K}e_{K}+\left[e_{I}, e_{J}\right], \quad R_{[IJK]L})\big|_{t=0}, \quad (R_{IJKL}-R_{KLIJ})|_{t=0},\quad\nabla_{[I}R_{JK]LM}|_{t=0}\\
	      &(R_{IJ}-\chi_{\Omega}(G\Theta_{I}\Theta_{J}-pg_{IJ}+\frac{1}{2}G\Vert V\Vert^{2}g_{IJ}))\big|_{t=0},\\ &\left\{\nabla^{J}R_{JIKL}-\chi_{\Omega}\left[\nabla_{K}\left(G\Theta_{L}\Theta_{I}-pm_{LI}+\frac{1}{2}G\Vert V\Vert^{2}m_{LI}\right)-\nabla_{L}\left(G\Theta_{K}\Theta_{I}-pm_{KI}+\frac{1}{2}G\Vert V\Vert^{2}m_{KI}\right)\right]\right\}\Big|_{t=0},\\
	      &(R_{\ MIJ}^{K}-(D_{I}\Gamma_{JM}^{K}-D_{J}\Gamma_{IM}^{K})-(\Gamma_{IL}^{K}\Gamma_{JM}^{L}-\Gamma_{JL}^{K}\Gamma_{IM}^{L})-(\Gamma_{JI}^{L}-\Gamma_{IJ}^{L})\Gamma_{LM}^{K}\big|_{t=0}=0
	   \end{aligned}
	\end{equation}
	for the geometric quantities. Additionally, we assume that the equations and the geometric properties of the curvature tensor are initially satisfied. We then present the main theorem of this study.
	\begin{theorem}\label{main thm}
		Assume that the regularity and compatibility conditions $(\ref{compatibility conditions})$- $(\ref{vanishing geometric quantities})$ are satisfied initially, with $K$ sufficiently large. Then, there exists $T>0$ such that a unique solution to the system $(\ref{frame main equ})$- $(\ref{transport Ch})$ exists in domain $\Omega=[0,T]\times\Omega_{0}$, satisfying the regularity conditions $(\ref{regularity conditions})$ for all $t\in[0,T]$. Moreover, the fluid velocity $u$, given by
		\begin{equation*}
		     u:=\frac{V}{\Vert V\Vert}=\frac{1}{\sqrt{-\sum_{J}\epsilon_{J}(\Theta^{J})^{2}}}\Theta^{K}e_{K}
		\end{equation*}
		and the metric $g$, defined by 
		\begin{equation*}
		    g^{\mu\nu}:=-e_{0}^{\mu}e_{0}^{\nu}+\sum_{\tilde{I}=1}^{3}e_{\tilde{I}}^{\mu}e_{\tilde{I}}^{\nu},
		\end{equation*}
		satisfy the Einstein's equations $(\ref{einstein equ})$, as well as the conservation laws $(\ref{energycons equ})$ and $(\ref{masscons equ})$.
	\end{theorem}
\subsection{History and related works}
    One of the earliest works on the existence of solutions to Einstein's equations is \cite{Bruhat}. Over the past few decades, well-posedness results have been established for the Einstein equations coupled with various matter fields, see \cite{Friedrich,Fournodavlos,Luk,Speck,Bieri,Kreiss,Mihalis}. Significant progress has also been made in local theory, particularly in the context of isolated bodies with free boundaries. In the non-relativistic setting, the well-posedness of free boundary fluid models is a subtle issue, requiring a variety of approaches to establish the local existence and uniqueness of solutions. In the irrotational case, which is closely related to the water wave problem describing the motion of the ocean's surface under Earth's gravity, early foundational contributions were made by Wu \cite{Wu1, Wu2}. Let $\tilde{u}$ denote the velocity and $\tilde{p}$ the pressure, with $B_{t}$ representing the fluid domain at time $t$ and $\partial B_{t}$ its boundary. Let $\tilde{n}$ be the unit outward normal to $\partial B_{t}$. Their work demonstrated that, by applying the material derivative $D_{t}:=\partial_{t}+\tilde{u}\cdot\nabla$ of the Euler equations, one can derive the following quasilinear system of equations
    \begin{equation}\label{water waves}
       \begin{cases}
           \Delta \tilde{u}=0 &\quad \text{in} \ B_{t} \\
           \left(D_{t}^{2}+\tilde{a}\nabla_{\tilde{n}}\right)\tilde{u}=-\nabla D_{t}\tilde{p} &\quad \text{on} \ \partial B_{t}  
       \end{cases},
    \end{equation} 
    where $\tilde{a}:=-\frac{\partial \tilde{p}}{\partial \tilde{n}}$ is positive on the boundary, and $\nabla_{\tilde{n}}$ represents the Dirichlet-Neumann map. In the context of general relativity, the well-posedness of a fluid free-boundary problem was first established, to our knowledge, in \cite{Miao1}. For results on a relativistic fluid free-boundary problem in a fixed spacetime, see \cite{Miao2,Oliynyk1,Oliynyk2,Oliynyk3,GL,Zeming}.  
    
    In the case of non-vanishing vorticity,  the local well-posedness of the free boundary problem for an incompressible ideal fluid in two space dimensions was established in \cite{Iguchi}. In \cite{Christodoulou2} the following system
    \begin{equation}\label{nonvanishing vorticity}
       \begin{cases}
           \Delta \tilde{p}=-\nabla_{i}\tilde{u}^{\ell}\nabla_{\ell}\tilde{u}^{i}&\quad \text{in}\ B_{t},\quad \tilde{p}=0\quad\text{on} \ \partial B_{t},\\
           \Delta D_{t}\tilde{p}=-(\partial_{k}\tilde{p})\Delta \tilde{u}^{k}-G(\nabla \tilde{u}, \nabla^{2}\tilde{p})&\quad \text{in}\ B_{t},\quad D_{t}\tilde{p}=0\quad\text{on}\  \partial B_{t},
       \end{cases}
    \end{equation}
    where $G(\nabla \tilde{u}, \nabla^{2}\tilde{p})=4\delta^{ij}\nabla_{i}\tilde{u}^{k}\nabla_{j}\nabla_{k}\tilde{p}+2(\nabla_{i}\tilde{u}^{j})(\nabla_{j}\tilde{u}^{k})\nabla_{k}\tilde{u}^{i}$ was considered. The authors in \cite{Christodoulou2} employed elliptic regularity theory to facilitate a prior estimates. The local existence of smooth solutions to this problem was proved in \cite{Lindblad}. Moreover, the results in \cite{Wu2} was extended to the case where $\mathrm{curl}\tilde{u}\ne 0$ in \cite{Ping}. Motivated by the aforementioned work \cite{Wu1,Wu2,Christodoulou2,Ping}, the main purpose of this paper is to extend the well-posedness result in \cite{Miao1} by removing the irrotational assumption and considering more general equations of state. The following section will provide a detailed overview of the ideas under discussion.

\subsection{Main ideas for the proof}	
	\subsubsection{Equations in the frame}\label{subsec1.3.1}
	We begin by recalling the transport equations $(\ref{transport e})$ and $(\ref{transport Ch})$, which govern the evolution of the connection coefficients $\Gamma_{IJ}^{K}$ and the components of the frame elements $e_{I}^{\mu}$ in the Lagrangian coordinates $\left\{x^{\mu}\right\}$. These equations are also coupled to the Bianchi equations for the curvature and the wave-boundary system for the fluid components $\Theta^{I}$.
	
	Next, we turn our attention to the fluid equations. We first derive the equations for $\Theta^{I}$ and $\Vert V\Vert^2$. By expressing the velocity $V$ in terms of its frame components $\Theta^{I}$, rather than using  arbitrary coordinate components $V^{\nu}$, the boundary equation for the fluid velocity involves only the connection coefficients $\Gamma$, without the appearance of their derivatives. We start by taking the covariant derivative $\nabla_{V}$ of the equation $(\ref*{alongu})$ to get
	\begin{equation}\label{2Vequ}
	    \nabla_{V}^{2}V^{\nu}-\frac{1}{2}\nabla^{\lambda}\Vert V \Vert^2 \nabla_{\lambda}V^{\nu}+\frac{1}{2}\nabla^{\nu}D_{V}\Vert V\Vert^{2}-\frac{1}{2}g^{\mu\nu}\nabla^{\lambda}\Vert V\Vert^{2}\omega_{\mu\lambda}=0,    
	\end{equation}
	where $\omega$ denotes the vorticity of the fluid, defined by
	\begin{equation*}
	        \omega=\mathrm{d}V.
	\end{equation*}
	Note that equation $(\ref{alongu})$ implies that $\omega$ satisfies the transport equation
	\begin{equation*}
	\mathcal{L}_{V}\omega=0.
	\end{equation*}
	Here $\mathcal{L}_{V}\omega$ denotes the Lie derivative of $g$ with respect to $V$. In coordinates, this can be rewritten as 
	\begin{equation}\label{transport w}
	    \nabla_{V}\omega_{\mu\nu}+\left(\nabla_{\mu}V^{\lambda}\right)\omega_{\lambda\nu}+\left(\nabla_{\nu}V^{\lambda}\right)\omega_{\mu\lambda}=0.
	\end{equation}
	As we will see, equation $(\ref{2Vequ})$ on $\partial\Omega$ satisfies
	\begin{equation*}
	    \nabla_{V}^{2}V^{\nu}+\frac{a}{2}\nabla_{n}V^{\nu}+\frac{1}{2}\nabla^{\nu}D_{V}\Vert V\Vert^{2}-\frac{1}{2}g^{\mu\nu}\nabla^{\lambda}\Vert V\Vert^{2}\omega_{\mu\lambda}=0.  
	\end{equation*}
    Taking the inner product with $e_{I}$ leads to
    \begin{equation*}
        \left(D_{V}^{2}+\frac{a}{2}D_{n}\right)\Theta^{I}=-\frac{\epsilon_{I}}{2}D_{I}D_{V}\Vert V\Vert^{2}+\frac{\epsilon_{K}}{2}D_{K}\Vert V\Vert^{2}\Gamma_{KJ}^{I}\Theta^{J}+\frac{\epsilon_{I}\epsilon_{J}}{2}\omega_{IJ}D_{J}\Vert V\Vert^{2}.
    \end{equation*}
    In the frame, $a$ takes the form 
	\begin{equation*}
	    a=\sqrt{\sum_{I}\epsilon_{I}\left(D_{I}\Vert V\Vert^{2}\right)^2}
	\end{equation*}
	and $\omega_{IJ}=\nabla_{I}\Theta_{J}-\nabla_{J}\Theta_{I}$. Using the fact $\partial_{t}=\hat{\Theta}^{0}D_{0}+\hat{\Theta}^{\tilde{I}}D_{\tilde{I}}$ and introducing the notation 
	\begin{equation*}
	     \gamma:=\frac{a}{2\Vert V\Vert^{2}}=\frac{\sqrt{\sum_{I}\epsilon_{I}\left(D_{I}\Vert V\Vert^{2}\right)^2}}{2\Vert V\Vert^{2}},
	\end{equation*}
	we arrive at
	\begin{equation}\label{2V boundary equ}
	   \begin{aligned}
	      \left(\partial_{t}^{2}+\gamma D_{n}\right)\Theta^{I}=&-\frac{\epsilon_{I}}{2\Vert V\Vert^{2}}D_{I}D_{V}\Vert V\Vert^{2}+\frac{\epsilon_{K}}{2\Vert V\Vert^{2}}D_{K}\Vert V\Vert^{2}\Gamma_{KJ}^{I}\Theta^{J}+\frac{\epsilon_{I}\epsilon_{J}}{2\Vert V\Vert^{2}}\omega_{IJ}D_{J}\Vert V\Vert^{2}\\
	      &-\frac{1}{2\Vert V\Vert^{2}}\left(\partial_{t}\Vert V\Vert^{2}\right)\partial_{t}\Theta^{I}.
	   \end{aligned}	   
	\end{equation}
	We next turn to the interior wave equation for $\Theta^{I}$. Using the definition of $\omega$, we get
	\begin{equation}\label{def w}
	    \nabla^{\mu}V_{\nu}-\nabla_{\nu}V^{\mu}=g^{\lambda\mu}\omega_{\lambda\nu}.
	\end{equation}
	Taking $\nabla_{\nu}$ to equation $(\ref{orthu})$ yields
	\begin{equation}\label{2orthu}
	    \nabla_{\nu}D_{V}G+\nabla_{\nu}\nabla_{\mu}V^{\mu}G+\nabla_{\mu}V^{\mu}\nabla_{\nu}G=0.
	\end{equation}
	Here 
	\begin{equation}\label{G equ}
	    \nabla_{\mu}G=G^{'}\nabla_{\mu}\Vert V\Vert^{2}, 
	\end{equation}
    with 
    \begin{equation*}
         G^{'}=\frac{1}{2}\left(\frac{1}{\eta^{2}}-1\right)\frac{G}{\Vert V\Vert^{2}}\ge 0.
    \end{equation*}
    Applying $\nabla_{\mu}$ to equation $(\ref{def w})$, commuting $\nabla_{\mu}$ and $\nabla_{\nu}$, and using equation $(\ref{2orthu})$, leads to
	\begin{equation}\label{2V interior equ}
	   \nabla_{\mu}\nabla^{\mu}V_{\nu}=R_{\sigma\nu}V^{\sigma}+\nabla^{\lambda}\omega_{\lambda\nu}-\frac{1}{G}\left(\nabla_{\nu}D_{V}G+\nabla_{\mu}V^{\mu}\nabla_{\nu}G\right),
	\end{equation}
	where in view of $(\ref{einstein equ})$
	\begin{equation}\label{Ricc}
	     R_{\sigma\nu}=\begin{cases}
	          GV_{\sigma}V_{\nu}-pg_{\sigma\nu}+\frac{1}{2}G\Vert V\Vert^{2}g_{\sigma\nu} \quad&\text{in}\ \Omega\\
	          0\quad&\text{in}\ \Omega^{c}
	     \end{cases}.
	\end{equation}
	Expressing the equation \ref{2V interior equ} in the frame yields
	\begin{equation}\label{2V interior frame equ}
	   \begin{aligned}
	       \Box\Theta^{I}&=-\epsilon_{K}\left(2\Gamma_{KN}^{I}D_{K}\Theta^{N}+\Gamma_{KN}^{I}\Gamma_{KM}^{N}\Theta^{M}+\Theta^{N}D_{K}\Gamma_{KN}^{I}-\Theta^{N}\Gamma_{KK}^{M}\Gamma_{MN}^{I}\right)\\
	       &+\epsilon_{I}R_{KI}\Theta^{K}+\epsilon_{I}\epsilon_{K}\nabla_{K}\omega_{KI}-\frac{\epsilon_{I}}{G}\left(D_{I}D_{V}G+D_{I}G\nabla_{J}\Theta^{J}\right).
	   \end{aligned}	   
	\end{equation}
	
    We are now in a position to derive wave eqautions for $\sigma^{2}$ and $D_{V}\sigma^{2}$ with Dirichlet boundary conditions. A direct computation using $\Vert V\Vert^{2}=-V_{\mu}V^{\mu}$ shows that
    \begin{equation*}
        \Box\Vert V\Vert^{2}=-2\left(\nabla_{\nu}V_{\mu}\right)\left(\nabla^{\nu}V^{\mu}\right)-2V^{\mu}\nabla_{\nu}\nabla^{\nu}V_{\mu},\quad \Vert V\Vert^{2}|_{\partial\Omega}\equiv 1.
    \end{equation*}
    Using equation $(\ref{2V interior equ})$, we have
    \begin{equation*}
       \Box\Vert V\Vert^{2}=-2\left(\nabla_{\nu}V_{\mu}\right)\left(\nabla^{\nu}V^{\mu}\right)-2V^{\mu}\left[R_{\sigma\mu}V^{\sigma}+\nabla^{\lambda}\omega_{\lambda\mu}-\frac{1}{G}\left(\nabla_{\mu}D_{V}G+\nabla_{\nu}V^{\nu}\nabla_{\mu}G\right)\right].
    \end{equation*}
	In terms of the frame, this becomes
	\begin{equation}\label{2V^2 frame equ}
	    \Box\Vert V\Vert^{2}=-2\epsilon_{I}\epsilon_{J}\left(D_{I}\Theta^{J}+\Gamma_{IK}^{J}\Theta^{K}\right)^{2}-2R_{KI}\Theta^{K}\Theta^{I}-2\epsilon_{K}\nabla_{K}\omega_{KI}\Theta^{I}+\frac{2}{G}\left(D_{V}^{2}G+\nabla_{J}\Theta^{J}D_{V}G\right).
	\end{equation}
	For $D_{V}\Vert V\Vert^{2}$, using equation $(\ref{alongu})$ and commuting the original equation for $\Vert V\Vert^{2}$ with $D_{V}$, we obtain
	\begin{equation*}
	   \begin{aligned}
	       \Box D_{V}\Vert V\Vert^{2}&=4\left(\nabla^{m}V^{a}\right)\left(\nabla_{m}V^{\lambda}\right)\left(\nabla_{\lambda}V_{a}\right)-4V^{m}\left(\nabla^{a}V^{\lambda}\right)R_{\sigma a\lambda m}V^{\sigma}+4\left(\nabla^{a}V^{m}\right)\left(\nabla_{a}\nabla_{m}\Vert V\Vert^{2}\right)\\
	       &+\left[R_{\sigma m}V^{\sigma}+\nabla^{\lambda}\omega_{\lambda m}-\frac{1}{G}\left(\nabla_{m}D_{V}G+\nabla_{\mu}V^{\mu}\nabla_{m}G\right)\right]\nabla^{m}\Vert V\Vert^{2}+V^{m}\nabla^{\lambda}\Vert V\Vert^{2}R_{m\lambda}\\
	       &-2\nabla_{V}\left(R_{\sigma\mu}V^{\mu}V^{\sigma}+V^{\mu}\nabla^{\lambda}\omega_{\lambda\mu}-\frac{1}{G}D_{V}^{2}G-\frac{1}{G}\nabla_{\nu}V^{\nu}D_{V}G\right) 
	    \end{aligned}
	\end{equation*}
	and $D_{V}\Vert V\Vert^{2}|_{\partial\Omega}\equiv 0$. Using the fact that $\partial_{t}\Vert V\Vert^{2}=\frac{1}{\Vert V\Vert}D_{V}\Vert V\Vert^{2}$, we get
	\begin{equation}\label{2DV^2 frame equ}
	   \begin{aligned}
	       \square\partial_{t}\Vert V\Vert^{2}&=\frac{2}{G\Vert V\Vert}D_{V}^{3}G-\frac{2\epsilon_{K}}{\Vert V\Vert}\Theta^{J}\nabla_{V}\nabla_{K}\omega_{KJ}-\frac{2}{G^{2}\Vert V\Vert}D_{V}GD_{V}^{2}G-\frac{2\epsilon_{K}}{\Vert V\Vert}\nabla_{V}\Theta^{J}\nabla_{K}\omega_{KJ}\\
	       &-\frac{2}{\Vert V\Vert}\nabla_{V}\left(R_{IJ}\Theta^{I}\Theta^{J}-\frac{1}{G}\nabla_{J}\Theta^{J}D_{V}G\right)+\frac{1}{\Vert V\Vert}\Theta^{I}R_{JI}\nabla^{J}\Vert V\Vert^{2}\\
	       &+\frac{\epsilon_{I}}{\Vert V\Vert}\left[R_{JI}\Theta^{J}+\epsilon_{J}\nabla_{J}\omega_{JI}-\frac{1}{G}\left(D_{I}D_{V}G+\nabla_{J}\Theta^{J}D_{I}G\right)\right]D_{I}\Vert V\Vert^{2}\\
	       &+\frac{4\epsilon_{I}}{\Vert V\Vert}\left(D_{I}\Theta^{J}+\Gamma_{IK}^{J}\Theta^{K}\right)\left(D_{I}D_{J}\Vert V\Vert^{2}-\Gamma_{IJ}^{K}D_{K}\Vert V\Vert^{2}\right)-\frac{4\epsilon_{I}}{\Vert V\Vert}R_{LIJK}\left(D_{I}\Theta^{J}+\Gamma_{IM}^{J}\Theta^{M}\right)\Theta^{K}\Theta^{L}\\
	       &+\frac{4\epsilon_{I}\epsilon_{K}}{\Vert V\Vert}\left(D_{I}\Theta^{J}+\Gamma_{IM}^{J}\Theta^{M}\right)\left(D_{J}\Theta^{K}+\Gamma_{JM}^{K}\Theta^{M}\right)\left(D_{I}\Theta^{K}+\Gamma_{IP}^{K}\Theta^{P}\right)\\     
	       &+\frac{\epsilon_{I}\partial_{t}\Vert V\Vert^{2}}{4\Vert V\Vert^{4}}D_{I}\Vert V\Vert^{2}D_{I}\Vert V\Vert^{2}-\frac{\epsilon_{I}}{\Vert V\Vert^{2}}D_{I}\Vert V\Vert^{2}D_{I}\partial_{t}\Vert V\Vert^{2}-\frac{1}{2\Vert V\Vert^{2}}\Box\Vert V\Vert^{2}\partial_{t}\Vert V\Vert^{2}.	          
	   \end{aligned}
	\end{equation}
    Equations $(\ref{2V boundary equ})$, $(\ref{2V interior frame equ})$, $(\ref{2V^2 frame equ})$ and $(\ref{2DV^2 frame equ})$ are the fluid equations in the frame.	
	
	We now proceed to derive the equations satisfied by the curvature components. It is important to note that the curvature is discontinuous at the boundary. To estimate the curvature, we must derive a first-order hyperbolic system by combining the contracted Bianchi identity
	\begin{equation}\label{the contracted Bianchi identity}
	   \nabla_{Z}R_{WJ}-\nabla_{W}R_{ZJ}=\sum_{I}\epsilon_{I}\nabla_{I}R_{IJZW}
	\end{equation} 
	with the differential Bianchi equation 
	\begin{equation}\label{the differential Bianchi equation}
	     \nabla_{[I}R_{JK]ZW}=0.
	\end{equation}
	If $Z$ and $W$ are tangential to the fluid boundary $\partial\Omega$, equation $(\ref{the contracted Bianchi identity})$ and $(\ref{the differential Bianchi equation})$ can be viewed as a weak equation on 
	\begin{equation}\label{whole domain}
	      \Sigma:=\Omega\cup\Omega^{c}.
	\end{equation}
    The two form $F^{AB}$, with components in the frame $\left\{e_{I}\right\}$ are given by
    \begin{equation*}
         F_{IJ}^{AB}=R(e_{I},e_{J},X_{A},X_{B}),
    \end{equation*}
    where $X_{A}$ and $X_{B}$ (with $A\ne B$) are tangential to the boundary. Here, $F_{IJ}\equiv F_{IJ}^{AB}$ of $F$ can be decomposed into the electric parts $E_{\tilde{I}}\equiv E_{\tilde{I}}^{AB}$ and the magnetic parts $H^{\tilde{I}}\equiv H_{AB}^{\tilde{I}}$, where $\tilde{I}=1, 2, 3$, as defined by
    \begin{equation*}
          E_{\tilde{I}}=F_{\tilde{I}0},\quad H^{1}=-F_{23}, \quad H^{2}=-F_{31},\quad H^{3}=-F_{12}.
    \end{equation*}
    We use the Bianchi equations $(\ref{the contracted Bianchi identity})$ and $(\ref{the differential Bianchi equation})$ to derive a Maxwell system for $(E, H)$ in the usual way. Contracting identity $\nabla^{\mu}R_{\mu\nu\lambda\kappa}=\nabla_{\lambda}R_{\kappa\nu}-\nabla_{\kappa}R_{\lambda\nu}$ with $X_{A}^{\lambda}$ and $X_{B}^{\kappa}$, we obtain in $\Omega$
    \begin{equation}\label{contract B equ}
       \begin{aligned}
          \nabla^{\mu}F_{\mu\nu}^{AB}&=
          X_{A}^{\lambda}X_{B}^{\kappa}\left[G^{'}\nabla_{\lambda}\Vert V\Vert^{2}\left(V_{\kappa}V_{\nu}+\frac{1}{2}\Vert V\Vert^{2}g_{\kappa\nu}\right)+G\nabla_{\lambda}\left(V_{\kappa}V_{\nu}\right)+\frac{1}{2}G\nabla_{\lambda}\Vert V\Vert^{2}g_{\kappa\nu}-\nabla_{\lambda}pg_{\kappa\nu}\right]\\
          &-X_{A}^{\lambda}X_{B}^{\kappa}\left[G^{'}\nabla_{\kappa}\Vert V\Vert^{2}\left(V_{\lambda}V_{\nu}+\frac{1}{2}\Vert V\Vert^{2}g_{\lambda\nu}\right)+G\nabla_{\kappa}\left(V_{\lambda}V_{\nu}\right)+\frac{1}{2}G\nabla_{\kappa}\Vert V\Vert^{2}g_{\lambda\nu}-\nabla_{\kappa}pg_{\lambda\nu}\right]\\
          &+R_{\mu\nu\lambda\kappa}\left(X_{B}^{\kappa}\nabla^{\mu}X_{A}^{\lambda}+X_{A}^{\lambda}\nabla^{\mu}X_{B}^{\kappa}\right),
       \end{aligned}  
    \end{equation}
    and in $\Omega^{c}$
    \begin{equation*}
       \nabla^{\mu}F_{\mu\nu}^{AB}=R_{\mu\nu\lambda\kappa}\left(X_{B}^{\kappa}\nabla^{\mu}X_{A}^{\lambda}+X_{A}^{\lambda}\nabla^{\mu}X_{B}^{\kappa}\right).
    \end{equation*}
    In the frame, this takes the form 
    \begin{equation}\label{I frame equ}
        \epsilon_{I}D_{I}F_{IK}^{AB}=\mathcal{I}_{K}^{AB}
    \end{equation}
	where
	\begin{equation}\label{I component1}
	   \begin{aligned}
	       \mathcal{I}_{K}^{AB}&=G^{'}\epsilon_{I}\epsilon_{K}\nabla_{X_{A}}\Vert V\Vert^{2}X_{B}^{I}\Theta^{I}\Theta^{K}+\frac{1}{2}G^{'}\nabla_{X_{A}}\Vert V\Vert^{2}\Vert V\Vert^{2}X_{B}^{K}\epsilon_{K}+\frac{1}{2}G\epsilon_{K}X_{B}^{K}\nabla_{X_{A}}\Vert V\Vert^{2}+G\epsilon_{K}X_{A}^{I}X_{B}^{J}\omega_{IJ}\Theta^{K}\\
 	       &+G\epsilon_{I}\epsilon_{K}\left(X_{B}^{I}X_{A}^{J}-X_{A}^{I}X_{B}^{J}\right)\Theta^{I}\left(D_{J}\Theta^{K}+\Gamma_{JM}^{K}\Theta^{M}\right)+\epsilon_{K}X_{B}^{K}\nabla_{X_{A}}p-G^{'}\epsilon_{I}\epsilon_{K}\nabla_{X_{B}}\Vert V\Vert^{2}X_{A}^{I}\Theta^{I}\Theta^{K}\\
	       &-\frac{1}{2}G^{'}\nabla_{X_{B}}\Vert V\Vert^{2}\Vert V\Vert^{2}X_{A}^{K}\epsilon_{K}-\frac{1}{2}G\epsilon_{K}X_{A}^{K}\nabla_{X_{B}}\Vert V\Vert^{2}-\epsilon_{K}X_{A}^{K}\nabla_{X_{B}}p+\epsilon_{I}\Gamma_{II}^{J}F_{JK}^{AB}+\epsilon_{I}\Gamma_{IK}^{J}F_{IJ}^{AB}\\
	       &+\epsilon_{J}R_{JKLI}\left(X_{B}^{I}D_{J}X_{A}^{L}+X_{A}^{L}D_{J}X_{B}^{I}+X_{B}^{I}\Gamma_{JM}^{L}X_{A}^{M}+X_{A}^{L}\Gamma_{JM}^{I}X_{B}^{M}\right), \quad \text{in}\ \Omega,
	   \end{aligned}
	\end{equation}
	and 
	\begin{equation}\label{I component2}
	       \mathcal{I}_{K}^{AB}=\epsilon_{I}\Gamma_{II}^{J}F_{JK}^{AB}+\epsilon_{I}\Gamma_{IK}^{J}F_{IJ}^{AB}+\epsilon_{J}R_{JKLI}\left(X_{B}^{I}D_{J}X_{A}^{L}+X_{A}^{L}D_{J}X_{B}^{I}+X_{B}^{I}\Gamma_{JM}^{L}X_{A}^{M}+X_{A}^{L}\Gamma_{JM}^{I}X_{B}^{M}\right), \ \text{in}\ \Omega^{c}.
	\end{equation}
	On the other hand, contracting $\nabla_{[\alpha}R_{\beta\lambda]\mu\nu}$ with $X_{A}^{\mu}$ and $X_{B}^{\nu}$ we get
	\begin{equation}\label{B equ}
	      \nabla_{[\alpha}F_{\beta\lambda]}^{AB}=X_{B}^{\nu}\left(\nabla_{[\alpha}X_{A}^{\mu}\right)R_{\beta\lambda]\mu\nu}+X_{A}^{\mu}\left(\nabla_{[\alpha}X_{B}^{\nu}\right)R_{\beta\lambda]\mu\nu}.
	\end{equation}
	In terms of the frame components we get
	\begin{equation}\label{J frame equ}
	     D_{[I}F_{JK]}^{AB}=\mathcal{J}_{IJK}^{AB}
	\end{equation}
	where
	\begin{equation}\label{J component}
	   \begin{aligned}
	       \mathcal{J}_{IJK}^{AB}&=2\Gamma_{[IK}^{L}F_{J]L}^{AB}+X_{B}^{L}R_{ML[JK}D_{I]}X_{A}^{M}+X_{B}^{L}R_{ML[JK}\Gamma_{I]N}^{M}X_{A}^{N}\\
	       &+X_{A}^{M}R_{ML[JK}D_{I]}X_{B}^{L}+X_{A}^{M}R_{ML[JK}\Gamma_{I]N}^{L}X_{B}^{N}.
	   \end{aligned}	    
	\end{equation}
	Equations $(\ref{I frame equ})$ and $(\ref{J frame equ})$ can be written in terms of $E$ and $H$ as 
	\begin{equation}\label{Maxwell equ curl}
	       D_{0}E+\mathrm{curl}H=\mathcal{I}, \quad D_{0}H-\mathrm{curl}E=\mathcal{J}^{\ast},   
	\end{equation}
	and
	\begin{equation}\label{Maxwell equ div}
	     \epsilon_{\tilde{I}}D_{\tilde{I}}E_{\tilde{I}}=\mathcal{I}_{0}^{AB}, \quad  -D_{\tilde{I}}H^{\tilde{I}}=\left(\mathcal{J}^{\ast}\right)^{0},
	\end{equation}
	where $\left(\mathcal{J}^{\ast}\right)^{I}=\epsilon^{IJKL}\mathcal{J}_{JKL}$ and $\mathrm{curl}$ denotes the usual $\mathrm{curl}$ operator with respect to $D_{\tilde{I}}, \tilde{I}=1, 2, 3$:
	\begin{equation*}
	    \mathrm{curl} H=\begin{pmatrix}
	      D_{2}H^{3}-D_{3}H^{2}\\
	      D_{1}H^{3}-D_{3}H^{1}\\
	      D_{1}H^{2}-D_{2}H^{1}	    
	    \end{pmatrix}.
	\end{equation*}
	Defining $W\equiv W^{AB}=\left(W_{1}, W_{2}, \dots, W_{6}\right)$ where $W_{m} = H^{m}$ for $m = 1, 2, 3$ and $W_{m} = E_{m-3}$ for $m = 4, 5, 6,$ and with
	\begin{equation*}
	   \mathcal{A}^{1}=\begin{pmatrix}
	       0&0&0&0&0&0\\
	       0&0&0&0&0&1\\
	       0&0&0&0&-1&0\\
	       0&0&0&0&0&0\\
	       0&0&-1&0&0&0\\
	       0&1&0&0&0&0\\
	   \end{pmatrix},
	   \mathcal{A}^{2}=\begin{pmatrix}
	   0&0&0&0&0&-1\\
	   0&0&0&0&0&0\\
	   0&0&0&1&0&0\\
	   0&0&1&0&0&0\\
	   0&0&0&0&0&0\\
	   -1&0&0&0&0&0\\
	   \end{pmatrix},
	   \mathcal{A}^{3}=\begin{pmatrix}
	   0&0&0&0&1&0\\
	   0&0&0&-1&0&0\\
	   0&0&0&0&0&0\\
	   0&-1&0&0&0&0\\
	   1&0&0&0&0&0\\
	   0&0&0&0&0&0\\
	   \end{pmatrix},
	\end{equation*}
	equation $(\ref{Maxwell equ curl})$ becomes a first order symmetric hyperbolic system for $W$ of the form
	\begin{equation}\label{1st hyperbolic system}
	    D_{0}W+\mathcal{A}^{\tilde{I}}D_{\tilde{I}}W=\mathcal{K}, \quad \mathcal{K}=\left(\mathcal{J}_{1}^{\ast}, \mathcal{J}_{2}^{\ast}, \mathcal{J}_{3}^{\ast}, \mathcal{I}_{1}, \mathcal{I}_{2}, \mathcal{I}_{3}\right).
	\end{equation}
    Since $e_{0}$ does not coincide with $\partial_{t}$, and is not adapted to the foliation by constant $t$ slices, we further decompose this as
    \begin{equation}\label{hyperbolic system}
       \mathcal{B}^{\mu}\partial_{\mu}W=\mathcal{K}, \quad \mathcal{B}^{0}:=e_{0}^{0}+e_{\tilde{I}}^{0}\mathcal{A}^{\tilde{I}}, \quad \mathcal{B}^{j}:=e_{0}^{j}+e_{\tilde{I}}^{j}\mathcal{A}^{\tilde{I}}, j = 1, 2, 3. 
    \end{equation}
	Here, in coordinates, we have 
	\begin{equation*}
	   e_{0}^{\mu}=\frac{1}{\hat{\Theta}^{0}}\left(\delta_{0}^{\mu}-\hat{\Theta}^{\tilde{I}}e_{\tilde{I}}^{\mu}\right).
	\end{equation*}
    Note that the first-order symmetric system for $W$ is weakly defined.
	
	To summarize, equations $(\ref{transport e})$, $(\ref{transport Ch})$, $(\ref{transport w})$, $(\ref{2V boundary equ})$, $(\ref{2V interior frame equ})$, $(\ref{2V^2 frame equ})$, $(\ref{2DV^2 frame equ})$, $(\ref{I frame equ})$, $(\ref{J frame equ})$  and $(\ref{hyperbolic system})$ are the working equations that can be used for a priori estimation and iteration. Let $\mathcal{F}$ denote the source term described exactly as the right-hand side of equations $(\ref{transport e})$, $(\ref{transport Ch})$, $(\ref{transport w})$, $(\ref{2V boundary equ})$, $(\ref{2V interior frame equ})$, $(\ref{2V^2 frame equ})$, $(\ref{2DV^2 frame equ})$, $(\ref{I frame equ})$, $(\ref{J frame equ})$ and $(\ref{hyperbolic system})$. Then we record the structure of these equations as 
	\begin{equation}\label{main structure equ}
        \begin{cases}
           \left(\partial_{t}^{2}+\gamma D_{n}\right)\Theta^{I}=\mathcal{F}_{\Theta^{I},\partial\Omega} \quad &\text{on}\ \partial\Omega\\
           \Box\Theta^{I}=\mathcal{F}_{\Theta^{I},\Omega} \quad &\text{in}\ \Omega\\
           \Box\Vert V\Vert^{2}=\frac{2}{G}D_{V}^{2}G+\mathcal{F}_{\Vert V\Vert^{2}} \quad &\text{in}\ \Omega\\
           \Box D_{V}\Vert V\Vert^{2}=\frac{2}{G}D_{V}^{3}G+\mathcal{F}_{D_{V}\Vert V\Vert^{2}} \quad &\text{in}\ \Omega\\
           \Vert V\Vert^{2}=1, \quad D_{V}\Vert V\Vert^{2}=0 \quad &\text{on}\ \partial\Omega\\
           \partial_{t}e_{\tilde{I}}=\mathcal{F}_{e_{\tilde{I}}}, \quad e_{0}=\frac{1}{\hat{\Theta}^{0}}(\hat{V}-\hat{\Theta}^{\tilde{I}}e_{\tilde{I}}) \quad &\text{in}\ \Sigma\\
           \partial_{t}\Gamma_{\tilde{I}J}^{K}=\mathcal{F}_{\Gamma_{\tilde{I}J}^{K}}, \quad \Gamma_{0J}^{K}=-\frac{\hat{\Theta}^{\tilde{I}}}{\hat{\Theta}^{0}}\Gamma_{\tilde{I}J}^{K} \quad &\text{in}\ \Sigma\\
           \nabla_{V}\omega_{IJ}=\mathcal{F}_{\omega_{IJ}}\quad &\text{in}\ \Sigma\\
           \mathcal{B}^{\mu}\partial_{\mu}W^{AB}=\mathcal{F}_{W^{AB}}, \quad &\text{in}\ \Sigma
        \end{cases}.
	\end{equation}
	Here, we follow the notation from $(\ref{whole domain})$, where
	\begin{equation*}
	     \Sigma_{t}=\Omega_{t}\cup\Omega_{t}^{c}=\Sigma\cap\left\{x^{0}=t\right\}
	\end{equation*}
	and $R$ denotes the Riemann curvature tensor of the metric $g$. Now, suppose there exists a soution to the equation $(\ref{main structure equ})$. To recover the original equations $(\ref{frame main equ})$ and $(\ref{Ricci equ})$, we focus on the initial data with respect to the equations in the derivation. As stated in Theorem \ref{main thm}, this is achieved by requiring that the vanishing requirements $(\ref{vanishing fluid quantities})$ and $\ref{vanishing geometric quantities}$ hold at the intial time. Next, we will outline the key ideas involved in the study of $(\ref{main structure equ})$.
	
  \subsubsection{A priori estimates} Motivated by the approach introduced in \cite{Miao2}, we start by deriving a quasilinear system for the fluid variables, i.e. $\Theta^{I}$ and $\Vert V\Vert^{2}$, from equations (\ref{einstein equ})-(\ref{main equ}). To control the higher-order derivatives of the unknowns, we commute $\partial_{t}^{k}$ with the evolution equations. It is crucial to carefully consider the vorticity term when analysing the orders of the derivatives. Since the vorticity satisfies a transport equation, commuting time derivatives allows a reduction in the order of operations. Then the quasilinear systems for the fluid quantities satisfy
  \begin{equation}\label{energy equ1'}
       \left(\partial_{t}^{2}+\gamma D_{n}\right)\partial_{t}^{k}\Theta^{I}=f_{k} \quad \text{on} \ \partial\Omega, \quad \square \partial_{t}^{k}\Theta^{I}=H_{k}\quad \text{in}\ \Omega,
  \end{equation} 
  and
  \begin{equation}\label{energy equ2'}
     \square \partial_{t}^{k+1}\Vert V\Vert^{2}=\frac{2G^{'}\Vert V\Vert^{2}}{G }\partial_{t}^{k+3}\Vert V\Vert^{2}-\partial_{t}^{k}\left(\frac{2\epsilon_{K}}{\Vert V\Vert}\Theta^{J}\nabla_{V}\nabla_{K}\omega_{KJ}\right)+F_{k}\quad \text{in}\ \Omega,\quad \partial_{t}^{k+1}\Vert V\Vert^{2}=0\quad \text{on}\ \partial\Omega.
  \end{equation}
 Here $f_{k}$, $H_{k}$ and $F_{k}$ correspond to the terms in lemma \ref{V higher order lemma}--\ref{V^2 higher order lemma}. For the fluid velocity equations $(\ref{energy equ1'})$, we multiply these equations by $\frac{1}{\gamma}\partial_{t}^{k}\Theta\sqrt{\vert g\vert}$, $\partial_{t}^{k}\Theta\sqrt{\vert g\vert}$,  respectively, and then integrate. By analyzing the signs of the boundary terms, we obtain the energy stated in lemma \ref{V energy lemma}
    \begin{equation*}
       \int_{\Omega_{t}} \left|\partial_{t,x}\partial_{t}^{k} \Theta\right|^{2}  \mathrm{d}x+\int_{\partial\Omega_{t}}\frac{1}{\gamma}\left|\partial_{t}^{k+1}\Theta\right|^{2} \mathrm{d}S.
    \end{equation*} 
For the wave equation for $\partial_{t}^{k+1}\Vert V\Vert^{2}$ with Dirichlet boundary conditions, we choose a suitable multiplier that consists of an appropriate linear combination of $\partial_{t}$ and the normal $n$. A technical challenge arises in controlling the terms $\frac{2G^{'}\Vert V\Vert^{2}}{G }\partial_{t}^{k+3}\Vert V\Vert^{2}$ and $\partial_{t}^{k}\left(\frac{2\epsilon_{K}}{\Vert V\Vert}\Theta^{J}\nabla_{V}\nabla_{K}\omega_{KJ}\right)$. However, since $\frac{2G^{'}\Vert V\Vert^{2}}{G} \geq 0$, the term $\frac{2G^{'}\Vert V\Vert^{2}}{G }\partial_{t}^{k+3}\Vert V\Vert^{2}$ can be merged with the wave operator to form a new wave operator $\overline{\square}$, that is,
\begin{equation*}
     \square \partial_{t}^{k+1}\Vert V\Vert^{2}-\frac{2G^{'}\Vert V\Vert^{2}}{G }\partial_{t}^{k+3}\Vert V\Vert^{2}\sim \overline{\square}\partial_{t}^{k+1}\Vert V\Vert^{2}.
\end{equation*} 
When using the transport equation satisfied by $\omega$, the term $\partial_{t}^{k}\nabla\nabla_{V}\omega$ will produce the higher-order derivative term $\nabla^{(2)}\partial_{t}^{k}\Theta$, leading to a loss of control. The key observation here is 
\begin{equation*}
    \partial_{t}^{k}\left(\frac{2\epsilon_{K}}{\Vert V\Vert}\Theta^{J}\nabla_{V}\nabla_{K}\omega_{KJ}\right)\sim \partial_{t}^{k}\left(\Theta^{J}\Box\Theta^{L}\omega_{LJ}\right)+\partial_{t}^{k-1}M,
\end{equation*} 
where $M$ consists of $D\partial_{t}\Theta$, $D^{2}\Theta$, $\partial_{t}R$, $D^{2}\partial_{t}\Vert V\Vert^{2}$. Then the system (\ref{energy equ2'}) can be used to get control of 
\begin{equation*}
    \int_{\Omega_{t}} \left|\partial_{t,x}\partial_{t}^{k}\Vert V\Vert^{2}\right|^{2}\mathrm{d}x+\int_{0}^{t}\int_{\partial\Omega_{t}} \left|\partial_{t,x}\partial_{t}^{k}\Vert V\Vert^{2}\right|^{2}\mathrm{d}S\mathrm{d}\tau.
\end{equation*}
We then apply elliptic estimates to bound two derivatives of $\Vert V\Vert^{2}$ in terms of one derivative $\partial_{t}\Vert V\Vert^{2}$. For the symmetric first-order hyperbolic system governing $R_{IJKL}$, we derive the estimates based on the existing results and algebraic relations. To obtain higher-order curvature estimates, we revisit the curl-divergence system (\ref{Maxwell equ curl}) and (\ref{Maxwell equ div}). The former represents the Maxwell system, while the latter represents the divergence system. Moreover, the equations for the curvature and fluid components are coupled with the transport equations of $e_{I}$ and $\Gamma$. The  detailed proof of a priori estimates is presented in Section \ref{sec2}.  

  \subsubsection{The iteration}
 Given that the curvature equation satisfies a first-order symmetric hyperbolic system, the linear existence theory is well-established. The idea for the iteration is to iteratively define $\Theta^{(m)}$, $\Lambda^{(m)}$, $\Sigma^{(m)}$ and $\omega^{(m)}$ as solutions of
 \begin{equation*}
    \begin{aligned}
      & \begin{cases}
      \left(\partial_{t}^{2}+\gamma^{(m)}\nabla_{n^{(m)}}\right)\Theta^{I,(m+1)}=S(\Theta^{(m)},\Lambda^{(m)}) \quad &\text{on}\ [0,T]\times\partial \Omega\\
      \square_{g^{(m)}}\Theta^{I,(m+1)}=F(\Theta^{(m)},\Lambda^{(m)})\quad&\text{in}\ [0,T]\times \Omega
      \end{cases},\\
      &\begin{cases}
      \square_{g^{(m)}}\Lambda^{(m+1)}=\left(\frac{1}{\eta^{2}}-1\right)\partial_{t}^{2}\Lambda^{(m+1)}+H(\Theta^{(m)},\Sigma^{(m)}), \quad&\text{in}\quad [0,T]\times \Omega\\
      \Lambda^{(m+1)}=0, \quad &\text{on} \quad [0,T]\times \partial \Omega
      \end{cases},\\
      &\partial_{t}\Sigma^{(m+1)}=\Lambda^{(m+1)},\\
      &\nabla_{V}\omega_{IJ}^{(m+1)}+\nabla_{I}\Theta^{K,(m+1)}\omega_{KJ}^{(m+1)}+\nabla_{J}\Theta^{K,(m+1)}\omega_{IK}^{(m+1)}=0,
    \end{aligned}    
 \end{equation*}
 where $S$, $F$ and $H$ are denoted in equations (\ref{fluid quantities for DV^2 equ}) and (\ref{fluid quantities for V equ}), respectively. We then proceed to establish the existence theory for the fluid equations using the Galerkin approximation. Based on the a priori estimates derived in Section \ref{sec2}, we demonstrate the convergence of the iterative scheme. Finally, we verify that our derived solutions satisfy the original equations (\ref{frame main equ}) and (\ref{Ricci equ}). It is important to note that the arbitrary curvature $R$ defined by the solution of (\ref{main structure equ}), (\ref{vanishing fluid quantities}) and (\ref{vanishing geometric quantities}), is actually the curvature tensor of the metric. In this case, we derive homogeneous equations for the vanishing fluid and geometric quantities, as presented in equations $(\ref{equy})$, $(\ref{equtrans})$, $(\ref{equZn})$, $(\ref{equqs})$, $(\ref{equbzdr})$. Owing to their vanishing on the initial slice and the arguments detailed in Section \ref{sec2}, these quantities remain zero. In the following section, we describe how to construct initial data for the system (\ref{main structure equ}) based on the initial data for (\ref{einstein equ}), which is an important step in solving Einstein's equations. 
 
\subsection{The initial data} Given that the energy momentum tensor $T_{\mu\nu}$ satisfies (\ref{energy momentum tensor}), the initial data for (\ref{einstein equ}) consist of a symmetric positive definite two form $\overline{\gamma}$, the second fundamental form $k$ and $V_{\mu}$. We assume that the above quantities satify the constraint equation
\begin{equation*}
   \begin{aligned}
      &\overline{R}-\Vert k\Vert_{\overline{\gamma}}^{2}+\left(\mathrm{tr}_{\overline{\gamma}}k\right)^{2}=2\chi_{\Omega_{0}}\left(G\left(\left(\left(V_{0}\right)^{2}-\overline{\gamma}^{ij}V_{i}V_{j}\right)^{\frac{1}{2}}\right)\left(V_{0}\right)^{2}-p\left(\left(\left(V_{0}\right)^{2}-\overline{\gamma}^{ij}V_{i}V_{j}\right)^{\frac{1}{2}}\right)\right),\\
      &\overline{\nabla}^{i}k_{ij}-\overline{\nabla}_{j}\left(\mathrm{tr}_{\overline{\gamma}}k\right)=\chi_{\Omega_{0}}\left(G\left(\left(\left(V_{0}\right)^{2}-\overline{\gamma}^{ij}V_{i}V_{j}\right)^{\frac{1}{2}}\right)V_{0}V_{i}\right),
   \end{aligned}    
\end{equation*} 
where $\overline{R}$ and $\overline{\nabla}$ are associated with the metric $\overline{\gamma}$ on the initial slice $\Sigma_{0}$. The initial slice is given as diffeomorphic to $\mathbb{R}^{3}$, with the Roman indices $i,j,k,\dots \in \left\{1,2,3\right\}$ corresponding to a fixed coordinate system on the initial slice. We use a Sobolev extension to extend $V_{\mu}$ over the entire slice $\left\{t=0\right\}$.  If the metric is provided, $V$ can also be extended to the exterior region. Subsequently, we work in Lagrangian coordinates so that $\partial_{t}=\frac{V}{\Vert V\Vert}$. It results in
\begin{equation}\label{g00}
    g_{00}|_{t=0}=-1,\quad \partial_{t}g_{00}|_{t=0}=0,
\end{equation}
and
\begin{equation}\label{g0k}
   g_{0k}|_{t=0}=\frac{1}{\Vert V\Vert}g(V,\partial_{k})=\frac{V_{k}}{\left(\left(V_{0}\right)^{2}-\overline{\gamma}^{ij}V_{i}V_{j}\right)^{\frac{1}{2}}},
\end{equation}
where
\begin{equation*}
    \Vert V\Vert\big|_{t=0}=\sqrt{-g(V,V)}=\left(\left(V_{0}\right)^{2}-\overline{\gamma}^{ij}V_{i}V_{j}\right)^{\frac{1}{2}}.
\end{equation*}
We also observe that the metric $g(0)=g|_{t=0}$ and its  inverse $g^{-1}(0)=g^{-1}\big|_{t=0}$ are fully determined by $\overline{\gamma}$ and the conditions specified in (\ref{g00}) and (\ref{g0k}). Next, we turn to the initial data for $\partial_{t}g_{ij}$. Let $T$ denote the future-directed unit normal vector field to  $\Sigma_{0}$, 
\begin{equation}\label{future directed unit vector}
   T=-\frac{g^{00}}{\sqrt{-g^{00}}}\partial_{t}-\frac{g^{0i}}{\sqrt{-g^{00}}}\partial_{i},
\end{equation}
where $g^{00}< 0$. By the definition of $k$, it has 
\begin{equation}
   2k_{ij}=\left(\mathcal{L}_{T}g\right)_{ij}=T(g_{ij})-g(\left[T,\partial_{i}\right],\partial_{j})-g(\partial_{i},\left[T,\partial_{j}\right]).
\end{equation}  
Thus we have 
\begin{equation}\label{gij'}
      \sqrt{-g^{00}}\partial_{t}g_{ij}=2k_{ij}+\partial_{i}\left(\frac{g^{0\beta}}{\sqrt{-g^{00}}}\right)g_{\beta j}+\partial_{j}\left(\frac{g^{0\beta}}{\sqrt{-g^{00}}}\right)g_{\beta i}+\frac{g^{0k}}{\sqrt{-g^{00}}}\partial_{k}g_{ij} \quad \text{at}\ t=0.
\end{equation}

The next step is to determine the initial values of the Christoffel symbols. For the components $\Gamma_{ij}^{\mu}$, by definition, it has
\begin{equation}\label{Gammaijmu}
   \Gamma_{ij}^{\mu}=g^{\mu k}g_{k\ell}\overline{\Gamma}_{ij}^{\ell}+g^{\mu 0}\left(-\frac{1}{\sqrt{-g^{00}}}k_{ij}+\frac{g^{0k}}{g^{00}}g_{k\ell}\overline{\Gamma}_{ij}^{\ell}\right) \quad \text{at}\ t=0.
\end{equation}
Since $g(\partial_{t},\partial_{t})=-1$ and $g(\nabla_{\partial_{t}}\partial_{t},\partial_{j})=\partial_{i}g(\partial_{t},\partial_{j})-g(\partial_{t},\nabla_{\partial_{i}}\partial_{j})$, we have
\begin{equation}\label{Gammai0mu}
    \Gamma_{i0}^{\mu}=g^{\mu j}\partial_{i}g_{0j}-g^{\mu j}\left(-\frac{1}{\sqrt{-g^{00}}}k_{ij}+\frac{g^{0k}}{g^{00}}g_{k\ell}\overline{\Gamma}_{ij}^{\ell}\right) \quad \text{at}\ t=0,
\end{equation}
and 
\begin{equation}\label{Gamma00mu1}
    \Gamma_{00}^{\mu}=g^{\mu i}g(\nabla_{\partial_{t}}\partial_{t},\partial_{i})=g^{\mu i}\partial_{t}g_{0i}.
\end{equation}
Hence we will find the initial data for $\partial_{t}g_{0i}$. Note that 
\begin{equation*}
    \partial_{t}g_{0i}=\frac{\partial_{t}V_{i}}{\Vert V\Vert}-\frac{V_{i}}{\Vert V\Vert^{2}}\partial_{t}\Vert V\Vert.
\end{equation*}
For the first term, by (\ref{alongu}),
\begin{equation*}
    \partial_{t}V_{i}=\overline{\Gamma}_{0i}^{\mu}V_{\mu}-\partial_{i}\Vert V\Vert \quad \text{at} \ t=0,
\end{equation*}
and 
\begin{equation*}
    \overline{\Gamma}_{0i}^{j}=\frac{1}{2}\overline{\gamma}^{j\ell}\partial_{0}\overline{\gamma}_{i\ell}=\frac{1}{2}\overline{\gamma}^{j\ell}\frac{D_{V}\overline{\gamma}_{i\ell}-V^{m}\partial_{m}\overline{\gamma}_{i\ell}}{V^{0}}=\frac{1}{2}\overline{\gamma}^{j\ell}\frac{\Vert V\Vert \partial_{t}g_{i\ell}-V^{m}\partial_{m}g_{i\ell}}{V^{0}} \quad \text{at} \ t=0.
\end{equation*} 
For the second term, by (\ref{orthu}) and (\ref{Gamma00mu1}),
\begin{equation*}
   \begin{aligned}
      \frac{1}{\eta^{2}}\partial_{t}\Vert V\Vert&=-\Gamma_{\mu 0}^{\mu}\Vert V\Vert=-\Gamma_{i0}^{i}\Vert V\Vert-g^{0i}\Vert V\Vert\partial_{t}g_{0i}\\
      &=-\Gamma_{i0}^{i}\Vert V\Vert-g^{0i}\partial_{t}V_{i}+\left(1+g^{00}\right)\partial_{t}\Vert V\Vert.
   \end{aligned}
\end{equation*}
Since $\frac{1}{\eta^{2}}-1-g^{00}>0$, we obtain 
\begin{equation*}
   \partial_{t}\Vert V\Vert=\frac{1}{\frac{1}{\eta^{2}}-1-g^{00}}\left(-\Gamma_{i0}^{i}\Vert V\Vert-g^{0i}\partial_{t}V_{i}\right).
\end{equation*}
Hence
\begin{equation*}
    \partial_{t}g_{0i}=\frac{\partial_{t}V_{i}}{\Vert V\Vert}+\frac{V_{i}}{\left(\frac{1}{\eta^{2}}-1-g^{00}\right)\Vert V\Vert^{2}}\left(-\Gamma_{i0}^{i}\Vert V\Vert-g^{0i}\partial_{t}V_{i}\right) \quad \text{at}\ t=0,
\end{equation*}
and by (\ref{Gamma00mu1}),
\begin{equation}
   \Gamma_{00}^{\mu}=g^{\mu i}\partial_{t}g_{0i} \quad \text{at}\ t=0.
\end{equation}
Then the initial data for the curvature can be determined in the usual way, provided that the data for $g, \partial_{t}g$, $\Gamma_{\alpha\beta}^{\mu}$ are known. For $R_{ij\ell m}$ and $R_{\ell Tij}$, by the Gauss and Codazzi equations,
\begin{equation}\label{Gauss-Codazzi equ}
   \begin{aligned}
      &R_{ij\ell m}|_{t=0}=\overline{R}_{ij\ell m}-k_{ij}k_{\ell m}+k_{i\ell}k_{jm},\\
      &R_{\ell Tij}|_{t=0}=\overline{\nabla}_{i}k_{j\ell}-\overline{\nabla}_{j}k_{i\ell}.
   \end{aligned}
\end{equation}
For the components $R_{iTjT}$, by (\ref{Ricc}),
\begin{equation}\label{remaining R}
   R_{iTjT}=\overline{R}_{ij}+k_{ij}\mathrm{tr}_{\overline{\gamma}}k-k_{i\ell}k_{j}^{\ell}-\left(GV_{i}V_{j}+\frac{1}{2}G\Vert V\Vert^{2}g_{ij}-pg_{ij}\right)\chi_{\Omega_{0}}\quad \text{at} \ t=0.
\end{equation}
Now we can fully determine the initial data for $g_{\alpha\beta}, \partial_{\mu}g_{\alpha\beta}, \Gamma_{\alpha\beta}^{\mu}, R_{\alpha\beta\mu\nu}$. Based on this information and a choice of initial frame $\left\{e_{I}\right\}$, the initial data of $\Gamma_{IJ}^{K}$ and $R_{IJKL}$ can be determined as well. Since $\nabla_{\partial_{t}}e_{I}=0$, it has
\begin{equation}\label{e'}
   \partial_{t}e_{I}^{\nu}=-\Gamma_{0\mu}^{\nu}e_{I}^{\mu}\quad \text{at}\ t=0
\end{equation}
and
\begin{equation}\label{GammaIJK}
   \Gamma_{IJ}^{K}=\epsilon_{K}e_{K}^{\alpha}e_{I}^{\mu}e_{J}^{\beta}g_{\alpha\nu}\Gamma_{\mu\beta}^{\nu}+\epsilon_{K}e_{K}^{\alpha}e_{I}^{\mu}g_{\alpha\beta}\partial_{\mu}e_{J}^{\beta} \quad \text{at}\ t=0.
\end{equation}

\subsection{Outline of the paper} In Section \ref{sec2}, we derive a priori estimates for the systems involving geometric quantities and fluid quantities $\Theta$ and $\Vert V\Vert^{2}$. In Section \ref{sec3}, we employ the Galerkin approximation to establish the iteration and prove its convergence. Based on these results, we show that the derived solutions satisfy the original equations, thus completing our main theorem \ref{main thm}.     	
\subsection*{Acknowledgment}
This work was supported by National Key R \& D Program of China 2021YFA1001700,  NSFC grants 12071360, 12221001, \& 12326344.
%%%%%%%%%%%%%%%%%
%%%%%%%%%%%%%%%%%
   \section{A Prior Estimates}\label{sec2}
   In this section, we consider a priori estimates for the system $(\ref{main structure equ})$. Let us start by defining energies as follows
   \begin{equation*}
       E_{k}(T):=\sup_{0\le t\le T}\left(\Vert \partial\partial_{t}^{k}\Theta\Vert_{L^{2}(\Omega_{t})}^{2}+\Vert \partial_{t}^{k+1}\Theta\Vert_{L^{2}(\partial\Omega_{t})}^{2}+\Vert \partial\partial_{t}^{k+1}\Vert V\Vert^{2}\Vert_{L^{2}(\Omega_{t})}^{2}+\Vert\partial_{t}^{k}R\Vert_{L^{2}(\Sigma_{t})}^{2}\right)+\Vert \partial\partial_{t}^{k+1}\Vert V\Vert^{2}\Vert_{L^{2}(\partial\Omega_{0}^{T})}^{2}
   \end{equation*}
   and 
   \begin{equation*}
        \mathcal{E}_{\ell}(T)=\sum_{k\le \ell}E_{k}(T).
   \end{equation*}
   And our main work in this section is to make the following estimates.
   \begin{proposition}\label{main prop}
      Suppose $(\Theta, R, \partial_{t}\Vert V\Vert^{2}, \Vert V\Vert^{2})$ is a solution to the system $(\ref{main structure equ})$ with
      \begin{equation*}
              \mathcal{E}_{\ell}(T)\le C_{1},
      \end{equation*}
      for some constant $C_{1}$ and $\ell$ sufficiently large. Let
      \begin{equation*}
          \mathscr{E}_{\ell}=\mathcal{E}_{\ell}(T)+\sup_{0\le t\le T}\sum_{2p+k\le \ell+2}\left(\Vert \partial^{p}\partial_{t}^{k}\Theta\Vert_{L^{2}(\Omega_{t})}^{2}+\Vert \partial^{p}\partial_{t}^{k+1}\Vert V\Vert^{2}\Vert_{L^{2}(\Omega_{t})}^{2}\right)+\sup_{0\le t\le T}\sum_{2p+k\le \ell+1}\left(\Vert \partial^{p}\partial_{t}^{k}R\Vert_{L^{2}(\Omega_{t})}^{2}+\Vert \partial^{p}\partial_{t}^{k}R\Vert_{L^{2}(\Omega_{t}^{c})}^{2}\right).
      \end{equation*}
      If $T>0$ is sufficiently small, depending on $C_{1}$ (independent of $C_{1}$), $\mathscr{E}_{\ell}(0)$, $\ell$, $c_0$, then
      \begin{equation*}
               \mathcal{E}_{\ell}(T)\le \mathcal{P}_{\ell}(\mathscr{E}_{\ell}(0))
      \end{equation*}
      for some polynomial function $\mathcal{P}_{\ell}$.
   \end{proposition}
   \subsection{Energy identities}\label{sec2.1}
   In this subsection we prove the basic energy estimates which can be used in the proof of Proposition $\ref{main prop}$. Let $Q=Q^{\mu}\nabla_{\mu}$ be any first order multiplier that is applied to the wave equation. The following direct calculation shows
   \begin{equation}\label{energy identity}
     \left(\Box u\right)\left(Qu\right)=\nabla_{\mu}\left(\left(Qu\right)\left(\nabla^{\mu}u\right)-\frac{1}{2}Q^{\mu}g^{\alpha\beta}\partial_{\alpha}u\partial_{\beta}u\right)+\frac{1}{2}\left(\nabla_{\mu}Q^{\mu}\right)g^{\alpha\beta}\partial_{\alpha}u\partial_{\beta}u-\left(\nabla^{\mu}Q^{\alpha}\right)\partial_{\mu}u\partial_{\alpha}u.
   \end{equation}
   
   Let us start with the main energy estimate for $\Theta^{I}$.
   \begin{lemma}\label{V energy lemma}
   	   Recall that $\Theta^{I}$ satisfies 
   	   \begin{equation}\label{energy equ1}
   	        \begin{cases}
   	            \left(\partial_{t}^{2}+\gamma D_{n}\right)\Theta^{I}=f  &\quad \text{on}\ \partial\Omega\\
   	            \Box \Theta^{I} =H &\quad \text{in}\ \Omega
   	        \end{cases}.
   	   \end{equation}
   	   Then 
   	   \begin{equation}\label{energy equ1.1}
   	      \begin{aligned}
   	          &\int_{\Omega_{T}}\left(-g^{0\nu}\partial_{t}\Theta^{I}\partial_{\nu}\Theta^{I}+\frac{1}{2}g^{\mu\nu}\partial_{\mu}\Theta^{I}\partial_{\nu}\Theta^{I}\right)\sqrt{\abs{g}}\mathrm{d}x+\frac{1}{2}\int_{\partial \Omega_{T}}\gamma^{-1}\left(\partial_{t}\Theta^{I}\right)^{2}\sqrt{\abs{g}}\mathrm{d}S\\
   	         &=\int_{\Omega_{0}}\left(-g^{0\nu}\partial_{t}\Theta^{I}\partial_{\nu}\Theta^{I}+\frac{1}{2}g^{\mu\nu}\partial_{\mu}\Theta^{I}\partial_{\nu}\Theta^{I}\right)\sqrt{\abs{g}}\mathrm{d}x+\frac{1}{2}\int_{\partial \Omega_{0}}\gamma^{-1}\left(\partial_{t}\Theta^{I}\right)^{2}\sqrt{\abs{g}}\mathrm{d}S\\
   	         &-\int_{0}^{T}\int_{\Omega_{t}}H\partial_{t}\Theta^{I}\sqrt{\abs{g}}\mathrm{d}x\mathrm{d}t+\int_{0}^{T}\int_{\partial\Omega}\gamma^{-1}f\partial_{t}\Theta^{I}\sqrt{\abs{g}}\mathrm{d}S\mathrm{d}t\\
   	         &+\frac{1}{2}\int_{0}^{T}\int_{\partial\Omega_{t}}\left(\partial_{t}\left(\gamma^{-1}\sqrt{\abs{g}}\right)\right)\left(\partial_{t}\Theta^{I}\right)^{2}\mathrm{d}S\mathrm{d}t+\frac{1}{2}\int_{0}^{T}\int_{\Omega_{t}}\left(\partial_{t}\left(\sqrt{\abs{g}}g^{\mu\nu}\right)\right)\partial_{\mu}\Theta^{I}\partial_{\nu}\Theta^{I}\mathrm{d}x\mathrm{d}t.
   	      \end{aligned}
   	   \end{equation}
   \end{lemma}
   \begin{proof}
   	 The lemma from multiplying the first and second equations in $(\ref{energy equ1})$ by $\gamma^{-1}\partial_{t}\Theta^{I}\sqrt{\abs{g}}$ and $\partial_{t}\Theta^{I}\sqrt{\abs{g}}$, respectively, and integrating.     
   \end{proof}
   The next energy estimate is used for $\partial_{t}\Vert V\Vert^{2}$.
   \begin{lemma}\label{V^2 energy lemma}
   Recall that $\partial_{t}\Vert V\Vert^{2}$ satisfies
   	   \begin{equation}\label{energy equ2}
   	      \begin{cases}
   	          \square\partial_{t}\Vert V\Vert^{2}=\frac{2}{G\Vert V\Vert}D_{V}^{3}G+F \quad &\text{in}\ \Omega\\
   	          \partial_{t}\Vert V\Vert^{2}\equiv 0\quad &\text{on}\ \partial\Omega
   	      \end{cases}.
   	   \end{equation}
   Then there is a future directed timelike vectorfield $Q=\partial_{t}-\alpha n$ such that for some constant $\alpha>0$, 
   	   \begin{equation}\label{energy equ2.1}
   	       \begin{aligned}
   	       &\int_{\Omega_{T}}\left(-Q\partial_{t}\Vert V\Vert^{2}g^{0\nu}\partial_{\nu}\partial_{t}\Vert V\Vert^{2}+\frac{Q^{0}}{2}g^{\mu\nu}\partial_{\mu}\partial_{t}\Vert V\Vert^{2}\partial_{\nu}\partial_{t}\Vert V\Vert^{2}\right)\sqrt{\abs{g}}\mathrm{d}x+\frac{\alpha}{2}\int_{0}^{T}\int_{\partial\Omega_{t}}\left(D_{n}\partial_{t}\Vert V\Vert^{2}\right)^{2}\sqrt{\abs{g}}\mathrm{d}S\mathrm{d}t\\
   	      &+\int_{\Omega_{T}}\frac{\Vert V\Vert^{2}G^{'}}{G}\left(\partial_{t}^{2}\Vert V\Vert^{2}\right)^{2}\sqrt{\abs{g}}\mathrm{d}x-\alpha\int_{\Omega_{T}}\frac{2\Vert V\Vert^{2}G^{'}}{G}\partial_{t}^{2}\Vert V\Vert^{2}D_{n}\partial_{t}\Vert V\Vert^{2}\sqrt{\abs{g}}\mathrm{d}x\\
   	      &=\int_{\Omega_{0}}\left(-Q\partial_{t}\Vert V\Vert^{2}g^{0\nu}\partial_{\nu}\partial_{t}\Vert V\Vert^{2}+\frac{Q^{0}}{2}g^{\mu\nu}\partial_{\mu}\partial_{t}\Vert V\Vert^{2}\partial_{\nu}\partial_{t}\Vert V\Vert^{2}\right)\sqrt{\abs{g}}\mathrm{d}x-\int_{0}^{T}FQ\partial_{t}\Vert V\Vert^{2}\sqrt{\abs{g}}\mathrm{d}x\mathrm{d}t\\
   	      &+\int_{\Omega_{0}}\frac{\Vert V\Vert^{2}G^{'}}{G}\left(\partial_{t}^{2}\Vert V\Vert^{2}\right)^{2}\sqrt{\abs{g}}\mathrm{d}x-\alpha\int_{\Omega_{0}}\frac{2\Vert V\Vert^{2}G^{'}}{G}\partial_{t}^{2}\Vert V\Vert^{2}D_{n}\partial_{t}\Vert V\Vert^{2}\sqrt{\abs{g}}\mathrm{d}x\\
   	      &-\int_{0}^{T}\int_{\Omega_{t}}\left(g^{\alpha\beta}\left(\partial_{\alpha}Q^{\mu}\right)\partial_{\mu}\partial_{t}\Vert V\Vert^{2}\partial_{\beta}\partial_{t}\Vert V\Vert^{2}-\frac{1}{2}\partial_{\mu}\left(\sqrt{\abs{g}}g^{\alpha\beta}Q^{\mu}\right)\partial_{\alpha}\partial_{t}\Vert V\Vert^{2}\partial_{\beta}\partial_{t}\Vert V\Vert^{2}\right)\mathrm{d}x\mathrm{d}t\\
   	      &-\int_{0}^{T}\int_{\Omega_{t}}\alpha D_{n}\left(\frac{\Vert V\Vert^{2}G^{'}}{G}\left(\partial_{t}^{2}\Vert V\Vert^{2}\right)^2\sqrt{\abs{g}}\right)\mathrm{d}x\mathrm{d}t.
   	      \end{aligned}
   	   \end{equation}
   \end{lemma}
   \begin{proof}
   	  Multiplying the interior equations by $Q\partial_{t}\Vert V\Vert^{2}\sqrt{\abs{g}}$ and using $(\ref{energy identity})$, we get
   	  \begin{equation}\label{energy identity2.2}
   	     \begin{aligned}
   	         \frac{2}{G\Vert V\Vert}D_{V}^{3}GQ\partial_{t}\Vert V\Vert^{2}\sqrt{\abs{g}}+FQ\partial_{t}\Vert V\Vert^{2}\sqrt{\abs{g}}&=\partial_{\alpha}\left(Q^{\mu}\partial_{\mu}\partial_{t}\Vert V\Vert^{2}\sqrt{\abs{g}}g^{\alpha\beta}\partial_{\beta}\partial_{t}\Vert V\Vert^{2}-\frac{1}{2}\sqrt{\abs{g}}g^{\mu\nu}Q^{\alpha}\partial_{\mu}\partial_{t}\Vert V\Vert^{2}\partial_{\nu}\partial_{t}\Vert V\Vert^{2}\right)\\
   	         &-g^{\alpha\beta}\left(\partial_{\alpha}Q^{\mu}\right)\partial_{\mu}\partial_{t}\Vert V\Vert^{2}\partial_{\beta}\partial_{t}\Vert V\Vert^{2}+\frac{1}{2}\partial_{\mu}\left(\sqrt{\abs{g}}g^{\alpha\beta}Q^{\mu}\right)\partial_{\alpha}\partial_{t}\Vert V\Vert^{2}\partial_{\beta}\partial_{t}\Vert V\Vert^{2}.
   	     \end{aligned}
   	  \end{equation}
   	 The first term on the left-hand side of identity $(\ref{energy identity2.2})$ can be written as
   	 \begin{equation*}
   	    \begin{aligned}
   	         \frac{2}{G\Vert V\Vert}D_{V}^{3}GQ\partial_{t}\Vert V\Vert^{2}\sqrt{\abs{g}}&=\partial_{t}\left(\frac{\Vert V\Vert^{2}G^{'}}{G}\left(\partial_{t}^{2}\Vert V\Vert^{2}\right)^{2}\sqrt{\abs{g}}\right)-\alpha\partial_{t}\left(\frac{2\Vert V\Vert^{2}G^{'}}{G}\partial_{t}^{2}\Vert V\Vert^{2}D_{n}\partial_{t}\Vert V\Vert^{2}\sqrt{\abs{g}}\right)\\
   	         &+\alpha D_{n}\left(\frac{\Vert V\Vert^{2}G^{'}}{G}\left(\partial_{t}^{2}\Vert V\Vert^{2}\right)^2\sqrt{\abs{g}}\right)+I_{1},
   	    \end{aligned}
   	 \end{equation*}
   	 where $I_{1}$ is a linear combination of $\partial_{t}^{2}\Vert V\Vert^{2}$, $D\partial_{t}\Vert V\Vert^{2}$ and $\partial_{t}\Vert V\Vert^{2}$ that can be included in $FQ\partial_{t}\Vert V\Vert^{2}\sqrt{\abs{g}}$. Using the boundary equation of $(\ref{energy equ2})$, we have 
   	 \begin{equation*}
   	      g^{\mu\nu}\partial_{\mu}\partial_{t}\Vert V\Vert^{2}\partial_{\nu}\partial_{t}\Vert V\Vert^{2}=\left(D_{n}\partial_{t}\Vert V\Vert^{2}\right)^{2}.
   	 \end{equation*}
   	 Thus, combining these identities and integrating $(\ref{energy identity2.2})$ over $\left[0,T\right]\times\Omega_{t}$ gives the desired result.
   \end{proof}
   Finally, we use the following lemma to estimate the curvature.
   \begin{lemma}\label{W energy lemma}
   	   Suppose $W$ satisfies 
   	   \begin{equation*}
   	          \mathcal{B}^{\mu}\partial_{\mu}W=\mathcal{K}.
   	   \end{equation*}
   	   Then
   	   \begin{equation*}
   	       \sup_{0\le t\le T}\Vert W(t)\Vert_{L^{2}(\mathbb{R}^{3})}\lesssim\Vert W(0)\Vert_{L^{2}(\mathbb{R}^{3})}+\int_{0}^{T}\left(\int_{\mathbb{R}^{3}}\abs{\mathcal{K}}^{2}\mathrm{d}x+\Vert\partial\mathcal{B}\Vert_{L^{\infty}}\right)\mathrm{d}t.
   	   \end{equation*}
   \end{lemma}
   \begin{proof}
   	  Given that $W$ satisfies a first-order hyperbolic system, we shall proceed with the standard approach to derive the energy estimate. Specifically, the desired estimate from multiplying the equation by $W$ and integrating.
   \end{proof}
  \subsection{Higher order equations}\label{sec2.2}
  Through direct calculation, we establish commutator identities with the main linear operators, which hold for any scalar function $u$:
  \begin{equation}\label{commute D}
   \left[\partial_{t}, D_{I}\right]u=-\left(D_{I}\hat{\Theta}^{J}+\hat{\Theta}^{K}\Gamma_{IK}^{J}\right)D_{J}u.
  \end{equation}
  \begin{equation}\label{commute box}
  \begin{aligned}
  \left[\partial_{t}, \Box \right]u&=-\sum_{I}\epsilon_{I}\left(D_{I}\hat{\Theta}^{J}+\hat{\Theta}^{K}\Gamma_{IK}^{J}\right)\left(D_{I}D_{J}u+D_{J}D_{I}u\right)-\Box\hat{\Theta}^{J}D_{J}u+\sum_{J}\epsilon_{J}\left(\Gamma_{IJ}^{K}-\Gamma_{JI}^{K}\right)D_{J}\hat{\Theta}^{I}D_{K}u\\
  &-\sum_{I}\epsilon_{I}\hat{\Theta}^{K}D_{I}\Gamma_{IK}^{J}D_{J}u+\sum_{I}\Gamma_{II}^{K}\Gamma_{KM}^{J}\hat{\Theta}^{M}D_{J}u-\hat{\Theta}^{I}\mathrm{Ric}_{\ I}^{K}D_{K}u+\sum_{J}\epsilon_{J}\hat{\Theta}^{I}\Gamma_{MJ}^{K}\Gamma_{JI}^{M}D_{K}u.
  \end{aligned}
  \end{equation}
  \begin{equation}\label{commute n}
     \begin{aligned}
        \left[\partial_{t}, \partial_{t}^{2}+\gamma D_{n}\right]u&=\frac{\epsilon_{I}}{2\Vert V\Vert^{2}}\left(D_{I}\partial_{t}\Vert V\Vert^{2}\right)D_{I}u-\frac{\partial_{t}\Vert V\Vert^{2}}{\Vert V\Vert^{2}}\gamma D_{n}u-\frac{\epsilon_{I}}{2\Vert V\Vert^{2}}\left(D_{I}\hat{\Theta}^{J}+\Gamma_{IK}^{J}\hat{\Theta}^{K}\right)\left(D_{J}\Vert V\Vert^{2}\right)D_{I}u\\
        &-\frac{\epsilon_{I}}{2\Vert V\Vert^{2}}\left(D_{I}\hat{\Theta}^{J}+\Gamma_{IK}^{J}\hat{\Theta}^{K}\right)\left(D_{I}\Vert V\Vert^{2}\right)D_{J}u.
     \end{aligned}
  \end{equation}
  Using these identities we can calculate the higher order equations of $(\ref{main structure equ})$, which we will state in the following lemmas. 
   \begin{lemma}\label{V higher order lemma}
   	  For any $k>0$, $\partial_{t}^{k}\Theta^{I}$ satisfies
   	  \begin{equation*}
   	      \Box \partial_{t}^{k}\Theta^{I}=H_{k} \quad \text{in}\ \Omega,
   	  \end{equation*}
   	  where $H_{k}$ is a linear combination of $\Gamma$, $D\Gamma$, $\partial_{t}^{j_{1}}\Theta$, $D\partial_{t}^{j_{2}}\Theta$,  $D\partial_{t}^{j_{3}+1}\Vert V\Vert^{2}$, $D\partial_{t}^{j_{4}}\Vert V\Vert^{2}$, $\partial_{t}^{j_{5}}\Vert V\Vert^{2}$ $\partial_{t}^{j_{6}}R$, $D\partial_{t}^{j_{6}}\Vert V\Vert^{2}$ $D\partial_{t}^{j_{8}}R$ and $D^{2}\partial_{t}^{j_{9}}\Theta$, with $\sum j_{i}\le k$, and $j_{1}, j_{2}, \dots, j_{5}\le k$, $j_{6}, \dots, j_{9}\le k-1.$
   \end{lemma}	
   \begin{proof}
   Let us complete the proof by induction on the order $k$ of $\partial_{t}^{k}$. For $k=1$, since
   \begin{equation*}
         \Box \partial_{t}\Theta^{I}=\partial_{t}\Box\Theta^{I}-\left[\partial_{t}, \Box\right]\Theta^{I},
   \end{equation*}
   using $(\ref{2V interior frame equ})$ and $(\ref{commute box})$, we need to concentrate on $\epsilon_{I}\epsilon_{K}\partial_{t}\nabla_{K}\omega_{KI}$, which can be written as 
   \begin{equation*}
       \epsilon_{I}\epsilon_{K}\partial_{t}\nabla_{K}\omega_{KI}=\frac{\epsilon_{I}\epsilon_{K}}{\Vert V\Vert}\left[D_{V}, \nabla_{K}\right]\omega_{KI}+\frac{\epsilon_{I}\epsilon_{K}}{\Vert V\Vert}\nabla_{K}D_{V}\omega_{KI}.
   \end{equation*}  
   Then using $(\ref{transport w})$ implies that $\epsilon_{I}\epsilon_{K}\partial_{t}\nabla_{K}\omega_{KI}$ is a linear combination of                                                                                                                    $R$, $D^{2}\Theta$ and $D\Theta$. Thus $H_{1}$ is a linear combination of $\Gamma$, $D\Gamma$, $D^{2}\Theta$, $D\partial_{t}\Theta$, $\partial_{t}\Theta$, $D\partial_{t}^{2}\Vert V\Vert^{2}$,   $D\partial_{t}\Vert V\Vert^{2}$, $D\Vert V\Vert^{2}$, $\partial_{t}\Vert V\Vert^{2}$, $R$ and $DR$. Suppose it holds for $k=j$, and let us prove it holds for $k=j+1$. By $(\ref{commute box})$, $\left[\partial_{t}, \Box\right]\partial_{t}^{j}\Theta^{I}$ has the right form. Similarly, $\partial_{t}$ applied to $H_{j}$ has the desired form.
   \end{proof}

	The equation for $\partial_{t}^{k}\Theta^{I}$ on $\partial\Omega$ is derived in the next lemma.
	\begin{lemma}\label{V boundary higher order lemma}
		For any $k>0$, $\partial_{t}^{k}\Theta^{I}$ satisfies
		\begin{equation*}
		\left(\partial_{t}^{2}+\gamma D_{n}\right)\partial_{t}^{k}\Theta^{I}=F_{k} \quad \text{on}\ \partial\Omega,
		\end{equation*}
		where $F_{k}$ is a linear combination of contractions of $\Gamma$, $D\partial_{t}^{j_{1}+1}\Vert V\Vert^{2}$, $D\partial_{t}^{j_{2}}\Vert V\Vert^{2}$,  $\partial_{t}^{j_{3}+1}\Theta$, $\partial_{t}^{j_{4}}\Theta$, $D\partial_{t}^{j_{5}}\Theta$ and $\partial_{t}^{j_{6}}R$, with $\sum j_{i}\le k$, and $j_{1}, j_{2}, \dots , j_{4}\le k$, $j_{5},j_{6}\le k-1.$
	\end{lemma}
   \begin{proof}
       We also proceed inductively. For $k=1$, 	
      \begin{equation*}
       \left(\partial_{t}^{2}+\gamma D_{n}\right)\partial_{t}\Theta^{I}=\partial_{t}\left[\left(\partial_{t}^{2}+\gamma D_{n}\right)\Theta^{I}\right]-\left[\partial_{t}, \partial_{t}^{2}+\gamma D_{n}\right]\Theta^{I}.
      \end{equation*}
      The second term on the right-hand above has the right form using $(\ref{commute n})$. For the first term, $\partial_{t}$ applied to $\frac{\epsilon_{I}}{2\Vert V\Vert^{2}}D_{I}D_{V}\Vert V\Vert^{2}$ can be express as a linear combination of $D\partial_{t}^{2}\Vert V\Vert^{2}$, $D\Vert V\Vert^{2}$ and $D\partial_{t}\Vert V\Vert^{2}$. And $\partial_{t}$ applied to $\frac{\epsilon_{I}\epsilon_{J}}{2}\omega_{IJ}D_{J}\Vert V\Vert^{2}$ consists of $\partial_{t}D_{J}\Vert V\Vert^{2}$ and $D\Theta$ using $(\ref{transport w})$. Then $\partial_{t}$ applied to other terms of the first term are of the desired form, so it holds for $k=1$. Let us assume it holds for $k=j$ and prove it for $k=j+1$. Since we have 
      \begin{equation*}
         \left(\partial_{t}^{2}+\gamma D_{n}\right)\partial_{t}^{j+1}\Theta^{I}=\partial_{t}\left[\left(\partial_{t}^{2}+\gamma D_{n}\right)\partial_{t}^{j}\Theta^{I}\right]-\left[\partial_{t}, \partial_{t}^{2}+\gamma D_{n}\right]\partial_{t}^{j}\Theta^{I},
      \end{equation*}
     it has the desired form in view of $(\ref{commute n})$ and $\partial_{t}$ applied to $F_{j}$.
   \end{proof}

	The wave equation for $\partial_{t}^{k+1}\Vert V\Vert^{2}$ is derived in the next lemma.
	\begin{lemma}\label{V^2 higher order lemma}
	    For any $k>0$, $\partial_{t}^{k+1}\Vert V\Vert^{2}$ satisfies
	    \begin{equation*}
	        \Box \partial_{t}^{k+1}\Vert V\Vert^{2}=\frac{2G^{'}\Vert V\Vert^{2}}{G}\partial_{t}^{k+3}\Vert V\Vert^{2}+H_{k}, \quad \text{in}\ \Omega,
	    \end{equation*}
	    where $H_{k}$ is a linear combination of $\Gamma$, $D\Gamma$, $\partial_{t}^{j_{1}+2}\Vert V\Vert^{2}$, $\partial_{t}^{j_{2}+1}\Vert V\Vert^{2}$, $\partial_{t}^{j_{3}}\Vert V\Vert^{2}$, $D\partial_{t}^{j_{4}}\Vert V\Vert^{2}$, $D\partial_{t}^{j_{5}+1}\Vert V\Vert^{2}$, $D^{2}\partial_{t}^{j_{6}}\Vert V\Vert^{2}$, $\partial_{t}^{j_{7}}\Theta$, $D\partial_{t}^{j_{8}}\Theta$, $\partial_{t}^{j_{9}}R$, $D\partial_{t}^{j_{10}}\Theta$, $D\partial_{t}^{j_{11}}R$ and $D^{2}\partial_{t}^{j_{12}}\Theta$ with $\sum j_{i}\le k$, and $j_{1},\dots,j_{9}\le k$, and $j_{10}, j_{11}, j_{12}\le k-1$.
	\end{lemma}
    \begin{proof}
    	We also proceed inductively. For $k=1$, by direct calculation gives
    	\begin{equation*}
    	      \Box \partial_{t}^{2}\Vert V\Vert^{2}=\partial_{t}\Box \partial_{t}\Vert V\Vert^{2}-\left[\partial_{t}, \Box\right]\partial_{t}\Vert V\Vert^{2}.
    	\end{equation*}
    	The last term on the right-hand side above follows from applying $u=\partial_{t}\Vert V\Vert^{2}$ to $(\ref{commute box})$. Using $(\ref{2DV^2 frame equ})$, $\partial_{t}\left(\frac{1}{G\Vert V\Vert}D_{V}^{3}G\right)$ can be manipulated into
    	\begin{equation*}
    	   \begin{aligned}
    	       \partial_{t}\left(\frac{2}{G\Vert V\Vert}D_{V}^{3}G\right)&=\frac{2G^{'}\Vert V\Vert^{2}}{G}\partial_{t}^{4}\Vert V\Vert^{2}+\partial_{t}\left(\frac{2G^{'}\Vert V\Vert^{2}}{G}\right)\partial_{t}^{3}\Vert V\Vert^{2}+\partial_{t}\left(\frac{2}{G}\right)\partial_{t}\left[\Vert V\Vert\partial_{t}\left(G^{'}\Vert V\Vert\right)\partial_{t}\Vert V\Vert^{2}\right]\\
    	       &+\frac{2}{G}\partial_{t}^{2}\left[\Vert V\Vert\partial_{t}\left(G^{'}\Vert V\Vert\right)\partial_{t}\Vert V\Vert^{2}\right]+\partial_{t}\left[\frac{2}{G}\partial_{t}\left(G\Vert V\Vert^{2}\right)\partial_{t}^{2}\Vert V\Vert^{2}\right].
    	   \end{aligned}
    	\end{equation*}
     It is obvious that the last four terms of the right-hand side above have the desired form. Next we will consider the term $\partial_{t}\left(\frac{2\epsilon_{K}}{\Vert V\Vert}\Theta^{J}\nabla_{V}\nabla_{K}\omega_{KJ}\right)$. Since
     \begin{equation}\label{main skill for higher order DV2}
       \begin{aligned}
         \epsilon_{K}\Theta^{J}\nabla_{V}\nabla_{K}\omega_{KJ}&=\epsilon_{K}\Theta^{J}\nabla_{K}\nabla_{V}\omega_{KJ}+\epsilon_{K}\Theta^{J}\left[\nabla_{V}, \nabla_{K}\right]\omega_{KJ}\\
         &=-\epsilon_{K}\Theta^{J}\nabla_{K}\left[\left(\nabla_{K}\Theta^{L}\right)\omega_{LJ}+\left(\nabla_{J}\Theta^{L}\right)\omega_{KL}\right]+\epsilon_{K}\Theta^{J}\left[\nabla_{V}, \nabla_{K}\right]\omega_{KJ}\\
         &=-\Theta^{J}\Box\Theta^{L}\omega_{LJ}-\epsilon_{K}\Theta^{J}\left(\nabla_{K}\omega_{LJ}\right)\left(\nabla_{K}\Theta^{L}\right)-\epsilon_{K}\nabla_{K}D_{V}\Theta^{L}\omega_{KL}\\
         &+\epsilon_{K}\left(\nabla_{K}\Theta^{J}\right)\left(\nabla_{J}\Theta^{L}\right)\omega_{KL}-\epsilon_{K}\Theta^{J}\nabla_{K}\omega_{KL}\left(\nabla_{J}\Theta^{L}\right)+\epsilon_{K}\Theta^{J}\left[\nabla_{V}, \nabla_{K}\right]\omega_{KJ}.
       \end{aligned}
     \end{equation} 
     By $(\ref{alongu})$ and $(\ref{transport w})$, $\partial_{t}$ applied to the above identity has the right form. Therefore, the conclusion holds for $k=1$. Now suppose that the result holds for $k=j$, and let us prove it for $k=j+1$. By $(\ref{commute box})$, the commutator $\left[\partial_{t}, \Box \right]\partial_{t}^{j}\Vert V\Vert^{2}$ is of the desired form. Similarly, $\partial_{t}$ applied to $H_{j}$ and $\frac{2G^{'}\Vert V\Vert^{2}}{G}\partial_{t}^{k+3}\Vert V\Vert^{2}$ also has the right form.
    \end{proof}

     Then we record the equation for $W$ by using $(\ref{I component1})$, $(\ref{I component2})$, $(\ref{J component})$ and $(\ref{hyperbolic system})$.
     \begin{lemma}\label{W higher order lemma}
     	 $\partial_{t}^{k}W$ satisfies
     	\begin{equation*}
     	   \mathcal{B}^{\mu}\partial_{\mu}\partial_{t}^{k}W=F_{k},
     	\end{equation*}
     	where $F_{k}$ is a linear combination of contractions of $\Gamma$, $D\Gamma$, $\mathrm{e}$, $D\partial_{t}^{j_{1}}\Vert V\Vert^{2}$, $\partial_{t}^{j_{2}}\Theta$, $D\partial_{t}^{j_{3}}\Theta$, $\partial_{t}^{j_{4}}R$, $D^{2}\partial_{t}^{j_{5}}\Vert V\Vert^{2}$, $D\partial_{t}^{j_{6}}W$, with $\sum j_{i}\le k$, and $j_{1},\dots,j_{5}\le k$, and $j_{6}\le k-1$.
     \end{lemma}
     \begin{proof}
     	We begin by commuting the time derivative operator $\partial_{t}$ with the equation $\mathcal{B}^{\mu}\partial_{\mu}W = \mathcal{K}$ in (\ref{hyperbolic system}). Since $\mathcal{B}^{\mu}$ involves the frame components $e_{I}^{\mu}$ and the constant coefficient matrices $\mathcal{A}^{I}$, the term $\partial_{t}\mathcal{B}^{\mu}$ takes the desired structure by applying the transport equation (\ref{transport e}) for $e_{I}$. On the right-hand side, we denote $\partial_{t}\mathcal{K}$ as $F_{1}$, which similarly acquires the desired form by using (\ref{I component1}), (\ref{I component2}), (\ref{J component}) and the transport equation (\ref{transport Ch}) for $\Gamma$. By proceeding inductively, we can extend this argument to establish the desired result for higher-order derivatives.
     \end{proof}
     
     Finally, the following lemma can be used to control $\partial_{t}^{k}\Theta$ on the boundary. 
     \begin{lemma}\label{Theta boundary control}
     	  For any $k>0$, $\partial_{t}^{k}\Theta$ satisfies
     	  \begin{equation*}
     	        \Box \partial_{t}^{k}\Theta= H, \quad \text{in}\ \Omega.
     	  \end{equation*}
     	Here $H$ is as defined in Lemma $\ref{V higher order lemma}$. Then there is a future directed timelike vectorfield $Q=\partial_{t}+\alpha n,\ \alpha>0$, such that
     	  \begin{equation*}
     	      \begin{aligned}
     	          &\int_{\Omega_{T}}\left(-Q\partial_{t}^{k}\Theta g^{0\alpha}\partial_{\alpha}\partial_{t}^{k}\Theta+\frac{Q^{0}}{2}g^{\alpha\beta}\partial_{\alpha}\partial_{t}^{k}\Theta\partial_{\beta}\partial_{t}^{k}\Theta\right)\sqrt{\abs{g}}\mathrm{d}x-\int_{0}^{T}\int_{\partial\Omega_{t}}\left(Q\partial_{t}^{k}\Theta D_{n}\partial_{t}^{k}\Theta-\frac{\alpha}{2}g^{\alpha\beta}\partial_{\alpha}\partial_{t}^{k}\Theta\partial_{\beta}\partial_{t}^{k}\Theta\right)\sqrt{\abs{g}}\mathrm{d}S\mathrm{d}t\\
     	          &=\int_{\Omega_{0}}\left(-Q\partial_{t}^{k}\Theta g^{0\alpha}\partial_{\alpha}\partial_{t}^{k}\Theta+\frac{Q^{0}}{2}g^{\alpha\beta}\partial_{\alpha}\partial_{t}^{k}\Theta\partial_{\beta}\partial_{t}^{k}\Theta\right)\sqrt{\abs{g}}\mathrm{d}x-\int_{0}^{T}\int_{\Omega_{t}}H\partial_{t}^{k}\Theta\sqrt{\abs{g}}\mathrm{d}x\mathrm{d}t\\
     	          &-\int_{0}^{T}\int_{\Omega_{t}}\left(g^{\alpha\beta}\left(\partial_{\alpha}Q^{\mu}\right)\partial_{\mu}\partial_{t}^{k}\Theta\partial_{\beta}\partial_{t}^{k}\Theta-\frac{1}{2}\partial_{\mu}\left(\sqrt{\abs{g}}g^{\alpha\beta}Q^{\mu}\right)\partial_{\alpha}\partial_{t}^{k}\Theta\partial_{\beta}\partial_{t}^{k}\Theta\right)\mathrm{d}x\mathrm{d}t.
     	      \end{aligned}
     	  \end{equation*} 
     \end{lemma}
   \begin{proof}
   	 This follow from integrating the identity $(\ref{energy identity})$, where $u=\partial_{t}^{k}\Theta$.
   \end{proof}
   
   \subsection{Elliptic estimates}\label{sec2.3}
    To prove Proposition \ref{main prop}, we first establish the following elliptic estimates. These estimates are crucial for obtaining both the $L^{\infty}$ bounds for the lower-order terms and the $L^{2}$ control of the derivatives for some next to top order terms. 
   \begin{proposition}\label{elliptic prop}
   	   Under the assumptions of Proposition $\ref{main prop}$, for any $t\in [0,T]$
   	   \begin{equation*}
   	         \sum_{2p+k\le \ell+2}\left(\Vert \partial^{p}\partial_{t}^{k}\Theta\Vert_{L^{2}\left(\Omega_{t}\right)}^{2}+\Vert \partial^{p}\partial_{t}^{k+1}\Vert V\Vert^{2}\Vert_{L^{2}\left(\Omega_{t}\right)}^{2}\right)\lesssim \mathcal{E}_{\ell}(t).
   	   \end{equation*} 
   	   Furthermore, 
   	   \begin{equation*}
   	        \Vert \partial\partial_{t}^{k}R\Vert_{L^{2}\left(\Omega_{t}\right)}^{2}+\Vert \partial\partial_{t}^{k}R\Vert_{L^{2}\left(\Omega_{t}^{c}\right)}^{2}\lesssim \mathcal{E}_{\ell}(t), \quad k\le \ell-1,
   	   \end{equation*}
   	   and 
   	   \begin{equation*}
   	      \sum_{2p+k\le \ell}\left(\Vert \partial^{p}\partial_{t}^{k}R\Vert_{L^{2}\left(\Omega_{t}\right)}^{2}+\Vert \partial^{p}\partial_{t}^{k}R\Vert_{L^{2}\left(\Omega_{t}^{c}\right)}^{2}\right)\lesssim \mathcal{E}_{\ell}(t). 
   	   \end{equation*}
   \end{proposition}
   
   We first provide a crucial estimate for $\partial^{p}\partial_{t}^{k+1}\Vert V\Vert^{2}$ and $\partial^{p}\partial_{t}^{k}\Vert V\Vert^{2}\Theta$.
   \begin{lemma}\label{M+1 control}
   	Under the assumotion of Proposition $\ref{main prop}$, with the implicit constant is independent of $C_{1}$, 
   	\begin{equation*}
   	\Vert\partial^{p}\partial_{t}^{k}\Theta\Vert_{L^{2}\left(\Omega_{t}\right)}+ \Vert\partial^{p}\partial_{t}^{k+1}\Vert V\Vert^{2}\Vert_{L^{2}\left(\Omega_{t}\right)}\lesssim 1, \quad \forall 2p+k\le M+1, \quad t\in\left[0,T\right],
   	\end{equation*} 
   	provided $T>0$ is sufficently small.
   \end{lemma}
   \begin{proof}
   	Let $\xi$ be the Lagrangian parameterization, that is $\partial_{\tau}\xi(\tau,y)=\left(\frac{V}{V^{0}}\right)\left(\xi(\tau,y)\right).$ For any function $\Theta$,
   	\begin{equation*}
   	   \frac{\mathrm{d}\Theta(\xi(\tau,y))}{\mathrm{d}t}=\frac{D_{V}\Theta}{V^{0}}(\xi).
   	\end{equation*}
   	Multiplying both sides of the equation by $\Theta$ and applying both Cauchy-Schwarz and Gronwall's inequality, we obtain 
   	\begin{equation*}
   	   \Vert \Theta\Vert_{L^{2}(\Omega_{t})}\lesssim \Vert \Theta\Vert_{L^{2}(\Omega_{0})}\mathrm{exp}\int_{0}^{t}\Vert D_{V}\Theta\Vert_{L^{2}(\Omega_{\tau})}\mathrm{d}\tau.
   	\end{equation*}
   	We apply to this estimate to $\Theta=\partial^{p}\partial_{t}^{k}\Theta$ and $\Theta=\partial^{p}\partial_{t}^{k+1}\Vert V\Vert^{2}$. Then, according to Proposition \ref{elliptic prop}, as long as $2p+k\le M+1$, $\Vert D_{V}\Theta\Vert_{L^{2}(\Omega_{\tau})}\lesssim_{C} 1$. Therefore, we can obtain the desired result.
   \end{proof}

   We continue with the proof of elliptic estimates for $\partial\partial_{t}^{k}R$. Using the equation $(\ref{contract B equ})$ and $\ref{B equ}$, we can express the Maxwell system in the form of 
   \begin{equation*}
         \nabla^{\mu}F_{\mu\nu}=J_{\nu}, \quad \nabla_{[\mu}F_{\nu\lambda]}=I_{\mu\nu\lambda}.
   \end{equation*}
   In the coordinates $(\partial_{t}, \partial_{i})$, $i= 1, 2, 3$, we denote $\barE^{AB}\equiv \barE$ and $\barH^{AB}\equiv \barH$, where
   \begin{equation*}
       \barE_{i}:=F_{it}, \quad \barH^{k}:=-\frac{1}{2}\epsilon^{ijk}F_{ij}.
   \end{equation*}
   We transition from the spacetime connection $\nabla$ to the induced connection $\overline{\nabla}$ on the hypersurface $\Sigma_{t}$. Note that the divergence-curl systems for $\barE$ and $\barH$ exhibit the following structure. On the one hand, the divergence equation for $\barE$ and $\barH$ is 
   \begin{equation}\label{div E}
       g^{ij}\overline{\nabla}_j\barE_{i}=\rho_{\barE}
   \end{equation}
   and 
   \begin{equation}\label{div H}
      \overline{\nabla}_{i}\barH^{i}=\rho_{\barH},
   \end{equation}
   respectively. Here $\rho_{\barE}$ depends on $e_{I}^{\mu}$, $De_{I}^{\mu}$, $\Gamma_{IJ}^{K}$, $\hat{\Theta}$, $D\hat{\Theta}$, $F_{IJ}$, $\partial_{t}F_{IJ}$ and $J_{t}$ and $\rho_{\barH}$ depends on $e_{I}^{\mu}$, $De_{I}^{\mu}$, $\Gamma_{IJ}^{K}$, $\hat{\Theta}$, $D\hat{\Theta}$, $F_{IJ}$ and $\partial_{t}F_{IJ}$. On the other hand, the curl equations for $\barE$ and $\barH$, respectively, can be expressed as 
   \begin{equation}\label{curl E}
      g^{jk}\epsilon_{ji\ell}\overline{\nabla}_{k}\barH^{\ell}+g^{tk}\overline{\nabla}_{k}\barE_{i}=\sigma_{i}^{\barH}
   \end{equation}
   and 
   \begin{equation}\label{curl H}
       \overline{\nabla}_{i}\barE_{j}-\overline{\nabla}_{j}\barE_{i}=\sigma_{ij}^{\barE}.
   \end{equation}
   Here $\sigma_{i}^{\barH}$ and $\sigma_{ij}^{\barE}$ both depend on $e_{I}^{\mu}$, $De_{I}^{\mu}$, $\Gamma_{IJ}^{K}$, $\hat{\Theta}$, $D\hat{\Theta}$, $F_{IJ}$ and $\partial_{t}F_{IJ}$. We check that $g^{ij}$ is positively definite. Using equations $(\ref{div E})$, $(\ref{div H})$, $\ref{curl E}$ and $\ref{curl H}$, the elliptic estimates for the curvature are recorded in the following lemmas. 
   \begin{lemma}\label{elliptic estimate for R1}
   	   Let $\Theta$ be a $1$-from satisfying
   	   \begin{equation*}
   	       \overline{\nabla}_{i}\Theta_{j}-\overline{\nabla}_{j}\Theta_{i}=\sigma_{ij},\quad g^{ij}\overline{\nabla}_{i}\Theta_{j}=\rho \quad \text{on}\ \Sigma_{t}.
   	   \end{equation*}
   	   Then
   	   \begin{equation*}
   	       \Vert \overline{\nabla}\Theta\Vert_{L^{2}(\Sigma_{t})}\lesssim \Vert \Theta\Vert_{L^{2}(\Sigma_{t})}+\Vert \rho\Vert_{L^{2}(\Sigma_{t})}+\Vert \sigma\Vert_{L^{2}(\Sigma_{t})},
   	   \end{equation*}
   	   where the $L^{2}$-norm is with respect to the tensor $g^{ij}$ and the implicit constant depends on the maximum norm of $e_{I}^{\mu}$, $De_{I}^{\mu}$ and $R$.
   \end{lemma}
   \begin{proof} Since
   	\begin{equation*}
   	   \begin{aligned}
   	       \vert\overline{\nabla}\Theta\vert^{2}&=\overline{\nabla}_{a}\left(g^{a\mu}g^{b\nu}\Theta_{b}\overline{\nabla}_{\nu}\Theta_{\mu}\right)-\overline{\nabla}_{\nu}\left(g^{a\mu}g^{b\nu}\Theta_{b}\overline{\nabla}_{a}\Theta_{\mu}\right)+\overline{\nabla}_{\nu}\left(g^{a\mu}g^{b\nu}\right)\Theta_{b}\overline{\nabla}_{a}\Theta_{\mu}\\
   	       &+\rho^{2}-g^{a\mu}g^{b\nu}\Theta_{b}\overline{R}_{a\nu\mu\ }^{\ \ \ \ c}\Theta_{c}-\overline{\nabla}_{a}\left(g^{a\mu}g^{b\nu}\right)\Theta_{b}\overline{\nabla}_{\nu}\Theta_{\mu}+g^{a\mu}g^{b\nu}\overline{\nabla}_{a}\Theta_{b}\sigma_{\mu\nu},
   	   \end{aligned}
   	\end{equation*}
   	we can obtain the desired estimate by integrating the above identity on $\Sigma_{t}$ and applying Cauchy-Schwarz inequality.
   \end{proof}

   Similarly, for a vectorfield $X$ on $\Sigma_{t}$, we can derive the estimates as follows.
   \begin{lemma}\label{elliptic estimate for R2}
   	  Let $X$ be vectorfield satisfying
   	  \begin{equation*}
   	  \overline{\nabla}_{i}X^{i}=\rho,\quad g^{jk}\epsilon_{ij\ell}\overline{\nabla}_{k}\Theta_{\ell}=\sigma_{i} \quad \text{on} \ \Sigma_{t}.
   	  \end{equation*}
   	  And suppose that the $L^{2}$-norm is also with respect to the tensor $g^{ij}$, that is,  
   	  \begin{equation*}
   	  \abs{\overline{\nabla}X}^{2}:=g^{ab}g^{ij}\overline{\nabla}_{a}X_{i}\overline{\nabla}_{b}X_{j}, \quad \abs{X}^{2}:=g^{ij}X_{i}X_{j}, \quad \abs{\sigma}^{2}:= g_{ab}g_{ij}\sigma^{ai}\sigma^{bj}.
   	  \end{equation*}
   	  Then
   	  \begin{equation*}
   	  \Vert \overline{\nabla}X\Vert_{L^{2}(\Sigma_{t})}\lesssim \Vert X\Vert_{L^{2}(\Sigma_{t})}+\Vert \rho\Vert_{L^{2}(\Sigma_{t})}+\Vert \sigma\Vert_{L^{2}(\Sigma_{t})},
   	  \end{equation*}
   	  where the implicit constant also depends on the maximum norm of $e_{I}^{\mu}$, $De_{I}^{\mu}$ and $R$. 
   \end{lemma}
   \begin{proof}
   	  By direct calculation, we arrive at
   	  \begin{equation*}
   	     \begin{aligned}
   	        \vert\overline{\nabla}X\vert^{2}&=\nabla_{a}\left(g^{a\mu}g^{b\nu}X_{b}\overline{\nabla}_{\nu}X_{\mu}\right)-\overline{\nabla}_{\nu}\left(g^{a\mu}g^{b\nu}X_{b}\overline{\nabla}_{a}X_{\mu}\right)+\overline{\nabla}_{\nu}\left(g^{a\mu}g^{b\nu}\right)X_{b}\overline{\nabla}_{a}X_{\mu}\\
   	        &+\rho^{2}-g^{a\mu}g^{b\nu}X_{b}\overline{R}_{a\nu\mu\ }^{\ \ \ \ c}X_{c}-\overline{\nabla}_{a}\left(g^{a\mu}g^{b\nu}\right)X_{b}\overline{\nabla}_{\nu}X_{\mu}+\overline{\nabla}_{a}X_{b}\sigma^{'ab},
   	     \end{aligned}
   	  \end{equation*}
   	  where $\sigma^{'ab}=g^{ak}g^{bm}\overline{\nabla}_{k}X_{m}-g^{bk}g^{am}\overline{\nabla}_{k}X_{m}$. Then, similar to the proof of lemma \ref{elliptic estimate for R1}, the desired estimate is obtained. 
   \end{proof}

    Next, we turn to prove elliptic estimates for the fluid. To do this, one key is to decompose the scalar wave operator into an elliptic part and a part with $\partial_{t}$ derivatives. Recall that $g^{ij}$ is positive definite, which is written as 
    \begin{equation*}
          g^{ij}=m^{IJ}e_{I}^{i}e_{J}^{j}=\left(\delta^{\tilde{I}\tilde{J}}-\frac{\hat{\Theta}^{\tilde{I}}}{\hat{\Theta}^{0}}\frac{\hat{\Theta}^{\tilde{J}}}{\hat{\Theta}^{0}}\right)e_{\tilde{I}}^{i}e_{\tilde{J}}^{j}.
    \end{equation*}
   At this point we have the following results without proof.
   \begin{lemma}\label{decomposition u}
   	   Let $\barA:=g^{ij}\partial_{ij}^{2}$ be the second order elliptic operator. Then for any function $u$,
   	   \begin{equation*}
   	       \begin{aligned}
   	           \barA u&=\Box u+\left(\left(\frac{1-\hat{\Theta}^{\tilde{I}}e_{\tilde{I}}^{0}}{\hat{\Theta}^{0}}\right)^{2}-\sum_{\tilde{I}}\left(e_{\tilde{I}}^{0}\right)^{2}\right)\partial_{t}^{2}u-2\left(e_{\tilde{I}}^{i}e_{\tilde{I}}^{0}+\frac{\left(1-\hat{\Theta}^{\tilde{I}}e_{\tilde{I}}^{0}\right)\hat{\Theta}^{\tilde{J}}e_{\tilde{J}}^{i}}{\left(\hat{\Theta}^{0}\right)^{2}}\right)\partial_{i}\partial_{t}u\\
   	           &-\epsilon_{I}\left(D_{I}e_{I}^{\mu}\right)\partial_{\mu}u+\epsilon_{I}\Gamma_{II}^{K}e_{K}^{\mu}\partial_{\mu}u.
   	       \end{aligned}   	       
   	   \end{equation*}
   \end{lemma}

   Based on Lemma $\ref{decomposition u}$, we will use the following standard elliptic estimates in \cite{Taylor}. 
   \begin{lemma}\label{elliptic estimate}
   	   Let $N$ is a transversal vectorfield to $\partial\Omega_{t}\subset \Omega_{t}$. For any scalar function $u$, 
   	   \begin{equation*}
   	         \Vert \partial_{x}^{2}u\Vert_{L^{2}(\Omega_{t})}\lesssim\Vert \barA u\Vert_{L^{2}(\Omega_{t})}+\Vert u\Vert_{H^{\frac{3}{2}}(\partial\Omega_{t})},
   	   \end{equation*}
   	   and 
   	   \begin{equation*}
   	        \Vert \partial_{x}^{2}u\Vert_{L^{2}(\Omega_{t})}\lesssim\Vert \barA u\Vert_{L^{2}(\Omega_{t})}+\Vert Nu\Vert_{H^{\frac{1}{2}}(\partial\Omega_{t})},
   	   \end{equation*}
   	   where the implicit constants depend on $\Omega_{t}$.
   \end{lemma}
 
   \begin{proof}[Proof of Proposition $\ref{elliptic prop}$]
   		 We proceed by induction on the order $p$ of the spatial derivatives $\partial_{x}^{p}$ applied to $\Theta$ and $\Vert V \Vert^{2}$. For $p=1$,  the result follows directly from the energy definition. Now, suppose that for $1\le p\le \frac{M+2}{2}-1$, the following holds: 
   		\begin{equation}\label{assumption 1}
   		   \sum_{q\le p}\sum_{k+2q\le M+2}\Vert \partial^{q}\partial_{t}^{k}\Theta\Vert_{L^{2}\left(\Omega_{t}\right)}^{2}+\sum_{q\le p}\sum_{k+2q\le M+2}\Vert \partial^{q}\partial_{t}^{k+1}\Vert V\Vert^{2}\Vert_{L^{2}\left(\Omega_{t}\right)}^{2}\lesssim\mathcal{E}_{\ell}\left(t\right).
   		\end{equation}
        We will now prove the estimates for $p+1$, specifically,
        \begin{equation}\label{p+1 equ}
            \sum_{k\le M-2p}\Vert \partial^{p+1}\partial_{t}^{k}\Theta\Vert_{L^{2}\left(\Omega_{t}\right)}^{2}+\sum_{k\le M-2p}\Vert \partial^{p+1}\partial_{t}^{k+1}\Vert V\Vert^{2}\Vert_{L^{2}\left(\Omega_{t}\right)}^{2}\lesssim\mathcal{E}_{\ell}\left(t\right).
        \end{equation}
       
       \textbf{Step 1: } We start with the estimate for $\Vert \partial^{p+1}\partial_{t}^{k+1}\Vert V\Vert^{2}\Vert_{L^{2}\left(\Omega_{t}\right)}$. To do this, we use the first estimate in Lemma $\ref{elliptic estimate}$ to estimate the equations for $ \partial^{p-1}\partial_{t}^{k+1}\Vert V\Vert^{2}$. Recall that $\partial_{t}^{k}\Vert V\Vert^{2}=0$ on $\partial\Omega$ for $k>0$. In fact, we need to commute tangential derivatives to reduce the boundary term. Among all the terms on the right-hand side of the equation in Lemma $\ref{decomposition u}$, our focus is on terms
       \begin{equation*}
         \partial^{p-1}\left(\frac{2\Vert V\Vert^{2}G^{'}}{G}\partial_{t}^{k+3}\Vert V\Vert^{2}\right)\ \text{and} \ \partial^{p-1}H_{k}.
       \end{equation*}
       And the remaining terms are bounded by the induction hypothesis $\left(\ref{assumption 1}\right)$. In particular, for $2p+k-1\le M-1\le \ell-1$ and $T$ is sufficiently small, the  highest curvature term $\partial^{p}\partial_{t}^{k-1}R$  can be bounded by $\mathcal{E}_{\ell}$, using fundamental theorem of calculus. In view of Lemma $\ref{V^2 higher order lemma}$, the top terms in $\partial^{p-1}H_{k}$ are 
       \begin{equation*}
             \partial^{p+1}\partial_{t}^{k}\Vert V\Vert^{2} \ \text{and} \ \partial^{p+1}\partial_{t}^{k-1}\Theta.
       \end{equation*} 
       Since $2\left(p+1\right)+k-1\le M+1$, we can use Lemma $\ref{M+1 control}$ to bound the $L^{2}\left(\Omega_{t}\right)$ norms of these terms by $\mathcal{E}_{\ell}$. All other terms appearing in $\partial^{p-1}H_{k}$ can be bounded by $\mathcal{E}_{\ell}$ using the induction hypothesis $\left(\ref{assumption 1}\right)$. Since $2\left(p-1\right)+k+2\le M\le \ell$, the term $\partial^{p-1}\left(\frac{2\Vert V\Vert^{2}G^{'}}{G}\partial_{t}^{k+3}\Vert V\Vert^{2}\right)$ can also be bounded by $\mathcal{E}_{\ell}$.
       
      \textbf{Step 2: }  Let us turn to the estimate for $\Vert\partial^{p+1}\partial_{t}^{k}\Theta\Vert_{L^{2}\left(\Omega_{t}\right)},\  k+2p+2\le M+2$. The second estimate in Lemma $\ref{elliptic estimate}$ gives
      \begin{equation*}
           \Vert\partial^{p+1}\partial_{t}^{k}\Theta\Vert_{L^{2}\left(\Omega_{t}\right)}\lesssim \Vert \barA \partial^{p-1}\partial_{t}^{k}\Theta\Vert_{L^{2}\left(\Omega_{t}\right)}+\Vert n^{\mu}\partial_{\mu}\left(\partial^{p-1}\partial_{t}^{k}\Theta\right)\Vert_{H^{\frac{1}{2}}\left(\partial\Omega_{t}\right)}.
      \end{equation*}
      The term $\barA \partial^{p-1}\partial_{t}^{k}\Theta$ has the similar structure to the corresponding term in Step 1, which can be bounded by $\mathcal{E}_{\ell}$ for $T$ sufficiently small. For the boundary contribution $\Vert n^{\mu}\partial_{\mu}\left(\partial^{p-1}\partial_{t}^{k}\Theta\right)\Vert_{H^{\frac{1}{2}}\left(\partial\Omega_{t}\right)}$, the trace theorem and Lemma $\ref{V boundary higher order lemma}$ gives  
      \begin{equation*}
         \begin{aligned}
           \Vert n^{\mu}\partial_{\mu}\left(\partial^{p-1}\partial_{t}^{k}\Theta\right)\Vert_{H^{\frac{1}{2}}\left(\partial\Omega_{t}\right)}&\lesssim \Vert n^{\mu}\partial_{\mu}\left(\partial^{p-1}\partial_{t}^{k}\Theta\right)\Vert_{H^{1}\left(\Omega_{t}\right)}\\
           &\lesssim \Vert \left[n^{\mu}\partial_{\mu},\partial^{p-1}\right]\partial_{t}^{k}\Theta\Vert_{H^{1}\left(\Omega_{t}\right)}+\Vert \partial^{p-1}\partial_{t}^{k+2}\Theta\Vert_{H^{1}\left(\Omega_{t}\right)}+\Vert\partial^{p-1}F_{k}\Vert_{H^{1}\left(\Omega_{t}\right)}.
         \end{aligned}
      \end{equation*}
      Using the induction hypothesis, the right terms above, except for the last term $\Vert\partial^{p-1}F_{k}\Vert_{H^{1}\left(\Omega_{t}\right)}$, are bounded by $\mathcal{E}_{\ell}$. Using Lemma $\ref{V boundary higher order lemma}$, the top order terms in $\partial^{p}F_{k}$ are
      \begin{equation*}
           \partial^{p+1}\partial_{t}^{k+1}\Vert V\Vert^{2}, \quad \partial^{p}\partial_{t}^{k+1}\Theta \ \text{and}\ \partial^{p+1}\partial_{t}^{k-1}\Theta.
      \end{equation*}
      Since $k+2\left(p+1\right)\le M+2$, the term $\partial^{p+1}\partial_{t}^{k+1}\Vert V\Vert^{2}$ can be handled as shown in Step 1. And the last two terms are bounded by Lemma $\ref{M+1 control}$ for $k+1+2p\le M+1$ and $k-1+2\left(p+1\right)\le M+1$.
      
      \textbf{Step 3: } We now arrive at the estimates for the curvature. By applying Lemmas \ref{V higher order lemma}, \ref{V boundary higher order lemma}, and \ref{V^2 higher order lemma}, along with the trace theorem, the highest-order curvature term requiring control is $\Vert \partial^{p}\partial_{t}^{k-1}R\Vert_{L^{2}(\Omega_{t})}$. This term is analogous to the highest-order curvature term treated in \cite{Miao1}. It is crucial to emphasize that all terms necessary for controlling the curvature are manageable. As a result, the proof follows a similar structure to Proposition 2.14 in \cite{Miao1}.   
   \end{proof}

\subsection{Proof of Proposition $\ref{main prop}$}\label{proof of Proposition 2.1}
 We will prove Proposition $\ref{main prop}$ in this section. To do this, we first introduce the following auxiliary lemmas. 
 
\begin{lemma}\label{2 V^2 estimate}
	Suppose the hypotheses of Proposition $\ref{main prop}$ hold. If $T>0$ is sufficiently small then for any $k\le \ell$ and $t\in[0,T]$
	\begin{equation*}
	        \Vert \partial^{2}\partial_{t}^{k}\Vert V\Vert^{2}\Vert_{L^{2}\left(\Omega_{t}\right)}^{2}\lesssim \mathcal{E}_{\ell}\left(t\right),
	\end{equation*}
	where the implicit constant depends polynomially on $\mathcal{E}_{k-1}\left(T\right)$ if $k$ is sufficiently large, and on $C_{1}$ for small $k$.
\end{lemma}
\begin{proof}
	For $k=0$ this follows by writing 
	\begin{equation*}
	     \nabla^{\mu}\Vert V\Vert^{2}=-2\nabla_{V}V^{\mu}.
	\end{equation*}
	 Inductively, we assume that the lemma holds for $k\le j-1\le \ell-1$ and prove it for $k=j$. In view of Lemma $\ref{decomposition u}$ and $\ref{elliptic estimate}$, we get
	\begin{equation*}
	   \Vert\partial^{2}\partial_{t}^{j}\Vert V\Vert^{2}\Vert_{L^{2}\left(\Omega_{t}\right)}^{2}\lesssim \Vert\partial_{t}^{j+2}\Vert V\Vert^{2}\Vert_{L^{2}\left(\Omega_{t}\right)}^{2}+\Vert\partial\partial_{t}^{j+1}\Vert V\Vert^{2}\Vert_{L^{2}\left(\Omega_{t}\right)}^{2}+\Vert H_{j-1}\Vert_{L^{2}\left(\Omega_{t}\right)}^{2},
	\end{equation*}
	where $H_{j-1}$ is given in Lemma $\ref{V^2 higher order lemma}$. The estimates for the first two terms follow from the definition of the energies. In view of Lemma $\ref{V^2 higher order lemma}$, except for the highest order terms 
	\begin{equation*}
	     \partial^{2}\partial_{t}^{j-1}\Vert V\Vert^{2} \ \text{and} \ \partial^{2}\partial_{t}^{j-2}\Theta
	\end{equation*}
	in $H_{j-1}$, the other terms are directly bounded by $\mathcal{E}_{\ell-1}\le \mathcal{E}_{\ell}$. The first term above can be bounded using induction hypothesis. While the second term is bounded using Proposition $\ref{elliptic prop}$, because $j-2+2\cdot 2\le \ell +2$. So can the term $\partial\partial_{t}^{j-2}R,\ j-2\le \ell -1$.
\end{proof}	
	
	The next lemma will be used to estimate $\partial\partial_{t}^{k-1}\Theta$ on the boundary.
\begin{lemma}\label{V boundary estimate}
	Under the hypotheses of Proposition $\ref{main prop}$ and the condition $\delta>0$ (small), if $T>0$ is sufficiently small then for any $k\le \ell$ and $t\in t\in[0,T]$
	\begin{equation*}
	    \Vert \partial\partial_{t}^{j}\Theta\Vert_{L^{2}\left(\partial\Omega_{0}^{T}\right)}^{2}\lesssim \mathcal{E}_{\ell}\left(0\right)+R_{j}\left(\mathcal{E}_{k-1}\left(T\right)\right)+\delta\mathcal{E}_{k}\left(T\right).
	\end{equation*} 
	Here the implicit constant depends on $C_{1}$, and $R_{j}$ is some polynomial function for each $j\le k-1$.
\end{lemma}	
\begin{proof}
	For $j=0$, using the trace theorem, we have
	\begin{equation*}
	    \Vert\partial\partial_{t}^{j}\Theta\Vert_{L^{2}\left(\partial\Omega_{0}^{T}\right)}^{2}\lesssim\int_{0}^{T}\Vert\partial\partial_{t}^{j}\Theta\Vert_{H^\frac{1}{2}\left(\Omega_{t}\right)}^{2}\mathrm{d}t.
	\end{equation*}
	Since $2\cdot(1+\frac{1}{2})+j\le k+2\le \ell +2$, the result follows from using Proposition $\ref{elliptic prop}$ and taking $T$ sufficiently small. Proceeding inductively, suppose that we have prove the statement for $j\le k-2$ and prove it for $j=k-1$. Using Lemma $\ref{Theta boundary control}$, we need to show that $\Vert H_{k-1}\Vert_{L^{2}\left(\Omega_{t}\right)}$ is bounded by $\mathcal{E}_{k}$. In view of Lemma $\ref{V higher order lemma}$, the top order terms are
	\begin{equation*}
	    \partial^{2}\partial_{t}^{k-2}\Theta\quad \text{and} \ \partial\partial_{t}^{k-2}R.
	\end{equation*}
	Since $k-2+2\cdot 2\le k +2$ and $k-2\le k-1$, the above terms are bounded using Proposition $\ref{elliptic prop}$. And other terms in $H_{k-1}$ can be bounded by the definition of $\mathcal{E}_{k}$. For the integral at the boundary in Lemma $\ref{Theta boundary control}$, we concentrate on the contribution of $\Vert \partial_{t}\partial_{t}^{k-1}\Theta\Vert_{L^{2}\left(\partial\Omega_{0}^{T}\right)}^{2}$ and $\Vert D_{n}\partial_{t}^{k-1}\Theta\Vert_{L^{2}\left(\partial\Omega_{0}^{T}\right)}^{2}$. The first term follows from the definition of $\mathcal{E}_{k-1}$. For the second term, we use Lemma $\ref{V boundary higher order lemma}$ to write
	\begin{equation*} 
	        D_{n}\partial_{t}^{k-1}\Theta^{I}=-\frac{1}{\gamma}\partial_{t}^{k+1}\Theta^{I}+F_{k-1}.
	\end{equation*} 
    The contribution of $\Vert \partial_{t}^{k+1}\Theta\Vert_{L^{2}(\partial\Omega_{0}^{T})}$ is of the desired form by taking $T$ small. Then we need to consider $\Vert F_{k-1}\Vert_{L^{2}(\partial\Omega_{0}^{T})}$. It is clearly that $\Vert \partial\partial_{t}^{k}\Vert V\Vert^{2}\Vert_{L^{2}(\partial\Omega_{0}^{T})}$ and $\Vert \partial_{t}^{k}\Theta\Vert_{L^{2}(\partial\Omega_{0}^{T})}$ can be bounded by $\mathcal{E}_{k-1}$. The contribution of $\Vert\partial_{t}^{k-2}R\Vert_{L^{2}(\partial\Omega_{0}^{T})}$ is of the desired form by using the trace theorem and Proposition $\ref{elliptic prop}$. And the desired form of other terms in $F_{k-1}$ follows from the induction hypothesis.   	
\end{proof}

We now turn to the proof of Proposition $\ref{main prop}$.
\begin{proof}[Proof of Proposition $\ref{main prop}$.] 
We prove the result inductively. First by Lemmas $\ref{V energy lemma}$, $\ref{V^2 energy lemma}$, $\ref{W energy lemma}$, and using elliptic and Sobolev estimates to bound lower-order terms in $L^{\infty}$, we have
\begin{equation*}
        \mathcal{E}_{0}\left(T\right)\le \mathcal{P}_{k}\left(\mathcal{E}_{\ell}\left(0\right)\right)+T\left(1+C_{1}\right)^{m},
\end{equation*}
for some $m>0$. Therefore, smallness of $T$ implies
\begin{equation*}
    \mathcal{E}_{0}\left(T\right)\le \mathcal{P}_{k}\left(\mathcal{E}_{\ell}\left(0\right)\right).
\end{equation*}
 Now suppose that if $T$ is sufficiently small
 \begin{equation*}
 \mathcal{E}_{j}\left(T\right)\le \mathcal{P}_{j}\left(\mathcal{E}_{\ell}\left(0\right)\right)
 \end{equation*}
 for all $j\le k-1$ and some $k\le \ell$. In this case, we will prove 
 \begin{equation*}
 \mathcal{E}_{k}\left(T\right)\le \mathcal{P}_{k}\left(\mathcal{E}_{\ell}\left(0\right)\right)
 \end{equation*} 
by taking $T$ even smaller.

\textbf{Step 1: } We show that given $\eta>0$ (small) if $T>0$ is sufficiently small then 
\begin{equation*}
    \sup_{0\le t\le T}\Vert \partial\partial_{t}^{k+1}\Vert V\Vert^{2}\Vert_{L^{2}\left(\Omega_{t}\right)}^{2}+\Vert \partial\partial_{t}^{k+1}\Vert V\Vert^{2}\Vert_{L^{2}\left(\partial\Omega_{0}^{T}\right)}^{2}\le \mathcal{P}_{k}\left(\mathcal{E}_{\ell}\left(0\right)\right)+\eta\mathcal{E}_{k}\left(T\right).
\end{equation*}	
In analogy to Lemma $\ref{V^2 energy lemma}$, we estimate the wave equation satisfied by $\partial_{t}^{k+1}\Vert V\Vert^{2}$. To do this, we only need to estimate 
\begin{equation*}
      \frac{2G^{'}\Vert V\Vert^{2}}{G}\partial_{t}^{k+3}\Vert V\Vert^{2}Q\partial_{t}^{k+1}\Vert V\Vert^{2}\sqrt{\abs{g}}\quad \text{and}\quad H_{k}Q\partial_{t}^{k+1}\Vert V\Vert^{2}\sqrt{\abs{g}}.
\end{equation*}
And other terms are of the desired form by taking $T$ to be small. For the first of above terms, we have
\begin{equation*}
    \begin{aligned}
        \frac{2G^{'}\Vert V\Vert^{2}}{G}\partial_{t}^{k+3}\Vert V\Vert^{2}Q\partial_{t}^{k+1}\Vert V\Vert^{2}\sqrt{\abs{g}}&=\partial_{t}\left(\frac{\Vert V\Vert^{2}G^{'}}{G}\left(\partial_{t}^{k+2}\Vert V\Vert^{2}\right)^{2}\sqrt{\abs{g}}\right)-\alpha\partial_{t}\left(\frac{2\Vert V\Vert^{2}G^{'}}{G}\partial_{t}^{k+2}\Vert V\Vert^{2}D_{n}\partial_{t}^{k+1}\Vert V\Vert^{2}\sqrt{\abs{g}}\right)\\
        &+\alpha D_{n}\left(\frac{\Vert V\Vert^{2}G^{'}}{G}\left(\partial_{t}^{k+2}\Vert V\Vert^{2}\right)^2\sqrt{\abs{g}}\right)+F^{'},
    \end{aligned}
\end{equation*} 
 with $F^{'}$ being a linear combination of $\partial_{t}^{k+2}\Vert V\Vert^{2}$, $D\partial_{t}^{k+1}\Vert V\Vert^{2}$ and $\partial_{t}^{k+1}\Vert V\Vert^{2}$ bounded by $\mathcal{E}_{k}$ for $T$ is small. Note that since $\frac{2G^{'}\Vert V\Vert^{2}}{G}$ is non-negative and $\alpha$ is sufficiently small, other terms can be absorbed in the integrals over $\Omega_{t}$ and $\Omega_{0}$. For the contibution of $H_{k}Q\partial_{t}^{k+1}\Vert V\Vert^{2}\sqrt{\abs{g}}$, except for $\partial^{2}\partial_{t}^{k}\Vert V\Vert^{2}Q\partial_{t}^{k+1}\Vert V\Vert^{2}\sqrt{\abs{g}}$ and $\partial^{2}\partial_{t}^{k-1}\Theta Q\partial_{t}^{k+1}\Vert V\Vert^{2}\sqrt{\abs{g}}$, the corresponding contribution can be handled by Cauchy-Schwarz, Proposition $\ref{elliptic prop}$, the induction hypothesis, and by taking $T$ small. Using the same argument and Lemma $\ref{2 V^2 estimate}$, the contribution of $\partial^{2}\partial_{t}^{k}\Vert V\Vert^{2}Q\partial_{t}^{k+1}\Vert V\Vert^{2}\sqrt{\abs{g}}$ can also be handled by $T\mathcal{E}_{k}$. Finally for terms invloving $\partial^{2}\partial_{t}^{k-1}\Theta\sqrt{\abs{g}}$, we need to perform a few integration by part. The corresponding term can be written as 
 \begin{equation*}
    \begin{aligned}
       F\cdot \nabla^{(2)}\partial_{t}^{k-1}\Theta Q\partial_{t}^{k+1}\Vert V\Vert^{2}&=\nabla\cdot\left(F\cdot \nabla\partial_{t}^{k-1}\Theta Q\partial_{t}^{k+1}\Vert V\Vert^{2}\right)-\partial_{t}\left(\left(\nabla Q\partial_{t}^{k}\Vert V\Vert^{2}\right)F\cdot \nabla\partial_{t}^{k-1}\Theta\right)\\
       &-Q\partial_{t}^{k+1}\Vert V\Vert^{2}\nabla F\cdot \nabla\partial_{t}^{k-1}\Theta-\left[\nabla Q, \partial_{t}\right]\partial_{t}^{k}\Vert V\Vert^{2} F\cdot \nabla\partial_{t}^{k-1}\Theta\\
       &+\left(\nabla Q\partial_{t}^{k}\Vert V\Vert^{2}\right)\partial_{t}F\cdot \nabla\partial_{t}^{k-1}\Theta+\nabla Q\partial_{t}^{k}\Vert V\Vert^{2} \cdot F\cdot \left[\partial_{t}, \nabla\right]\partial_{t}^{k-1}\Theta\\
       &+\nabla Q\partial_{t}^{k}\Vert V\Vert^{2}\cdot F\cdot\nabla\partial_{t}^{k-1}\Theta.
    \end{aligned}
 \end{equation*}
 Here $F$ consists of lower order terms. Integrating the above identity over $\Omega_{t}\times \left[0, \tau\right]$, the last five terms can be handled using Cauchy-Schwarz, Proposition $\ref{elliptic prop}$, Lemma $\ref{2 V^2 estimate}$, and taking T small. For the first line we integrate by parts. We can arrive at 
 \begin{equation*}
    \begin{aligned}
       &\int_{0}^{\tau}\int_{\Omega_{t}}\nabla\cdot\left(F\cdot \nabla\partial_{t}^{k-1}\Theta Q\partial_{t}^{k+1}\Vert V\Vert^{2}\right)-\partial_{t}\left(\left(\nabla Q\partial_{t}^{k}\Vert V\Vert^{2}\right)F\cdot \nabla\partial_{t}^{k-1}\Theta\right)\mathrm{d}x\mathrm{d}t\\
       &=\int_{\Omega_{\tau}}F\cdot \nabla\partial_{t}^{k-1}\Theta Q\partial_{t}^{k+1}\Vert V\Vert^{2}\mathrm{d}x+\int_{\Omega_{\tau}}\left(\nabla Q\partial_{t}^{k}\Vert V\Vert^{2}\right)F\cdot \nabla\partial_{t}^{k-1}\Theta\mathrm{d}x-\int_{\Omega_{0}}F\cdot \nabla\partial_{t}^{k-1}\Theta Q\partial_{t}^{k+1}\Vert V\Vert^{2}\mathrm{d}x\\
       &-\int_{\Omega_{0}}\left(\nabla Q\partial_{t}^{k}\Vert V\Vert^{2}\right)F\cdot \nabla\partial_{t}^{k-1}\Theta\mathrm{d}x+\int_{0}^{\tau}\int_{\partial\Omega_{t}}F\cdot\nabla_{n}\partial_{t}^{k-1}\Theta Q\partial_{t}^{k+1}\Vert V\Vert^{2}\mathrm{d}S\mathrm{d}t.
    \end{aligned}
 \end{equation*}
 Then, using Cauchy-Schwarz with a small constant and Lemma $\ref{2 V^2 estimate}$,  the first line on the right-hand side can be bounded. The contribution of the integrals over $\Omega_{0}$ can also be handled with the same argument. And the last term can be estimated using Cauchy-Schwarz and Lemma $\ref{V boundary estimate}$.

\textbf{Step 2: }We show that given $\eta>0$ (small) if $T>0$ is sufficiently small then
\begin{equation*}
     \abs{\int_{0}^{T}\int_{\Omega_{t}}\left(\Box\partial_{t}^{k}\Theta\right)\left(\partial_{t}^{k+1}\Theta\right)\sqrt{\abs{g}}\mathrm{d}x\mathrm{d}t}\le \mathcal{P}_{k}\left(\mathcal{E}_{\ell}\left(0\right)\right)+\eta\mathcal{E}_{k}\left(T\right).
\end{equation*}
We preform the similar argument in Step 1. Here we need to focus on the contribution of $\partial^{2}\partial_{t}^{k-1}\Theta$. And other terms can be handled using Cauchy-Schwarz and by taking $T$ small. For the contribution of $\partial^{2}\partial_{t}^{k-1}\Theta$, we have
\begin{equation*}
   \begin{aligned}
      F\cdot \nabla^{(2)}\partial_{t}^{k-1}\Theta\partial_{t}^{k+1}\Theta&=\nabla\cdot\left(F\cdot \nabla\partial_{t}^{k-1}\Theta \partial_{t}^{k+1}\Theta\right)-\partial_{t}\left(F\cdot \nabla\partial_{t}^{k-1}\Theta\nabla\partial_{t}^{k}\Theta\right)\\
      &-\nabla F\cdot \nabla\partial_{t}^{k-1}\Theta\partial_{t}^{k+1}\Theta-\left[\nabla, \partial_{t}\right]\partial_{t}^{k}\Theta \cdot F\cdot \nabla\partial_{t}^{k-1}\Theta \\
      &+\partial_{t}F\cdot \nabla\partial_{t}^{k-1}\Theta\nabla\partial_{t}^{k}\Theta+\nabla \partial_{t}^{k}\Theta\cdot F\cdot \left[\partial_{t}, \nabla\right]\partial_{t}^{k-1}\Theta\\
      &+\nabla\partial_{t}^{k}\Theta\cdot F\cdot\nabla\partial_{t}^{k}\Theta.
   \end{aligned}
\end{equation*}
All of the above terms can be handled using similar arguments as in Step 1.

\textbf{Step 3: }We show that given $\eta>0$ (small) if $T>0$ is sufficiently small then
\begin{equation*}
\sup_{0\le t\le T}\Vert \partial_{t}^{k}R\Vert_{L^{2}\left(\Sigma_{t}\right)}^{2}\le \mathcal{P}_{k}\left(\mathcal{E}_{\ell}\left(0\right)\right)+\eta\mathcal{E}_{k}\left(T\right).
\end{equation*}
Since $R$ can be expressed in terms of $W$ and $R_{\mu\nu}=GV_{\mu}V_{\nu}-pg_{\mu\nu}+\frac{G}{2}\Vert V\Vert^{2}g_{\mu\nu}$, we can prove this estimate with $R$ replaced by $W$. The result follows from using Lemma $\ref{W energy lemma}$ and taking $T$ small, where $W$ satisfies the equation in Lemma $\ref{W higher order lemma}$. We just need to focus on $D\partial_{t}^{k-1}W$. Note that it is of the desired form using Proposition $\ref{elliptic prop}$.

\textbf{Step 4: }We show that given $\eta>0$ (small) if $T>0$ is sufficiently small then
\begin{equation*}
\sup_{0\le t\le T}\left(\Vert \partial\partial_{t}^{k}\Theta\Vert_{L^{2}\left(\Omega_{t}\right)}^{2}+\Vert \partial_{t}^{k+1}\Theta\Vert_{L^{2}\left(\partial\Omega_{t}\right)}^{2}\right)\le \mathcal{P}_{k}\left(\mathcal{E}_{\ell}\left(0\right)\right)+\eta\mathcal{E}_{k}\left(T\right).
\end{equation*}
We apply Lemma $\ref{V energy lemma}$ to $\partial_{t}^{k}\Theta$. Here $F$ is described in Lemma $\ref{V higher order lemma}$ and $f$ is in the form of Lemma $\ref{V higher order lemma}$. Except for $F$ and $f$, all other terms on the right-hand side of $(\ref{energy equ1.1})$ can be estimated by making $T$ sufficiently small. The contribution of $F$ has already been dealt with in Step 2. For the contribution of $f$, we need to estimate $\Vert \partial_{t}^{k-1}R\Vert_{L^{2}\left(\partial\Omega_{0}^{T}\right)}$. To do this, we use the trace theorem and Proposition $\ref{elliptic prop}$. The induction claim $\mathcal{E}_{k}\left(T\right)\lesssim\mathcal{P}\left(\mathcal{E}_{\ell}\left(0\right)\right)$ follows by choosing $\eta$ sufficiently small in Steps $1-4$. Then we finish the proof of Proposition $\ref{main prop}$.
\end{proof}

\section{Linear theory and iteration}\label{sec3}	
	In this section we discuss the linear existence theory for the equations we use for the interation. The linear equations for the geometric quantities are of first order hyperbolic or transport types.
	
\subsection{The linear existence theory of the equation for $R$. }Since the curvature equation satisfies a first order symmetric hyperbolic system, its linear existence theory is known. 
	
    In the following, we will deal with the problem in the frame with slightly modified, with $\check{e}_{0}=\hat{V}=\partial_{t}$ and
    \begin{equation*}
       \begin{aligned}
           \check{e}_{3}&=\frac{\left(e_{3}+g(e_{3},\check{e}_{0})\check{e}_{0}\right)}{\Vert e_{3}+g(e_{3},\check{e}_{0})\check{e}_{0}\Vert},\\
           \check{e}_{2}&=\frac{\left(e_{2}+g(e_{2},\check{e}_{0})\check{e}_{0}-g(e_{2},\check{e}_{3})\check{e}_{3}\right)}{\Vert e_{2}+g(e_{2},\check{e}_{0})\check{e}_{0}-g(e_{2},\check{e}_{3})\check{e}_{3}\Vert},\\
           \check{e}_{1}&=\frac{\left(e_{1}+g(e_{1},\check{e}_{0})\check{e}_{0}-g(e_{1},\check{e}_{2})\check{e}_{2}-g(e_{1},\check{e}_{3})\check{e}_{3}\right)}{\Vert e_{1}+g(e_{1},\check{e}_{0})\check{e}_{0}-g(e_{1},\check{e}_{2})\check{e}_{2}-g(e_{1},\check{e}_{3})\check{e}_{3}\Vert}.           
       \end{aligned}
    \end{equation*}
    We denote contractions with $\check{e}_{I}$ with a check in a slight abuse of notation. So for instance we write
    \begin{equation*}
        \check{F}_{IJ}=F(\check{e}_{I},\check{e}_{J})=R(\check{e}_{I}, \check{e}_{J}, \check{X}_{A}, \check{X}_{B}),\quad \check{\Gamma}_{IJ}^{K}\check{e}_{K}=\nabla_{\check{e}_{I}}\check{e}_{J},\quad  \check{D}_{I}=\check{e}_{I}^{\mu}\partial_{\mu},\quad \mathrm{etc.}
    \end{equation*}
    We require that the components $R(\check{X}_{A},\check{X}_{B},\check{X}_{C},n)$ are defined by 
    \begin{equation*}
      R(\check{X}_{A},\check{X}_{B},\check{X}_{C},n)=-R(\check{X}_{A},\check{X}_{B},n,\check{X}_{C})=R(\check{X}_{C},n,\check{X}_{A},\check{X}_{B}),
    \end{equation*}
    the components $R(\check{X}_{A},n,\check{X}_{B},n)$, $A\le B$ are defined by 
    \begin{equation*}
        \begin{aligned}
           R(\check{X}_{A},n,\check{X}_{B},n)&=-R(n,\check{X}_{A},\check{X}_{B},n)\\
           &=\chi_{\Omega}\left(\frac{1}{2}G\Vert V\Vert^{2}m_{AB}-pm_{AB}+Gg(\Theta^{I}e_{I},\check{X}_{A})g(\Theta^{I}e_{I},\check{X}_{B})\right)\\
           &-\sum_{C=0}^{2}\epsilon_{C}R(\check{X}_{A},\check{X}_{C},\check{X}_{B},\check{X}_{C})           
        \end{aligned}     
    \end{equation*}
    and the components $R(\check{X}_{B},n,\check{X}_{A},n)$, $A<B$ are defined by
    \begin{equation*}
        R(\check{X}_{B},n,\check{X}_{A},n)=R(\check{X}_{A},n,\check{X}_{B},n).
    \end{equation*}         
    Note that since $g(e_{J},\check{e}_{0})=\epsilon_{J}\hat{\Theta}^{J}$, the frame $\left\{\check{e}_{I}\right\}$ is a linear combination of $\left\{\check{e}_{I}\right\}$ with coefficients which depend algebraically on $\hat{\Theta}$, and therefore have the same regularity properties. The electric and magnetic parts of $\check{F}$ are denote by $\check{E}$ and $\check{H}$ and are given by 
    \begin{equation*}
           \check{E}_{I}=\check{F}_{0I},\quad \check{H}^{I}=-\frac{1}{2}\sum_{J,K=1}^{3}\epsilon^{IJK}\check{F}_{JK}, \quad I=1, 2, 3.
    \end{equation*}
    
    We then define the following first order differential operators (where $\check{\nabla}=\check{D}+\check{\Gamma}$)
    \begin{equation*}
       \begin{aligned}
          &\check{\mathrm{div}}\check{E}:=\sum_{\tilde{I}=1}^{3}\check{\nabla}_{\tilde{I}}\check{E}_{\tilde{I}},\quad \check{\mathrm{div}}\check{H}:=\sum_{\tilde{I}=1}^{3}\check{\nabla}_{\tilde{I}}\check{H}^{\tilde{I}},\\
          &(\check{\mathrm{curl}}\check{E})_{\tilde{I}}:=\sum_{\tilde{J}, \tilde{K}=1}^{3}\epsilon_{\tilde{I}\tilde{J}\tilde{K}}\check{\nabla}_{\tilde{J}}\check{E}_{\tilde{K}}, \quad (\check{\mathrm{curl}}\check{H})^{\tilde{I}}:=\sum_{\tilde{J}, \tilde{K}=1}^{3}\epsilon_{\tilde{I}\tilde{J}\tilde{K}}\check{\nabla}_{\tilde{J}}\check{H}^{\tilde{K}},
       \end{aligned}
    \end{equation*}
    where 
    \begin{equation*}
       \epsilon_{\tilde{I}\tilde{J}\tilde{K}}=\begin{cases}
         1, &\quad (\tilde{I}\tilde{J}\tilde{K})\ \text{is an even permutation of} \ (123)\\
         -1,&\quad (\tilde{I}\tilde{J}\tilde{K})\ \text{is an odd permutation of} \ (123)
       \end{cases}.
    \end{equation*}
   Although these operators are not strictly geometric divergence and curl operators, we will continue to refer to them as such for convenience. The reason for decomposing $F$ using  $\left\{\check{e}_{I}\right\}$ rather than $\left\{e_{I}\right\}$ is that, in proving elliptic estimates for $\check{E}$ and $\check{H}$, we need to use the divergence equations satisfied by them. To demonstrate that the divergence of $\check{E}$ and $\check{H}$ are lower order during the iteration, it is more convenient if $\check{e}_{0}$ coincides with $\partial_{t}$. We turn to the details. The equations satisfied by $\check{E}$ and $\check{H}$ are 
    \begin{equation}\label{modified frame maxwell system}
         \partial_{t}\check{E}+\check{\mathrm{curl}}\check{H}=\check{\mathcal{I}},\quad \partial_{t}\check{H}-\check{\mathrm{curl}}\check{E}=\check{J}^{\ast},
    \end{equation}
    where $\check{\mathcal{I}}$ and $\check{\mathcal{J}}$ are defined as in $(\ref{I component1})$, $(\ref{I component2})$, $(\ref{J component})$ with $D, \Theta, \Gamma, F, X, R$ being replaced by $\check{D}, \check{\Theta}, \check{\Gamma}, \check{F}, \check{X}, \check{R}$ respectively. In terms of $\check{W}=(\check{E}, \check{H})$ and $\check{K}=(\check{\mathcal{I}}, \check{\mathcal{J}^{\ast}})$, equation $(\ref{modified frame maxwell system})$ becomes the first order symmetric hyperbolic system
    \begin{equation}\label{new hyperbolic system}
        \sum_{\nu=0}^{3}\check{\mathcal{B}}^{\mu}\partial_{\mu}\check{W}=\check{\mathcal{K}}, \quad \check{\mathcal{B}}^{0}=1+\sum_{\tilde{I}=1}^{3}\check{e}_{\tilde{I}}^{0}\mathcal{A}^{\tilde{I}},\quad \check{\mathcal{B}}^{j}:=\sum_{\tilde{I}=1}^{3}\check{e}_{\tilde{I}}^{j}\mathcal{A}^{\tilde{I}},\  j=1, 2, 3.
    \end{equation}
We will establish the iterative result for $\check{W}$ as Proposition 3.4 in \cite{Miao1}.
\begin{proposition}\label{local existence for R}
	Suppose $\check{e}_{\tilde{I}}, \tilde{I}=1, 2, 3,$ and $\check{\mathcal{K}}$ satisfy the following conditions:
	\begin{equation}\label{W condition}
	    \begin{aligned}
	       1-\left(\sum_{\tilde{I}=1}^{3}(\check{e}_{\tilde{I}}^{0})^{2}\right)^{\frac{1}{2}}\ge \kappa>0,&\\
	       \partial_{t}^{k}\check{e}_{I}\in L^{\infty}([0,T];L^{2}(\mathbb{R}^{3})),&\quad k\le K,\\
	       \partial^{p}\partial_{t}^{k}\check{e}_{I}\in L^{\infty}([0,T];L^{2}(\Omega_{0})\cap L^{2}(\Omega_{0}^{c})),& \quad 2p+k\le K,\\
	       \partial_{t}^{k}\check{\mathcal{K}}\in L^{2}([0,T]\times \mathbb{R}^{3}),&\quad k\le K\\
	       \partial_{t}^{k}\check{\mathcal{K}}\in L^{\infty}([0,T]; H^{p-1}(\Omega_{0})\cap H^{p-1}(\Omega_{0}^{c})),&\quad 2p+k\le K,\\
	       \partial_{t}^{\ell}(\check{\mathrm{div}}\check{\mathcal{K}}_{\check{E}}, \check{\mathrm{div}}\check{\mathcal{K}}_{\check{H}})\in L^{\infty}([0,T]; H^{p-1}(\Omega_{0})\cap H^{p-1}(\Omega_{0}^{c})),&\quad 2p+k\le K.	       
	    \end{aligned}
	\end{equation} 
Let $\check{\mathcal{B}}^{\mu}$ be defined as in $(\ref{new hyperbolic system})$. Then there is a unique solution $\check{W}=(\check{E}, \check{H})$ to 
\begin{equation}\label{new hyperbolic system1}
     \sum_{\nu=0}^{3}\check{\mathcal{B}}^{\mu}\partial_{\mu}\check{W}=\check{\mathcal{K}}=(\check{\mathcal{K}}_{\check{E}}, \check{\mathcal{K}}_{\check{H}}),    
\end{equation}	
which in addition satisfies 
\begin{equation}\label{space estimate for W}
     \begin{aligned}
         &\sup_{t\in[0,T]}\left(\Vert \partial_{t}^{k}\check{W}(t)\Vert_{L^{2}(
         	\mathbb{R}^{3})}+\Vert \partial\partial_{t}^{k-1}\check{W}(t)\Vert_{L^{2}(\mathbb{R}^{3})}\right)\\
         &\le C_{1}e^{C_{2}T}\left(\Vert \partial_{t}^{k}\check{W}(0)\Vert_{L^{2}(\mathbb{R}^{3})}+\Vert \partial\partial_{t}^{k-1}\check{W}(0)\Vert_{L^{2}(\mathbb{R}^{3})}+\Vert \partial_{t}^{k}\mathcal{K}\Vert_{L^{2}([0,T]\times \mathbb{R}^{3})}\right),
     \end{aligned}
\end{equation}
for $k\le K$. Moreover, there are functions $P_{k}$ depending polynomially on their arguments such that for $k\le K$ and $2p+k\le K$, and $\tau\le T$,
\begin{equation}\label{time estimate for W}
    \begin{aligned}
        \sup_{t\in[0,\tau]}&\left(\Vert \partial^{p}\partial_{t}^{k}\check{W}(t)\Vert_{L^{2}(\Omega_{0})}+\Vert \partial^{p}\partial_{t}^{k}\check{W}(t)\Vert_{L^{2}(\Omega_{0}^{c})}\right)\\
        \le P_{k}&\left(\sup_{t\le \tau}\left(\sum_{\ell\le 2p+k}\Vert \partial_{t}^{\ell}\check{W}(t)\Vert_{L^{2}(\mathbb{R}^{3})}+\sum_{\ell\le 2p+k-1}\Vert \partial \partial_{t}^{\ell}\check{W}(t)\Vert_{L^{2}(\mathbb{R}^{3})}\right),  \right.\\
        &\sum_{\ell\le k}\Vert \partial_{t}^{\ell}\check{\mathcal{K}}\Vert_{L^{\infty}([0,T];H^{p-1}(\Omega_{0}))},\sum_{\ell\le k}\Vert \partial_{t}^{\ell}\check{\mathcal{K}}\Vert_{L^{\infty}([0,T];H^{p-1}(\Omega_{0}^{c}))},\\
        &\left.\sum_{\ell\le k}\Vert \partial_{t}^{\ell}(\check{\mathrm{div}}\check{\mathcal{K}}_{\check{E}}, \check{\mathrm{div}}\check{\mathcal{K}}_{\check{H}})\Vert_{L^{\infty}([0,T];H^{p-1}(\Omega_{0}))}, \sum_{\ell\le k}\Vert \partial_{t}^{\ell}(\check{\mathrm{div}}\check{\mathcal{K}}_{\check{E}}, \check{\mathrm{div}}\check{\mathcal{K}}_{\check{H}})\Vert_{L^{\infty}([0,T];H^{p-1}(\Omega_{0}^{c}))} \right).        
    \end{aligned}
\end{equation}
Here the constants in $(\ref{space estimate for W})$ and $(\ref{time estimate for W})$ depend on $\kappa$ and the norms of $\check{e}_{I}$ appearing in $(\ref{W condition})$.
\end{proposition}

\subsection{The linear existence theory of the equation for $\Theta$. } First, we will establish the weak formulation for the fluid equations. Let us recall that
\begin{equation}\label{form V equation}
   \begin{cases}
      \square \Theta =F   \ &\text{in} \ \Omega,\\
      (\partial_{t}^{2}+\gamma D_{n})\Theta=f \ &\text{on} \ \partial\Omega,
   \end{cases}
\end{equation}   	
where $F$ and $f$ denote the right hand sides of $(\ref{2V interior frame equ})$ and $(\ref{2V boundary equ})$ respectively. Assume that all functions are smooth, we can derive the following weak formulation
\begin{equation}\label{integration equ of V}
   \begin{aligned}
      \int_{\partial\Omega_{0}}\frac{1}{\gamma}f\phi\mathrm{d}S-\int_{\Omega_{0}}F\phi \mathrm{d}x&=\int_{\Omega_{0}}-g^{tt}\partial_{t}^{2}\Theta\phi\mathrm{d}x+\int_{\partial\Omega_{0}}\frac{1}{\gamma}\partial_{t}^{2}\Theta\phi\mathrm{d}S+\int_{\Omega_{0}}g^{ab}\partial_{a}\Theta\partial_{b}\phi\mathrm{d}x\\
      &+2\int_{\Omega_{0}}g^{ta}\partial_{t}\Theta\partial_{a}\phi\mathrm{d}x-\int_{\partial\Omega_{0}}g^{tr}\partial_{t}\Theta\phi\mathrm{d}S-\frac{1}{2}\int_{\Omega_{0}}g^{\alpha\beta}\partial_{\beta}\Theta\phi\partial_{\alpha}\log\abs{g}\mathrm{d}x\\
      &-\int_{\Omega_{0}}\partial_{\alpha}\Theta\phi\partial_{t}g^{t\alpha}\mathrm{d}x+\int_{\Omega_{0}}\partial_{t}\Theta\phi\partial_{a}g^{ta}\mathrm{d}x,
   \end{aligned}
\end{equation}
where $\phi$ is a test function.

We denote $(\cdot,\cdot)$ the pairing of $(H^{1})^{\ast}$ with its dual space $H^{1}$. Then a bounded linear map $\Phi: H^{1}(\Omega_{0})\to (H^{1}(\Omega_{0}))^{\ast}$ can be defined by
\begin{equation*}
      \left(\Phi(u),\phi\right):=\left\langle-g^{tt}u,\phi\right\rangle+\llangle\gamma^{-1}\mathrm{tru},\mathrm{tr\phi}\rrangle, \ \phi\in H^{1}(\Omega_{0}).
\end{equation*} 
Here $\left\langle \cdot,\cdot \right \rangle$ denotes the inner product in $L^{2}(\Omega_{0})$ with respect to $\mathrm{d}x$, $\llangle\cdot,\cdot\rrangle$ denotes the inner product in $L^{2}(\partial\Omega_{0})$ with respect to the induced Euclidean measure $\mathrm{d}S$ and $\mathrm{tr}$ denotes the Sobolev trace operator. To simplify the notation, we use the following bilinear forms
\begin{equation}\label{bilinear forms}
    \begin{aligned}
       &B: H^{1}(\Omega_{0})\times H^{1}(\Omega_{0})\to\mathbb{R},\qquad C: L^{2}(\Omega_{0})\times H^{1}(\Omega_{0})\to \mathbb{R}, \qquad D, E: L^{2}(\partial\Omega_{0})\times H^{1}(\Omega_{0})\to \mathbb{R},\\
       &B(u,v):=\left\langle g^{ab}\partial_{a}u,\partial_{b}v\right\rangle-\frac{1}{2}\left\langle \partial_{a}u, vg^{a\alpha}\partial_{\alpha}\log\abs{g}\right\rangle-\left\langle\partial_{a}u, v\partial_{t}g^{ta}\right\rangle+\left\langle u, v\partial_{t}^{2}g^{tt}\right\rangle,\\
       &C(u,v):=2\left\langle u, g^{ta}\partial_{a}v\right\rangle-\frac{1}{2}\left\langle u, vg^{t\alpha}\partial_{\alpha}\log\abs{g}\right\rangle+\left\langle u, v\partial_{a}g^{ta}\right\rangle+\left\langle u ,v\partial_{t}g^{tt}\right\rangle,\\
       &D(u,v):=-\llangle u, g^{tr}\mathrm{tr}v\rrangle-2\llangle u, (\gamma^{-1})^{'}\mathrm{tr}v\rrangle,\\
       &E(u,v):=-\llangle u, (\gamma^{-1})^{''}\mathrm{tr}v\rrangle.
    \end{aligned}   
\end{equation}
For any $\Theta: [0,T]\to H^{1}(\Omega_{0})$ satisfying
\begin{equation}\label{initial assumption1}
     \begin{aligned}
         &\Theta\in L^{2}([0,T];H^{1}(\Omega_{0})),\quad \Theta^{'}\in L^{2}([0,T];L^{2}(\Omega_{0})), \quad (\mathrm{tr}\Theta)^{'}\in L^{2}([0,T];L^{2}(\partial\Omega_{0})),\\
         & \Phi(\Theta), \Phi(\Theta)^{'}, \Phi(\Theta)^{''}\in L^{2}([0,T];(H^{1}(\Omega_{0}))^{\ast}),
     \end{aligned}
\end{equation}	
let 
\begin{equation}\label{L definition}
    \mathcal{L}(\Theta, v):=B(\Theta,v)+C(\Theta^{'},v)+D((\mathrm{tr}\Theta)^{'})+E(\mathrm{tr}\Theta,v).
\end{equation}
Then for almost every $t\in[0,T]$ the weak equation $(\ref{integration equ of V})$ becomes
\begin{equation}\label{simplify weak equ V}
    (\Phi(\Theta)^{''},v)+\mathcal{L}(\Theta,v)=\llangle \gamma^{-1}f, \mathrm{tr}v \rrangle-\left\langle F,v\right\rangle, \forall v\in H^{1}(\Omega_{0}).
\end{equation}
To complete the formulation of the weak problem, we require that 
\begin{equation}\label{intial statemant for V}
  \begin{aligned}
      &\Theta(0)=\theta_{0},\quad \text{in}\quad L^{2}(\Omega_{0}),\\
      &(\Phi(\Theta)^{'}(0),v)=\left\langle \theta_{1}, v\right\rangle+\llangle \tilde{\theta}_{1},\mathrm{tr}v\rrangle, \quad \forall v\in H^{1}(\Omega_{0}),
  \end{aligned}
\end{equation}
for given initial data 
\begin{equation*}
    \theta_{0}\in H^{1}(\Omega_{0}),\quad \theta_{1}\in L^{2}(\Omega_{0}),\quad \tilde{\theta}_{1}\in L^{2}(\partial\Omega_{0}).
\end{equation*}	

Next we consider the equations obtained by commuting $\partial_{t}$ derivatives. Let us denote the right-hand sides of the system $(\ref{form V equation})$ by $f_{0}=\gamma^{-1}f$ and $F_{0}=-F$. Similarly, we use $f_{k}$ and $F_{k}$ to denote the right-hand sides of the interior and boundary equations after commuting $\partial_{t}^{k}$, and let $\Theta_{k}$ be the $k$-times differentiated unknown.

We need to concentrate on the term $\nabla \omega$ in $F_{0}$ and the term $\omega$ in $f_{0}$. Then if $k=1$, $f_{1}$ is a linear combination of $\Gamma$, $D\partial_{t}^{2}\Vert V\Vert^{2}$, $D\partial_{t}\Vert V\Vert^{2}$, $\partial_{t}^{2}\Vert V\Vert^{2}$, $D\Theta$, $\partial_{t}^{2}\Theta$, $\partial_{t}\Theta$ and $R$ using $(\ref{transport Ch})$, $\ref{transport w}$ and $(\ref{2V boundary equ})$. For $F_{1}$, we need to consider the term $\epsilon_{I}\epsilon_{K}\partial_{t}\nabla_{K}\omega_{KI}$. 
\begin{equation*}
     \epsilon_{I}\epsilon_{K}\partial_{t}\nabla_{K}\omega_{KI}=\epsilon_{I}\epsilon_{K}\frac{1}{\Vert V\Vert}\left[\nabla_{V},\nabla_{K}\right]\omega_{KI}+\epsilon_{I}\epsilon_{K}\frac{1}{\Vert V\Vert}\nabla_{K}\nabla_{V}\omega_{KI}.
\end{equation*}
Then $\epsilon_{I}\epsilon_{K}\partial_{t}\nabla_{K}\omega_{KI}$ is a linear combination of $R$, $D^{2}\Theta$ and $D\Theta$ by $(\ref{transport w})$. So $F_{1}$ is a linear combination of $\Gamma$, $D\Gamma$, $D^{2}\Theta$, $\partial_{t}\Theta$, $D\partial_{t}\Theta$, $D\partial_{t}^{2}\Vert V\Vert^{2}$, $D\partial_{t}\Vert V\Vert^{2}$, $D\Vert V\Vert^{2}$, $\partial_{t}\Vert V\Vert^{2}$, $DR$ and $R$. Let 
\begin{equation}\label{lower terms note}
    \begin{aligned}
       \left\langle\mathcal{C}(\Theta),v\right\rangle:=&-\frac{1}{2}\left\langle\partial_{a}\Theta,v\partial_{t}(g^{a\alpha}\partial_{\alpha}\log\abs{g})\right\rangle-\left\langle\partial_{a}\Theta,v\partial_{t}^{2}g^{t\alpha}\right\rangle\\
       &-\frac{1}{2}\left\langle\Theta^{'},v\partial_{t}(g^{t\alpha}\partial_{\alpha}\log\abs{g})\right\rangle+\left\langle \Theta^{'},v\partial_{t}\partial_{a}g^{ta}\right\rangle,\\
       \left\langle\mathcal{C}^{a}(\Theta),\partial_{a}v\right\rangle:=&\left\langle\partial_{t}g^{ab}\partial_{b}\Theta,\partial_{a}v\right\rangle+2\left\langle\Theta^{'},(\partial_{t}g^{ta})\partial_{a}v\right\rangle,\\
       \left\langle\tilde{\mathcal{C}}(\Theta),v\right\rangle:=&\left\langle\Theta^{'},-v\partial_{t}g^{tt}\right\rangle,\\
        \left\langle\tilde{\tilde{\mathcal{C}}}(\Theta),v\right\rangle:=&\left\langle\Theta^{'},-v\partial_{t}^{2}g^{tt}\right\rangle,\\
        \llangle\mathcal{C}_{\mathcal{B}}(\Theta),\mathrm{tr}v\rrangle:=&-\llangle(\mathrm{tr}\Theta)^{'},(\partial_{t}g^{tr})\mathrm{tr}v\rrangle,\\
        \llangle\tilde{\mathcal{C}}_{\mathcal{B}}(\Theta),\mathrm{tr}v\rrangle:=&\llangle(\mathrm{tr}\Theta)^{'},(\gamma^{-1})^{'}\mathrm{tr}v\rrangle,\\
        \llangle\tilde{\tilde{\mathcal{C}}}_{\mathcal{B}}(\Theta),\mathrm{tr}v\rrangle:=&\llangle(\mathrm{tr}\Theta)^{'},(\gamma^{-1})^{''}\mathrm{tr}v\rrangle.
    \end{aligned}
\end{equation}

Using an induction argument, we denote the higher order source terms for $k\ge1$ by
\begin{equation}\label{higher order source term}
   \begin{aligned}
      F_{k}&=\partial_{t}^{k-1}F_{1}-\sum_{\ell =0}^{k-1}\partial_{t}^{\ell}\mathcal{C}(\Theta_{k-\ell-1})-\sum_{\ell}^{k-1}(k-\ell)\partial_{t}^{\ell}\tilde{\tilde{\mathcal{C}}}(\Theta),\\
      \mathcal{F}_{k}^{a}&=-\sum_{k=0}^{k-1}\partial_{t}^{\ell}\mathcal{C}^{a}(\Theta_{k-\ell-1}),\\
      f_{k}&=\partial_{t}^{k-1}f_{1}-\sum_{\ell=0}^{k-1}\partial_{t}^{\ell}\mathcal{C}_{B}(\Theta_{k-\ell-1})-\sum_{\ell=0}^{k-2}(k-\ell-1)\partial_{t}^{\ell}\tilde{\tilde{\mathcal{C}}}_{\mathcal{B}}(\Theta_{k-\ell-1}).
   \end{aligned}
\end{equation}

Then we see that $\Theta_{k}$ satisfies the weak equation 
\begin{equation}\label{k weak equ for V}
  (\Phi(\Theta_{k})^{''},v)+\mathcal{L}(\Theta_{k},v)+k\left\langle\tilde{\mathcal{C}}(\Theta_{k}),v\right\rangle+k\llangle\tilde{\mathcal{C}}_{\mathcal{B}}(\Theta_{k}),v\rrangle=\left\langle F_{k}, v\right\rangle+\left\langle\mathcal{F}_{k}^{a},\partial_{a}v\right\rangle+\llangle f_{k},\mathrm{tr}v\rrangle, \forall v\in H^{1}(\Omega_{0}).
\end{equation}

Finally, we turn to the main linear theorem for the coupled system satisfied by $\Theta^{I}$. Given that the higher-order terms in the weak equation for $\Theta_{k}$, are of the same form, we can easily derive the following conclusions, as stated in Proposition 3.2 of \cite{Miao2}. We will impose the following regularity assumptions on the coefficients
\begin{equation}\label{regular assumptions}
   \begin{aligned}
       \partial^{a}\partial_{t}^{k}g\in L^{\infty}([0,T];L^{2}(\Omega_{0})),\qquad&k\le K+1, 
       \begin{cases}
           2a\le K+1-k, \quad& k\ge 1\\
           2a\le K,\quad &k=0
       \end{cases},\\
       \partial_{t}^{k}(\mathrm{tr}g)\in L^{2}([0,T];L^{2}(\partial\Omega_{0})),\qquad&k\le K,\\
       \partial_{t}^{k}\gamma^{-1},\partial_{t}^{k}\gamma\in L^{2}([0,T];L^{2}(\partial\Omega_{0})),\qquad&k\le K,\\
       \partial_{t}^{k}\gamma^{-1},\partial_{t}^{k}\gamma\in L^{\infty}([0,T];L^{\infty}(\partial\Omega_{0})),\qquad&k\le K-5,\\
       \partial_{t}^{k-1}f_{1}\in L^{2}([0,T];L^{2}(\partial\Omega_{0})),\qquad&k\le K,\\
       \partial_{t}^{k-1}F_{\Vert V\Vert^{2},1}\in L^{2}([0,T];L^{2}(\Omega_{0})),\qquad&k\le K.
   \end{aligned}
\end{equation}
Here $g$ is related to the frame $\left\{e_{I}\right\}$	by the relation $g^{\alpha\beta}=\sum_{I}\epsilon_{I}e_{I}^{\alpha}e_{I}^{\beta}$.

\begin{proposition}\label{local existence for V}
	Suppose $(\ref{regular assumptions})$ holds and that there are
	\begin{equation*}
	    \theta_{k}\in H^{1}(\Omega_{0}), \theta_{k+1}\in L^{2}(\Omega_{0}), \tilde{\theta}_{k+1}\in L^{2}(\partial\Omega_{0}), k=0,\dots,K
	\end{equation*}
	such that 
	\begin{enumerate}
		\item For $k=0,\dots,K-1$,
		\begin{equation*}
		\left\langle-g^{tt}\theta_{k+2},v\right\rangle+\llangle\tilde{\theta}_{k+2},\gamma^{-1}\mathrm{tr}v\rrangle+\mathcal{L}(\theta_{k},v)+k\left\langle\tilde{\mathcal{C}}(\theta_{k}),v\right\rangle+k\llangle \tilde{\mathcal{C}}_{\mathcal{B}}(\theta_{k}),v\rrangle=\llangle f_{k}(0),\mathrm{tr}v\rrangle+\left\langle F_{k}(0),v\right\rangle+\left\langle\mathcal{F}_{k}^{a}(0),\partial_{a}v\right\rangle.
		\end{equation*}
		\item $\tilde{\theta_{k}}=\mathrm{tr}\theta_{k}$ for $k=1,\dots,K$.
	\end{enumerate}
Then there is a unique $\Theta_{k}$ satisfying $(\ref{initial assumption1})$ and $(\ref{intial statemant for V})$, such that for all $v\in H^{1}(\Omega_{0})$ equation $(\ref{k weak equ for V})$ holds for almost every $t\in[0,T]$. The solution $\Theta_{k}$ satisfies 
\begin{equation}\label{V energy estimate}
  \begin{aligned}
      &\sup_{t\in[0,T]}\left(\Vert \Theta_{k}^{'}\Vert_{L^{2}(\Omega_{0})}+\Vert \Theta_{k}\Vert_{H^{1}(\Omega_{0})}+\Vert \mathrm{tr}\Theta_{k}^{'}\Vert_{L^{2}(\partial\Omega_{0})}\right)\\
      &\le C_{1}e^{C_{2}T}\left(\Vert \theta_{k}\Vert_{H^{1}(\Omega_{0})}+\Vert \theta_{k+1}\Vert_{L^{2}(\Omega_{0})}+\Vert \tilde{\theta}_{k+1}\Vert_{L^{2}(\partial\Omega_{0})}+\Vert f_{k}\Vert_{L^{2}([0,T];L^{2}(\partial\Omega_{0}))}\right.\\
      &\left. +\Vert F_{k}\Vert_{L^{2}([0,T];L^{2}(\Omega_{0}))}+\Vert \mathcal{F}_{k}^{a}\Vert_{L^{\infty}([0,T];L^{2}(\Omega_{0}))}\right).
  \end{aligned}
\end{equation}
Here $C_{1}$, $C_{2}$, and $C_{3}$ depend on the norms of $g$, $\mathrm{tr}g$, $\gamma$, $\gamma^{-1}$ appearing in $(\ref{regular assumptions})$. Moreover, we have $\Theta_{k-1}^{'}=\Theta_{k}$ for $k=1,\dots, K$, and there are functions $P_{k}$ depending polynomially on their arguments such that for $k\le K$ and $2a+k\le K+2$ and for $\tau\le T$
\begin{equation}\label{space partial estimate fotr V}
    \begin{aligned}
        \Vert \partial^{a}\Theta_{k}(\tau)\Vert_{L^{\infty}([0,\tau]; L^{2}(\Omega_{0}))}&\le P_{k}\left(\sup_{t\le \tau}\sum_{\ell\le 2a+k-2}\left(\Vert \nabla\Theta_{\ell}(t))\Vert_{L^{2}(\Omega_{0})}+\Vert \Theta_{\ell+1(t)}\Vert_{L^{2}(\Omega_{0})}+\Vert \Theta_{\ell+1}(t)\Vert_{L^{2}(\partial\Omega_{0})}\right)\right.,\\
        &\left.\Vert g\Vert_{L^{\infty}([0,\tau];H^{\max\left\{a-2,5\right\}}(\Omega_{0}))},\sum_{\ell\le k}\Vert \partial_{t}^{\ell}f\Vert_{L^{\infty}([0,T];H^{a-\frac{3}{2}}(\Omega_{0}))}\right).
    \end{aligned}
\end{equation}
\end{proposition}

We also have the following lemma, which is an analogue of Lemma $\ref{Theta boundary control}$ and proved in Lemma 3.12 of \cite{Miao2}. 

\begin{lemma}\label{boundaryV estimates}
	Under the assumptions of Proposition $\ref{local existence for V}$, and with the same notation, for any $k\le K-1$
	\begin{equation*}
	    \begin{aligned}
	        \Vert \nabla\Theta_{k}\Vert_{L^{2}([0,T];L^{2}(\partial\Omega_{0}))}^{2}&\lesssim \sum_{j\le k}\left(\Vert F_{j}\Vert_{L^{2}([0,T];L^{2}(\Omega_{0}))}+\Vert\mathcal{F}_{j}\Vert_{L^{2}([0,T];L^{2}(\Omega_{0}))}^{2}+\Vert f_{j}\Vert_{L^{2}([0,T];L^{2}(\partial\Omega_{0}))}^{2}\right)+\Vert \theta\Vert_{H^{1}(\Omega_{0})}^{2}\\
	        &+\Vert\theta_{k+1}\Vert_{L^{2}(\Omega_{0})}^{2}+\Vert \tilde{\theta}_{k+1}\Vert_{L^{2}(\partial\Omega_{0})}^{2}+\Vert (\mathrm{tr}\Theta_{k})^{'}\Vert_{L^{2}([0,T];L^{2}(\partial\Omega_{0}))}^{2}+\Vert (\mathrm{tr}\Theta_{k})^{''}\Vert_{L^{2}([0,T];L^{2}(\partial\Omega_{0}))}^{2},
	    \end{aligned}
	\end{equation*}
	where the implicit constant depends only on $g$, $\gamma$, and their first three derivatives.
\end{lemma}

\subsection{The linear existence theory of the equation for $\partial_{t}\Vert V\Vert^{2}$. } Recall from $(\ref{2DV^2 frame equ})$ that 
\begin{equation*}
   \begin{cases}
        \square \partial_{t}\Vert V\Vert^{2}=\left(\frac{1}{\eta^{2}}-1\right)\partial_{t}^{3}\Vert V\Vert^{2}+F\quad &\text{in}\quad \Omega,\\
        \partial_{t}\Vert V\Vert^{2}=0 \quad &\text{on}\quad \partial\Omega,
   \end{cases}   
\end{equation*} 
where $F$ consists of $D^{2}\Vert V\Vert^{2}$, $\partial_{t}^{2}\Vert V\Vert^{2}$, $\partial_{t}\Vert V\Vert^{2}$, $D\Vert V\Vert^{2}$, $D\partial_{t}\Vert V \Vert^{2}$, $\Vert V\Vert^{2}$, $\partial_{t}\Theta$, $D\Theta$, $\epsilon_{K}\Theta^{J}\nabla_{V}\nabla_{K}\omega_{KJ}$, $D\omega$, $\Gamma$ and $R$. Let $\Lambda$ be the linearized unknown for the equation $D_{V}\Vert V\Vert^{2}$. Employing calculations analogous to those utilized for $\Theta$, we can derive the weak equation for $\Lambda$. We define the following bilinear forms
\begin{equation*}
     \begin{aligned}
        &\tilde{B}: H_{0}^{1}(\partial\Omega_{0})\times H_{0}^{1}(\Omega_{0})\to \mathbb{R}, \qquad \tilde{C}: L^{2}(\Omega_{0})\times H_{0}^{1}(\Omega_{0})\to \mathbb{R},\\
        &\tilde{B}(u,v):=\left\langle g^{ab}\partial_{a}u,\partial_{b}v\right\rangle-\frac{1}{2}\left\langle \partial_{a}u, vg^{a\alpha}\partial_{\alpha}\log\abs{g}\right\rangle-\left\langle\partial_{a}u, v\partial_{t}g^{ta}\right\rangle,\\
        &\tilde{C}(u,v):=2\left\langle u, g^{ta}\partial_{a}v\right\rangle-\frac{1}{2}\left\langle u, vg^{t\alpha}\partial_{\alpha}\log\abs{g}\right\rangle+\left\langle u, v\partial_{a}g^{ta}\right\rangle-\left\langle u ,v\partial_{t}g^{tt}\right\rangle.
     \end{aligned}
\end{equation*}
The standard embedding of $H_{0}^{1}(\Omega_{0})$ in $H^{-1}(\Omega_{0}):=(H_{0}^{1}(\Omega))^{\ast}$ is given by $\iota$
\begin{equation*}
       (\iota(u),v):=\left\langle u,v\right\rangle, \quad u\in H_{0}^{1}(\Omega_{0}),\quad v\in H_{0}^{1}(\Omega_{0}).
\end{equation*}
We will now simply write $u$ for $\iota(u)$. For any $\Lambda$ with 
\begin{equation}\label{initial condition for DtV2}
    \Lambda\in L^{2}([0,T]; H_{0}^{1}(\Omega_{0})), \quad \Lambda^{'}\in L^{2}([0,T];L^{2}(\Omega_{0})), \quad \Lambda^{''}\in L^{2}([0,T];H^{-1}(\Omega_{0})),
\end{equation}
and $v\in H_{0}^{1}(\Omega_{0})$ let 
\begin{equation*}
     \mathcal{L}_{\Vert V\Vert}(\Lambda,v):=\tilde{B}(\Lambda,v)+\tilde{C}(\Lambda^{'},v).
\end{equation*}
Then the linearized weak equation for $\Lambda$ is 
\begin{equation}\label{weak equ for Lambda}
     (-g^{tt}\Lambda^{''},v)+\mathcal{L}_{\Vert V\Vert}(\Lambda,v)=\left\langle-(\frac{1}{\eta^{2}}-1)\Lambda_{2},v\right\rangle+\left\langle-F,v\right\rangle, \quad\forall v\in H_{0}^{1}(\Omega_{0}),
\end{equation}
where $F\in L^{2}([0,T];L^{2}(\Omega_{0}))$ is a given function. The initial conditions are 
\begin{equation}\label{initial data for Lambda}
   \Lambda(0)=\lambda_{0}, \quad (\lambda^{'}(0),v)=\left\langle \lambda_{1},v\right\rangle, \quad \forall v\in H_{0}^{1}(\Omega_{0}).
\end{equation}
Here $\lambda_{0}\in H_{0}^{1}(\Omega_{0})$ and $\lambda_{1}\in L^{2}(\Omega_{0})$ are given initial data. We also need to consider the higher order equations for $\Lambda_{k}:=\partial_{t}^{k}\Lambda$. Let  
\begin{equation}\label{higher order source term for Lambda}
   \begin{aligned}
      &F_{\Vert V\Vert,0}:=-F_{\Vert V\Vert}, \quad \mathcal{F}_{\Vert V\Vert,0}^{a}:=0,\\
      &F_{\Vert V\Vert,k}=\partial_{t}^{k-1}F_{\Vert V\Vert^{2}, 1}-\sum_{\ell=0}^{k-1}\partial_{\ell}\mathcal{C}(\Lambda_{k-\ell-1})-\sum_{\ell=0}^{k-1}(k-\ell)\partial_{\ell}\tilde{\tilde{\mathcal{C}}}(\Lambda_{k-\ell-1}),\\
      &\mathcal{F}_{\Vert V\Vert,k}^{a}=-\sum_{\ell=0}^{k-1}\partial_{t}^{\ell}\mathcal{C}^{a}(\Lambda_{k-\ell-1}).
   \end{aligned}
\end{equation}
Then the higher order equations for $\Lambda_{k}$ are 
\begin{equation}\label{higher order weak equ for Lambda}
    (-g^{tt}\Lambda_{k}^{''},v)+\mathcal{L}_{\Vert V\Vert}(\Lambda_{k},v)+k\left\langle\tilde{\mathcal{C}}(\Lambda_{k},v)\right\rangle=\left\langle -(\frac{1}{\eta^{2}}-1)\Lambda_{k+2},v\right\rangle+\left\langle F_{\Vert V\Vert,k},v\right\rangle+\left\langle \mathcal{F}_{\Vert V\Vert,k}^{a},\partial_{a}v\right\rangle, \quad \forall v\in H_{0}^{1}(\Omega_{0}).
\end{equation}
Note that $F_{\Vert V\Vert,1}$ is a linear combination of $\Gamma$, $\partial_{t}^{3}\Vert V\Vert^{2}$, $\partial_{t}\Vert V\Vert^{2}$, $\partial_{t}\Vert V\Vert^{2}$, $D\partial_{t}^{2}\Vert V\Vert^{2}$, $D\partial_{t}\Vert V\Vert^{2}$, $D^{2}\partial_{t}\Vert V\Vert^{2}$, $\partial_{t}^{2}\Theta$, $D\partial_{t}\Theta$, $D^{2}\Theta$, $\partial_{t}R$ and $DR$ using $(\ref{transport w})$, $(\ref{2V interior frame equ})$, $(\ref{2DV^2 frame equ})$ and $(\ref{main skill for higher order DV2})$. We can also obtain the following result, which is the analogue of Proposition \ref{local existence for V} for $\Lambda$, as proved in Lemma 3.10 and Proposition 3.11 in \cite{Miao2}.
\begin{proposition}\label{local existence for Lambda}
	Suppose $(\ref{regular assumptions})$ holds and that there exist 
	\begin{equation*}
	   \lambda_{k}\in H_{0}^{1}(\Omega_{0}), \quad \lambda_{k+1}\in L^{2}(\Omega_{0}), \quad k=0,\dots,K
	\end{equation*}
	such that 
	\begin{equation*}
	    (-g^{tt}\lambda_{k}^{''},v)+\mathcal{L}_{\Vert V\Vert}(\lambda_{k},v)+k\left\langle\tilde{C}(\lambda_{k},v)\right\rangle=\left\langle -(\frac{1}{\eta^{2}}-1)\lambda_{k+2},v\right\rangle+\left\langle F_{\Vert V\Vert,k}(0),v\right\rangle+\left\langle \mathcal{F}_{\Vert V\Vert,k}^{a}(0),\partial_{a}v\right\rangle.
	\end{equation*}
	Then there is a unique $\Lambda_{k}$ satisfying $(\ref{initial condition for DtV2})$ and $(\ref{initial data for Lambda})$, such that for all $v\in H_{0}^{1}(\Omega_{0})$ equation $(\ref{higher order weak equ for Lambda})$ holds for almost every $t\in[0,T]$. The solution satisfies 
	\begin{equation}\label{Lambda energy estimate}
	   \begin{aligned}
	      &\sup_{t\in[0,T]}\left(\Vert\Lambda_{k}^{'}\Vert_{L^{2}(\Omega_{0})}+\Vert \Lambda_{k}\Vert_{H^{1}(\Omega_{0})}\right)+\Vert \nabla\Lambda_{k}\Vert_{L^{2}([0,T];L^{2}(\partial\Omega_{0}))}^{2}\\
	      &\le C_{1}e^{C_{2}T}\left(\Vert \lambda_{k}\Vert_{H^{1}(\Omega_{0})}+\Vert \lambda_{k+1}\Vert_{L^{2}(\Omega_{0})}+\Vert F_{\Vert V\Vert,k}\Vert_{L^{2}([0,T];L^{2}(\Omega_{0}))}+\Vert \mathcal{F}_{\Vert V\Vert,k}\Vert_{L^{2}([0,T];L^{2}(\Omega_{0}))}\right).
	   \end{aligned}	    
	\end{equation}
	Here the constants $C_{1}$, $C_{2}$ and $C_{3}$ depend on norms of $g$ appearing in $(\ref{regular assumptions})$. Moreover, $\Lambda_{k-1}^{'}=\Lambda_{k}$ for $k=1, \dots, K$, and there are functions $P_{k}$ depending polynomically on their arguments such that for $k\le K$ and $2a+k\le K+2$, and for $\tau\le T$, 
	\begin{equation}\label{space partial estimate for Lambda}
	   \begin{aligned}
	       \Vert \partial^{a}\Lambda_{k}(\tau)\Vert_{L^{\infty}([0,\tau];L^{2}(\Omega_{0}))}&\le P_{k}\left(\sup_{t\le \tau}\sum_{\ell\le 2a+k-2}\left(\Vert \nabla \Lambda_{\ell}(t)\Vert_{L^{2}(\Omega_{0})}+\Vert \Lambda_{\ell+1}(t)\Vert_{L^{2}(\Omega_{0})}\right),\right.\\
	       &\left. \Vert g\Vert_{L^{\infty}([0,\tau];H^{\max\left\{a-2,5\right\}}(\Omega_{0}))},\sum_{\ell\le k}\Vert \partial_{t}^{\ell}F_{\Vert V\Vert}\Vert_{L^{2}([0,T];H^{a}(\Omega_{0}))}\right).
	    \end{aligned}
	\end{equation}
\end{proposition}

Based on the a priori estimates in Proposition $\ref{main prop}$ and the existence results in Propositions $\ref{local existence for R}$, $\ref{local existence for V}$ and $\ref{local existence for Lambda}$, we proceed to prove Theorem \ref{main thm} as follows.

\begin{proof}[Proof of Themrem $\ref{main thm}$] The initial step is to set up an iteration for the system defined by $(\ref{transport e})$, $(\ref{transport Ch})$, $(\ref{connection of g})$, $(\ref{2V boundary equ})$, $(\ref{2V interior frame equ})$, $(\ref{2DV^2 frame equ})$, $(\ref{new hyperbolic system})$.Furthermore, we demonstrate the existence of a solution to the system $(\ref{geometric quantities equ})$, $(\ref{fluid quantities for V equ})$, $(\ref{fluid quantities for DV^2 equ})$. The uniqueness of the solution can be shown by similar estimates. Once the iteration has been completed, we shall proceed to demonstrate how the solution can be shown to satisfy the original equations $(\ref{frame main equ})$ and $(\ref{Ricci equ})$. 
	
The derived system for the iteration is described as follows. The geometric quantities will be defined using the equatios
\begin{equation}\label{geometric quantities equ}
   \begin{cases}
       e_{0}^{\mu}=\frac{1}{\hat{\Theta}^{0}}(\delta_{0}^{\mu}-\hat{\Theta}^{\tilde{I}}e_{\tilde{I}}^{\mu}),\\
       \partial_{t} e_{\tilde{I}}^{\mu}=-e_{J}^{\mu}D_{\tilde{I}}\hat{\Theta}^{J}-\hat{\Theta}^{J}\Gamma_{\tilde{I}J}^{K}e_{K}^{\mu},\\
       \Gamma_{0J}^{K}=-\frac{\hat{\Theta}^{\tilde{I}}}{\hat{\Theta}^{0}}\Gamma_{\tilde{I}J}^{K},\\
       \partial_{t}\Gamma_{\tilde{I}J}^{K}=\hat{\Theta}^{I}(R_{\ JI\tilde{I}}^{K}-\Gamma_{LJ}^{K}\Gamma_{\tilde{I}I}^{L})-\Gamma_{IJ}^{K}D_{\tilde{I}}\hat{\Theta}^{I}, \\
       \sum_{\nu=0}^{3}\check{\mathcal{B}}^{\mu}\partial_{\mu}\check{W}_{AB}=\check{\mathcal{K}}_{AB},
   \end{cases}
\end{equation} 
where $\check{\mathcal{K}}$ is defined as in $(\ref{new hyperbolic system})$. Here the first and the third equations of $(\ref{geometric quantities equ})$ are algebraic definitions, while the second and fourth equations of $(\ref{geometric quantities equ})$ are transport equations. The fluid quantities are defined using the equations 
\begin{equation}\label{fluid quantities for DV^2 equ}
\begin{cases}
\square\partial_{t}\Vert V\Vert^{2}=
H\quad &\text{in} \quad \Omega\\
\partial_{t}\Vert V\Vert^{2} =0 \quad &\text{on} \quad \partial\Omega
\end{cases}
\end{equation}
for $\partial_{t}\Vert V\Vert^{2}$, and
\begin{equation}\label{fluid quantities for V equ}
   \begin{cases} 
      \square\Theta^{I}=F\quad &\text{in}\quad \Omega\\
       \left(\partial_{t}^{2}+\gamma D_{n}\right)\Theta^{I}=f\quad &\text{on} \quad \partial\Omega
   \end{cases},
\end{equation}
for the $\Theta^{I}$, where $f$, $F$ and $H$ are the terms on the right-hand side of equations (\ref{2V boundary equ}), $(\ref{2V interior frame equ})$ and (\ref{2DV^2 frame equ}), respectively.

Let $\Theta^{(m)}$, $\Lambda^{(m)}$, $\Sigma^{(m)}$ denote the iterates of $\Theta$, $\partial_{t}\Vert V\Vert^{2}$ and $\Vert V\Vert^{2}$ respectively, where $\Sigma^{(m)}$ is determined by the relation $\partial_{t}\Sigma^{(m)}=\Lambda^{(m)}$. The iterates of the geometric quantities are denoted by $e_{I}^{(m)}$, $g^{(m)}$, $\Gamma_{IJ}^{K,(m)}$, $\check{W}^{(m)}$. The zeroth iterates are defined by simply extending the initial data. 

The $m$th iterate is then defined inductively using equations $(\ref{geometric quantities equ})$, $(\ref{fluid quantities for V equ})$ and $(\ref{fluid quantities for DV^2 equ})$. In this framework, the unknowns on the left-hand side of each equation are taken at the $m$th iterate, while the terms on the right-hand side, including the coefficients of $\square$, $\gamma$, and $\check{\mathcal{B}}$, use the $(m-1)$th iterate. Here $\check{\mathcal{B}}^{(m)}$, $g^{(m)}$, $R^{(m)}$, $\gamma^{(m)}$ are defined algebraically in terms of $\Theta^{(m)}$, $\Lambda^{(m)}$, $\Sigma^{(m)}$, $\Gamma^{(m)}$, $\check{W}^{(m)}$, as described previously. Specifically, the $m$th iterate of $\omega$ satisfies that 
\begin{equation*}
    \nabla_{V}\omega_{IJ}^{(m)}+\nabla_{I}\Theta^{K,(m)}\omega_{KJ}^{(m)}+\nabla_{J}\Theta^{K,(m)}\omega_{IK}^{(m)}=0.
\end{equation*}
Let 
\begin{equation*}
  \begin{aligned}
      \mathcal{E}_{k}^{m}(T)&:=\sup_{0\le t\le T}\sum_{\ell\le k}(\Vert \partial\partial_{t}^{\ell}\Theta^{(m)}(t)\Vert_{L^{2}(\Omega_{0})}^{2}+\Vert \partial\partial_{t}^{\ell}\Lambda^{(m)}(t)\Vert_{L^{2}(\Omega_{0})}^{2}+\Vert \partial_{t}^{\ell}R^{(m)}\Vert_{L^{2}(\Omega_{0})}^{2}+\Vert \partial_{t}^{\ell}\Theta^{(m)}(t)\Vert_{L^{2}(\partial\Omega_{0})}^{2})\\
      &+\int_{0}^{T}\Vert \partial\partial_{t}^{\ell}\Lambda^{(m)}\Vert_{L^{2}(\Omega_{0})}^{2}\mathrm{d}t.
  \end{aligned} 
\end{equation*}
\textbf{Step 1} We will claim that the $\mathcal{E}_{k}^{m}(T)$ is uniformly bounded. In other words, there are constants $A_{k}, k=0,\dots, K$ such that for all $m$
\begin{equation}\label{energy boundedness}
   \mathcal{E}_{k}^{m}(T)\le A_{k}
\end{equation} 
with $T$ sufficiently small. In light of Propositions $\ref{local existence for R}$, $\ref{local existence for V}$ and $\ref{local existence for Lambda}$, the proof of this claim is similar to Proposition $\ref{main prop}$. The key step is to verify the fulfillment of the hypotheses of Proposition \ref{local existence for R}. To establish the regularity of $\check{\mathrm{div}}\check{\mathcal{K}}_{\check{E}^{(m)}}^{(m)}$ and $\check{\mathrm{div}}\check{\mathcal{K}}_{\check{H}^{(m)}}^{(m)}$, careful attention must be paid to the placement of second-order derivatives on $\check{X}_{A}^{(m)}$. Fortunately, each second-order derivative of $\check{X}_{A}^{(m)}$ can be adeptly replaced by a commutator due to $R_{IJKL}=-R_{JIKL}$. For instance, we can consider the expression
\begin{equation*}
   \epsilon_{I}\epsilon_{J}\epsilon_{K}\check{R}_{JKLI}\left(\check{X}_{A}^{I}\check{D}_{K}\check{D}_{J}\check{X}_{B}^{L}\right)=\frac{1}{2}\epsilon_{I}\epsilon_{J}\epsilon_{K}\check{R}_{JKLI}\left(\check{X}_{A}^{I}\left[\check{D}_{K},\check{D}_{J}\right]\check{X}_{B}^{L}\right).
\end{equation*}  
In this case, the two derivatives acting on $\check{X}_{B}^{L}$ can be repalced by the commutator $\left[\check{D}_{K},\check{D}_{J}\right]$. Next, we apply Propositions $\ref{local existence for R}$, $\ref{local existence for V}$ and $\ref{local existence for Lambda}$ to bound lower-order terms in $L^{\infty}$, and apply Lemma \ref{boundaryV estimates} to bound $\Vert \partial\partial_{t}^{k}\Theta\Vert_{L^{2}([0,T]\times\partial\Omega_{0})}$. Subsequently, we can deduce (\ref{energy boundedness}) inductively.

\textbf{Step 2} We establish the convergence of the sequences $\Theta^{(m)}$, $\Lambda^{(m)}$, $e_{I}^{(m)}$, $\Gamma^{(m)}$, and $\check{W}^{(m)}$. To do this, by (\ref{energy boundedness}), it is necessary to consider the equations for the difference of the iterates with zero initial data. The subtle difference lies in the schematic forms of the equations satisfied by $\Lambda^{(m+1)}$:
\begin{equation*}
\begin{cases}
\square_{g^{(m)}}\Lambda^{(m+1)}=\left(\frac{1}{\eta^{2}}-1\right)\partial_{t}^{2}\Lambda^{(m+1)}+F(\Theta^{(m)},\Sigma^{(m)}), \quad&\text{in}\quad [0,T]\times \Omega\\
\Lambda^{(m+1)}=0, \quad &\text{on} \quad [0,T]\times \partial \Omega
\end{cases}.
\end{equation*}
Since $\frac{1}{\eta^{2}}-1$ is non-negative, we can also apply the energy estimates from Proposition $\ref{local existence for Lambda}$. We then employ the standard method to establish convergence. For further details, refer to the proof of Theorem 1.1 in \cite{Miao2}.

\textbf{Step 3} We demonstrate that our derived solution satisfies the original equations $(\ref{frame main equ})$ and $(\ref{Ricci equ})$. To achieve this goal, we will formulate homogeneous equations for the vanishing unknowns that are consistent with the structures previously encountered in our a priori estimates. In our calculation, the vanishing fluid quantities are denoted by 
\begin{equation}\label{the vanishing fluid quantities}
   \begin{aligned}
      &H:=D_{V}G+G\nabla_{I}\Theta^{I},\\
      &S:=\Sigma+\Theta^{I}\Theta_{I},\\
      &X_{I}:=\Theta^{J}\Theta_{J}\Theta_{I}+\frac{1}{2}\nabla_{I}\Sigma,\\
      &Y_{I}:=\Theta^{J}\nabla_{J}\left(\Theta^{K}\nabla_{K}\Theta_{I}\right)-\frac{1}{2}\nabla^{J}\Sigma\nabla_{J}\Theta_{I}+\frac{1}{2}\nabla_{I}(\Theta^{J}\nabla_{J}\Sigma)-\frac{1}{2}\epsilon_{J}\omega_{IJ}\nabla_{J}\Sigma,\\
      &D_{IJ}:=\nabla_{I}\Theta_{J}-\nabla_{J}\Theta_{I}-\omega_{IJ},\\
      &A:=\square\Sigma+2(\nabla^{I}\Theta^{J})(\nabla_{I}\Theta_{J})+2R_{KI}\Theta^{K}\Theta^{I}+2\epsilon_{K}\nabla_{K}\omega_{KI}\Theta^{I}-\frac{2}{G}(\Theta^{J}\nabla_{J}(\Theta^{K}\nabla_{K}G)+\nabla_{J}\Theta^{J}\Theta^{K}\nabla_{K}G).
   \end{aligned}
\end{equation}
These quantities are expected to vanish in $\Omega$, which are used in arriving at our derived system $(\ref{frame main equ})$. Similarly, the vanishing geometric quantities are
\begin{equation}\label{the vanishing geomeric quantities}
   \begin{aligned}
      &Q_{IJ}^{K}e_{K}:=-\Gamma_{IJ}^{K}e_{K}+\Gamma_{JI}^{K}e_{K}+\left[e_{I},e_{J}\right],\\
      &B_{IJKL}:=R_{[IJk]L},\\
      &\tilde{Y}_{ABCD}:=R_{ABCD}-R_{CDAB},\\
      &\Delta_{I}^{KL}:=\nabla^{J}R_{JIKL}-\chi_{\Omega}\left[\nabla_{K}\left(G\Theta_{L}\Theta_{I}-pm_{LI}+\frac{1}{2}G\Vert V\Vert^{2}m_{LI}\right)-\nabla_{L}\left(G\Theta_{K}\Theta_{I}-pm_{KI}+\frac{1}{2}G\Vert V\Vert^{2}m_{KI}\right)\right],\\
      &Z_{IJK}^{An}:=\nabla_{[I}R_{JK]An},\\
      &Z_{IJK}^{AB}:=\nabla_{[I}R_{JK]AB},\\
      &\tilde{R}_{AB}:=R_{AB}-\chi_{\Omega}(G\Theta_{A}\Theta_{B}-Pg_{AB}+\frac{1}{2}G\Vert V\Vert^{2}g_{AB}),\quad \mathrm{for}\quad A>B,\\
      &\tilde{R}_{An}:=R_{An},\\
      &\tilde{R}_{nn}:=R_{nn}-\frac{1}{2}\chi_{\Omega}g_{nn},\\
      &\tilde{S}_{\ MIJ}^{K}:=R_{\ MIJ}^{K}-(D_{I}\Gamma_{JM}^{K}-D_{J}\Gamma_{IM}^{K})-(\Gamma_{IL}^{K}\Gamma_{JM}^{L}-\Gamma_{JL}^{K}\Gamma_{IM}^{L})-(\Gamma_{JI}^{L}-\Gamma_{IJ}^{L})\Gamma_{LM}^{K}.
   \end{aligned}
\end{equation}
 In addition to the Einstein equations $(\ref{einstein equ})$, these vanishing quantities capture the symmetries of the curvature tensor, the differential Bianchi identities for the entire curvature tensor, the torsion-free property of the connection, and the Ricci identity. As with the fluid quantities, The vanishing of these quantities was employed in the derivation of our system  $(\ref{main structure equ})$.

In the following, indices are omitted from the defining letters when specific components are not essential. For instance, the symbols $S$, $Z$, $\tilde{S}$, $Z^{An}$, $Z^{AB}$, $\Delta^{An}$ and $\Delta^{AB}$ are used to represent the corresponding components. By definition, the fluid and geometric vanishing quantities are initially zero. To demonstrate that they remain zero throughout the evolution, it is necessary to derive a system of differential equations governing their behavior. After a tedious calculation, the equations for the fluid quantities yield 
\begin{equation}\label{equy}
   \begin{cases}
       \square Y_{I}=F_{Y,I}\quad &\mathrm{in}\quad \Omega\\
       Y_{I}=0\quad  &\mathrm{on}\quad \partial\Omega
   \end{cases}.
\end{equation}
Here the source term $F_{Y}$ depends algebraically on 
\begin{equation*}
   \tilde{S}, \quad\nabla\partial_{t}\tilde{S}, \quad\nabla\partial_{t}\tilde{Y}, \quad\nabla\tilde{S}, \quad\Delta, \quad\nabla Q,\quad D, \quad\nabla D,\quad X, \quad\nabla X, \quad\nabla^{2}X, \quad\nabla A.
\end{equation*}
The remaining vanishing fluid quantities can be related to $Y$ using the transport equations
\begin{equation}\label{equtrans}
   \partial_{t}H=F_{H},\quad \partial_{t}D_{IJ}=F_{D,IJ}, \quad \partial_{t}X_{I}=F_{X,I}, \quad\partial_{t}S=F_{S},\quad \partial_{t}A=F_{A},  
\end{equation}
where $F_{H}$ depends algebraically on $\tilde{R},\nabla D, Q$,
$F_{D}$ depends algebraically on
\begin{equation*}
      \tilde{Y}, \quad \nabla X,\quad Q,
\end{equation*}
$F_{X}$ depends algebraically on
\begin{equation*}
       Y, \quad D,\quad Q,
\end{equation*}
$F_{S}$ depends algebraically on $X$, and $F_{A}$ depends algebraically on
\begin{equation*}
     B, \quad \tilde{R},\quad Q, \quad D, \quad \nabla X.
\end{equation*}
For the geometric vanishing quantities, the main system of differential equations involves the unknowns $\Delta^{Jn}$ and $Z^{Jn}$. It can be shown that for each choice of $J$, the corresponding unknowns $\Delta_{B}^{Jn}$, $\Delta_{n}^{Jn}$, $Z_{012}^{Jn}$, and $Z_{BCn}^{Jn}$, with $B < C$, satisfy a coupled first-order hyperbolic system. This system serves as the foundation for establishing both $L^{2}$ and elliptic estimates. Moreover, since these quantities vanish on the boundary, the system can be further differentiated to yield additional results. After a cumbersome calculation, similar to the approach in Section \ref{subsec1.3.1}, we derive the following wave equations, which hold on $\cup_{0 \leq t \leq T} \Sigma_{t}$, 
\begin{equation}\label{equZn}
     \square Z_{IJK}^{An}=F_{Z,IJK}^{An}, \quad \square \Delta_{I}^{An}=F_{\Delta,I}^{An},
\end{equation}
where $F_{Z}^{An}$ depends algebraically on 
\begin{equation*}
    Z, \quad \nabla Z, \quad\chi_{\Omega}D, \quad\chi_{\Omega}\nabla D, \quad B,\quad \nabla B,\quad \Delta, \quad \nabla\Delta,
\end{equation*}
and $F_{\Delta}^{An}$ depends algebraically on
\begin{equation*}
    \Delta, \quad \nabla\Delta, \quad Z, \quad \nabla Z, \quad \chi_{\Omega}D, \quad \chi_{\Omega}\nabla D, \quad \tilde{Y}.
\end{equation*}
In the subsequent analysis, the values $0$, $1$, and $2$ refer to tangential components. Finally, the remaining vanishing geometric quantities can be expressed in terms of $Z$ and $\Delta$ through the following transport equations and first-order hyperbolic systems.:
\begin{equation}\label{equqs}
    \partial_{t}Q_{IJ}^{K}=F_{Q,IJ}^{K}, \quad \partial_{t}\tilde{S}_{IJKL}=F_{\tilde{S},IJKL},
\end{equation}
and 
\begin{equation}\label{equbzdr}
    \begin{aligned}
       &D_{0}B_{nACD}=F_{B,nACD}, \quad D_{0}Z_{IJK}^{AB}=F_{Z,IJK}^{AB}, \quad D_{0}\Delta_{I}^{AB}=F_{\Delta,I}^{AB},\\
       &D_{0}\tilde{R}_{nA}=F_{\tilde{R},nA},\quad D_{0}\tilde{R}_{10}=F_{\tilde{R},10},\\
       &D_{0}\tilde{R}_{20}-D_{1}\tilde{R}_{21}=F_{\tilde{R},20}, \quad D_{0}\tilde{R}_{21}-D_{1}\tilde{R}_{20}=F_{\tilde{R},21}.
    \end{aligned}
\end{equation}
Here $F_{Q}$ depends algebraically on
\begin{equation*}
    B,\quad \tilde{Y},\quad \tilde{S}, \quad Q,
\end{equation*}
$F_{\tilde{S}}$  depends algebraically on
\begin{equation*}
   Z,\quad \tilde{S}, \quad \tilde{Y},\quad \nabla\tilde{Y},\quad B,\quad Q,
\end{equation*}
$F_{B}$ depends algebraically on $\Delta,  Z$, $F_{\tilde{R}}$ depends algebraically on $Z, \Delta, \tilde{R}$ and $F_{Z}^{AB}$ and $F_{\Delta}^{AB}$ depends algebraically on $\tilde{Y}, \Delta, B, Z$. Equations \eqref{equy}, \eqref{equtrans}, \eqref{equZn}, \eqref{equqs}, and \eqref{equbzdr} form the desired system for the fluid and geometric vanishing quantities. Based on the a priori estimates in Section \ref{sec2}, the quantities in $(\ref{the vanishing fluid quantities})$ and $(\ref{the vanishing geomeric quantities})$ will remain zero for all $t \leq T$, provided they vanish initially.  We then complete the proof of the theorem.
\end{proof}
\bibliographystyle{plain}
\bibliography{mybib}
	
	\bigskip
	
	\centerline{\scshape Zeming Hao}
	\smallskip
	{\footnotesize
		\centerline{School of Mathematics and Statistics, Wuhan University}
		\centerline{Wuhan, Hubei 430072, China}
		\centerline{\email{2021202010062@whu.edu.cn}}
	}
	
	\medskip
	
	\centerline{\scshape Wei Huo}
	\smallskip
	{\footnotesize
		\centerline{School of Mathematics and Statistics, Wuhan University}
		\centerline{Wuhan, Hubei 430072, China}
		\centerline{\email{huowei@whu.edu.cn}}
	}
	
	\medskip
	
	\centerline{\scshape Shuang Miao}
	\smallskip
	{\footnotesize
		\centerline{School of Mathematics and Statistics, Wuhan University}
		\centerline{Wuhan, Hubei 430072, China}
		\centerline{\email{shuang.m@whu.edu.cn}}
	} 
	
\end{document}